\documentclass[a4paper, 11pt]{amsart}

\usepackage{amsmath, amsthm, amsfonts, amssymb}
\usepackage{enumerate}
\usepackage[bookmarks=false]{hyperref} 
\usepackage{tikz-cd}
\usepackage{url}

\usepackage{xcolor}

\usepackage{verbatim}
\usepackage[normalem]{ulem}

\title{Biexact von Neumann algebras}
\author{Changying Ding}
\author{Jesse Peterson}

\address{Department of Mathematics, University of California, Los Angeles, Los Angeles, CA 90095, USA}
\email{cding@math.ucla.edu}

\address{Department of Mathematics, Vanderbilt University, 1326 Stevenson Center, Nashville, TN 37240, USA}
\email{jesse.d.peterson@vanderbilt.edu}

\newtheorem{thm}{Theorem}[section]
\newtheorem{prop}[thm]{Proposition}
\newtheorem{cor}[thm]{Corollary}
\newtheorem{lem}[thm]{Lemma}
\theoremstyle{definition}
\newtheorem{defn}[thm]{Definition}
\newtheorem{defn/lem}[thm]{Definition/Lemma}
\newtheorem{rem}[thm]{Remark}
\newtheorem{examp}[thm]{Example}

\newcommand{\B}{{\mathbb B}}
\newcommand{\C}{{\mathbb C}}

\newcommand{\K}{{\mathbb K}}
\newcommand{\M}{{\mathbb M}}
\newcommand{\N}{{\mathbb N}}

\newcommand{\bS}{{\mathbb S}}

\newcommand{\X}{{\mathbb X}}
\newcommand{\Y}{{\mathbb Y}}

\newcommand{\HH}{{\mathcal H}}

\newcommand{\cP}{{\mathcal P}}

\newcommand{\cC}{{\mathcal C}}

\newcommand{\cH}{{\mathcal H}}

\newcommand{\cK}{{\mathcal K}}

\newcommand{\cU}{{\mathcal U}}

\newcommand{\Ad}{\operatorname{Ad}}

\newcommand{\id}{\operatorname{id}}

\newcommand{\diag}{\operatorname{diag}}

\newcommand{\ovt}{\, \overline{\otimes}\,}
\newcommand{\oovt}[1]{\, \overline{\otimes}_{#1}\,}

\newcommand{\alg}{{\operatorname{alg}}}

\newcommand{\ds}{{\sharp\kern-.5pt\sharp}}

\newcommand{\actson}{{\, \curvearrowright \,}}

\makeatletter
\DeclareRobustCommand\frownotimes{\mathbin{\mathpalette\frown@otimes\relax}}
\newcommand{\frown@otimes}[2]{%
  \vbox{
    \ialign{##\cr
      \hidewidth$\m@th#1{}_\frown$\kern-\scriptspace\hidewidth\cr
      \noalign{\nointerlineskip\kern-1pt}
      $\m@th#1\otimes$\cr
    }%
  }%
}
\makeatother

\begin{document}

\maketitle
\begin{abstract}
We introduce the notion of biexactness for general von Neumann algebras, naturally extending the notion from group theory. We show that biexactness implies solidity for von Neumann algebras, and that many of the examples of solid von Neumann algebras contained in the literature are, in fact, biexact. We also give examples of certain crossed products arising from Gaussian actions that are solid but not biexact, and we give examples of certain $q$-Gaussian von Neumann algebras that are strongly solid but not biexact. The techniques developed involve studying a certain weak form of nuclear embeddings, and we use this setting to give a new description of weak exactness for von Neumann algebras, which allows us to answer several open problems in the literature about weakly exact von Neumann algebras. 
\end{abstract}

\section{Introduction}

Recall that a finite von Neumann algebra $M$ is solid if, for any diffuse von Neumann subalgebra $A \subset M$, the relative commutant $A' \cap M$ is amenable. A seminal result of Ozawa \cite{Oza04} is that free group factors are solid and, as a consequence, every nonamenable subfactor of a free group factor is prime. Ozawa's proof, which is based on ${\rm C}^*$-algebra theory, holds more generally for the class of biexact groups \cite[Definition 15.1.2]{BO08}, which are groups that admit a topologically amenable action on a certain small-at-infinity boundary. Biexact groups have since been shown to have many applications to the theory of von Neumann algebras \cite{Oza04, Oza06, Oza06b}. Combining the techniques of biexact groups with Popa's deformation/rigidity theory has produced a number of striking results where certain structural properties of a von Neumann algebra can be recovered as in \cite{OzPo04, CS13, PoVa14b, CdSS16, CaIsWa21, HaIsKa22, Ma17}.

In this article we introduce the notion of biexactness in the setting of von Neumann algebras, thereby allowing for the previous results to be put in a systematic framework. This allows for a more integrated use of ${\rm C}^*$ and von Neumann algebraic techniques, and allows us to find a common von Neumann algebraic setting for many solidity results such as those obtained in \cite{Oza04, po07b, Pet09, ChIo18, Bo13, HoRa15, Is15, Shl04}. This then leads to natural extensions of these results to a larger class of von Neumann algebras. 

To describe biexactness for von Neumann algebras, we first recall in greater detail the corresponding property for groups (see \cite[Chapter 15]{BO08} for full details). Given a countable group $\Gamma$, we have commuting actions of $\Gamma$ on $\ell^\infty \Gamma$ given by left and right multiplication, respectively, i.e., $L_t(f) (x) = f(t^{-1} x)$ and $R_t(f)(x) = f( xt)$ for $f \in \ell^\infty \Gamma$ and $x, t \in \Gamma$. We denote 
\[
S(\Gamma) = \{ f \in \ell^\infty \Gamma \mid f - R_t(f) \in c_0\Gamma, \ \forall\  t \in \Gamma \},
\]
which is a left-invariant ${\rm C}^*$-subalgebra of $\ell^\infty \Gamma$. By Gelfand duality, we may identify $S(\Gamma)$ with continuous functions on its spectrum $\overline \Gamma$, and we then have an action of $\Gamma$ on $\overline \Gamma$ by homeomorphisms. The space $\overline \Gamma$ is called the small-at-infinity compactification of $\Gamma$. The group $\Gamma$ is biexact if the action $\Gamma \actson \overline \Gamma$ is topologically amenable in the sense of \cite{AD87}, which is equivalent to the inclusion $C^*_\lambda \Gamma \subset C(\overline \Gamma) \rtimes_r \Gamma$ being nuclear. Topological amenability is inherited through factor maps, and so in practice one usually checks biexactness of a group via natural boundary type actions, e.g., hyperbolic groups acting on their Gromov boundary. 

Let now $M$ be a von Neumann algebra and $X \subset \B(\mathcal H)$ an operator $M$-bimodule. Magajna introduced in \cite{Ma98, Ma00} the $M$-topology on $X$, which is a locally convex topology intermediate between the uniform and ultraweak topologies (see Section~\ref{sec:Mtopology} for the precise definition). As $\B(L^2M)$ is both an operator $M$-bimodule and an $M'$-bimodule, we may let $\K^{\infty, 1}(M) \subset \B(L^2M)$ denote the closure of $\K(L^2M)$ under both the $M$-topology and the $M'$-topology. In the case when $M$ has a normal faithful trace $\tau$, we may view $M \subset L^2(M, \tau) \subset L^1(M, \tau)$, and then by \cite[Proposition 3.1]{DKEP22} we have that $\K^{\infty, 1}(M)$ is the closure of $\K(L^2M)$ in the topology given by the norm $\| T \|_{\infty, 1} = \sup_{x \in M, \| x \| \leq 1} \| T \hat{x} \|_1$. 

In \cite{DKEP22} we introduced a von Neumann analog of the small-at-infinity compactification as
\[
\bS(M)= \{ T \in \B(L^2M) \mid [T, a] \in \K^{\infty, 1}(M),  \forall\ a \in M' \}. 
\]
The advantage of working with the generalized compact operators from $\K^{\infty, 1}(M)$ is that to check that an operator $T$ is contained in $\bS(M)$, it suffices to check that $[T, a] \in \K^{\infty, 1}(M)$ for $a$ in some generating set of $M$ (generated as a von Neumann algebra). In particular, if $[T, a]$ is compact for each element $a$ in some generating set, then $T \in \bS(M)$. This shows that the space $\bS(M)$ is robust enough so that the inclusion $M \subset \bS(M)$ may have interesting properties. For example, if $M = L\Gamma$ for a group $\Gamma$, then we have a natural inclusion $S(\Gamma) \rtimes_r \Gamma \subset \bS(L\Gamma)$. Also, the space $\bS(M)$ is large enough so that it admits non-trivial 1-cohomology \cite[Theorem 7.3]{HiPePo23}. A consequence of this robustness, though, is that the the space $\bS(M)$ is not a ${\rm C}^*$-algebra, but is rather an $M$-system, i.e., an operator system that is also an operator $M$-bimodule. 

Since $\bS(M)$ is an operator $M$-bimodule, we may again consider the $M$-topology on $\bS(M)$. We say that the von Neumann algebra $M$ is biexact if the inclusion $M \subset \bS(M)$ is $M$-nuclear in the sense that there exists nets of u.c.p.\ maps $\phi_i: M \to \mathbb M_{n(i)}(\mathbb C)$ and $\psi_i: \mathbb M_{n(i)}(\mathbb C) \to \bS(M)$ so that for each $x \in M$ we have $\psi_i \circ \phi_i(x) \to x$ in the $M$-topology. 

This notion is flexible enough to show that many von Neumann algebras already studied in the literature are biexact. In particular, we show in Theorem~\ref{thm:condition AO+} that all von Neumann algebras satisfying condition $(AO)^+$ or strong property $(AO)$ are biexact. We therefore have that each of the following classes of von Neumann algebras is contained in the class of biexact von Neumann algebras:

\begin{enumerate}
\item Amenable von Neumann algebras.
\item Group von Neumann algebras associated to locally compact second countable biexact groups (Proposition 6.2 in \cite{Dep20}).
\item Von Neumann algebras of universal orthogonal and unitary discrete quantum groups (Proposition 3.1.2 in  \cite{Iso15}).
\item Free Araki-Woods von Neumann algebras (The proof of Theorem 4.4 in \cite{Ho07}).
\item $q$-Gaussian von Neumann algebras associated to finite-dimensional real Hilbert space (Section 4 in \cite{Shl04} for some $q$, and \cite{Ku22} in general). 
\item Free products of biexact von Neumann algebras with respect to normal faithful states (Proposition~\ref{prop:free product biexact}).
\item Von Neumann subalgebras with expectation of biexact von Neumann algebras (Proposition~\ref{prop:subalgebra}). 
\item Amplifications and commutants of biexact von Neumann algebras (Proposition~\ref{prop:amplificiationstable}).
\end{enumerate}

In the case of a group von Neumann algebra associated to a discrete group $M = L\Gamma$, we have a canonical diagonal embedding of $\ell^\infty \Gamma$ into $\B(\ell^2 \Gamma)$, as well as a normal conditional expectation $E_{\ell^\infty \Gamma}: \B(\ell^2 \Gamma) \to \ell^\infty \Gamma$. We show that $E_{\ell^\infty \Gamma}( \bS(L\Gamma) ) \subset S(\Gamma)$, and so if $L\Gamma$ is biexact, then by considering the positive type functions $\Gamma \ni t \mapsto E_{\ell^\infty}( \psi_i \circ \phi_i( \lambda_t) \lambda_t^* ) \in S(\Gamma)$ we deduce that the action $\Gamma \actson \overline \Gamma$ is amenable. The converse implication also holds by using local reflexivity of $C^*_\lambda \Gamma$ to show that nuclearity of the inclusion $C^*_\lambda \Gamma \subset C(\overline \Gamma) \rtimes_r \Gamma$ can be upgraded to give $L\Gamma$-nuclearity for the inclusion $L\Gamma \subset \bS(L\Gamma)$. A consequence, therefore, of this new framework is that we deduce that biexactness for groups is stable under $W^*$-equivalence (see Corollary~\ref{cor:w*-equivalence} for the complete details), i.e.\ if $\Gamma$ and $\Lambda$ are countable groups such that $L\Gamma \cong L\Lambda$, then $\Gamma$ is biexact if and only if $\Lambda$ is biexact. This should be compared with Sako's result \cite{Sa09} where he shows that biexactness for groups is stable under measure equivalence.

One of the technical achievements of our approach (carried out in Section~\ref{sec:nuclearity}) is that, in fact, we do not need to use local reflexivity of $C^*_\lambda \Gamma$ directly.  If we have a von Neumann algebra $M$ and an ultraweakly dense ${\rm C}^*$-subalgebra $M_0$, we consider a version of the $M$-topology called the $(M_0 \subset M)$-topology that takes the dense ${\rm C}^*$-algebra $M_0$ into account. If the inclusion $M_0 \subset \bS(M)$ is $(M_0 \subset M)$-nuclear, we then have a general upgrading result showing that $M$ is biexact. While this general result is not needed to show that biexactness is stable under $W^*$-equivalence, it is crucial in showing how biexactness interacts with general constructions. For instance, this allows us to show that the free product of biexact von Neumann algebras (with respect to arbitrary normal faithful states) is again biexact. This also allows us to show that tensor products of biexact von Neumann algebras satisfy a relative biexactness property from which one deduces the unique prime decomposition results of Ozawa and Popa \cite{OzPo04}. This holds, in fact, for general biexact von Neumann algebras, and puts this into the framework of Houdayer and Isono \cite{HoIs17} so that unique prime decomposition results can be deduced also in the type III setting as well (see Corollary~\ref{cor:primedecomposition}). 

If we consider the larger $M$-bimodule $\B(L^2M)$ instead of $\mathbb S(M)$, then $(M_0 \subset M)$-nuclearity for the inclusion $M \subset \B(L^2M)$ is related to Isono's notion of $M_0$ being weakly exact in $M$ \cite{Iso13}. In Section~\ref{sec:wkexact} we make this connection explicit, and as an application we show in Corollary~\ref{cor:free product weak exact} that the free product (with respect to arbitrary normal faithful states) of weakly exact von Neumann algebras is again weakly exact, which was previously shown by Isono to hold in many cases. We also use this perspective to provide a new characterization of weak exactness for von Neumann algebras, answering Problem 10.4.3 from \cite{BO08} (see Theorem~\ref{thm:brownozawa}).

Biexact von Neumann algebras in general also have many of the indecomposability type properties that are known to hold for the group von Neumann algebras associated to biexact groups. We show a biexact von Neumann algebra $M$ is always full, solid, and properly proximal. Moreover, if $M$ is finite and weakly amenable, then $M$ is also strongly solid in the sense of Ozawa and Popa \cite{OzPo10a}. 

It was shown in \cite{DKEP22} that if $\Gamma$ is a nonamenable group and we have an embedding $L\Gamma \subset L\mathbb F_2$, then it follows that $\Gamma$ is properly proximal and consequently is not inner amenable (an alternate proof of this fact was also found recently in \cite{Dr22}). Since biexactness is inherited to subalgebras it follows that $\Gamma$ is even biexact. In fact, what we establish is a general principle that higher-rank lattices cannot ``$W^*$-embed'' into rank one lattices. Specifically, if $G$ is a connected real semisimple Lie group with finite center and $\mathbb R{\rm -rank}(G) \geq 2$, and if $\Gamma < G$ is a lattice, then by \cite[Exercise 8.2.7]{Mo15}, together with \cite{Sa09, Dep20}, it follows that $\Gamma$ is not biexact and hence $L\Gamma$ cannot embed into $L\Lambda$ for any lattice $\Lambda$ in a rank one connected real simple Lie group with finite center.

It was shown recently by Caspers \cite{Cas22} that if $\mathcal H$ is an infinite dimensional real Hilbert space and $-1 < q < 1$ with $q \not= 0$, then the $q$-Gaussian von Neumann algebra $M_q(\mathcal H)$ is not isomorphic to a free group factor. Casper's result shows, in fact, that $M_q(\mathcal H)$ is not biexact, and hence it then follows that $M_q(\mathcal H)$ cannot even embed into a free group factor. This last fact can also be deduced directly from \cite{Oza10}, together with Theorem~\ref{thm:weakcontainultrapower} below, which is independent of the results in the earlier sections. We remark that this can also be deduced by combining Caspers work with Corollary 3.5 in \cite{Mar23}. 

Despite the fact that $M_q(\mathcal H)$ is not biexact when $-1 < q < 1$, $q\not= 0$, and $\mathcal H$ is infinite dimensional, we show in Section~\ref{sec:solidnotbiexact}, that $M_q(\mathcal H)$ shares many of the same indecomposability results that biexact von Neumann algebras have. In particular, it is strongly solid and every von Neumann subalgebra not having an amenable summand is properly proximal. Strong solidity for $M_q(\mathcal H)$ was shown previously by Avsec \cite{Av11} in the case when $\mathcal H$ is finite-dimensional. Our proof in the infinite dimensional case uses deformation/rigidity techniques, together with techniques developed in \cite{DKEP22}, to reduce the problem to the finite-dimensional case where Avsec's results may then be applied. 

In Section~\ref{sec:solidnotbiexact} we also give another source of solid von Neumann algebras that are not biexact through the use of Gaussian actions. Specifically, we show that if $\Gamma$ is a biexact group and $\pi: \Gamma \to \mathcal O(\mathcal H)$ is an orthogonal representation satisfying $\pi \not\prec \lambda$, but $\pi^{\otimes k} \prec \lambda$ for some $k > 1$, then the crossed product of the Gaussian action $A_{\mathcal H \ovt \ell^2 \mathbb N} \rtimes^{\sigma_{\pi \otimes 1}} \Gamma$ is solid, but is not biexact. Solidity of this von Neumann algebra is a consequence of Boutonnet's solid ergodicity result from \cite{Bo12}, and in Section~\ref{sec:solidnotbiexact} we also give a relatively direct proof of this result. 

\subsection*{Acknowledgements} We thank Cyril Houdayer and Amine Marrakchi for useful comments. We also thank the referee for a careful reading and helpful comments that improved the exposition.

\section{Preliminaries and notation}

For ${\rm C}^*$-algebras or operator systems $A$ and $B$, we denote the algebraic tensor product by $A \odot B$ and the minimal tensor product by $A \otimes B$. We will use the abbreviation c.c.\ for completely contractive, c.p.\ for completely positive, c.c.p.\ for contractive completely positive, and u.c.p.\ for unital completely positive.

For a von Neumann algebra $M$, we denote by $(M, L^2M, J, \mathfrak B)$ its standard form \cite{Ha75}. If $M$ is $\sigma$-finite and we have a normal faithful state $\mu$ on $M$, then the standard form may be realized as $(M, L^2(M, \mu), J_\mu, \mathfrak B_\mu)$, where the representation of $M$ on $L^2(M, \mu)$ is given by the GNS-construction, $J_\mu$ is modular conjugation on $L^2(M, \mu)$, and $\mathfrak B_\mu$ is the positive part of $L^2(M, \mu)$. We denote by $\mu^{1/2} \in L^2(M, \mu)$ the canonical cyclic vector so that $\mu( x) = \langle x \mu^{1/2}, \mu^{1/2} \rangle$ for $x \in M$. The centralizer algebra with respect to $\mu$ is denoted by $M^\mu$.

For von Neumann algebras $M$ and $N$, a normal Hilbert $M$-$N$ bimodule is a Hilbert space $\mathcal H$, together with a unital $*$-representation $\pi_{\mathcal H}$ of the algebraic tensor product $M \odot N^{\rm op}$ such that $\pi_{\mathcal H}$ is normal and faithful when restricted to $M$ and $N^{\rm op}$ separately. We will write simply $x \xi y = \pi_{\mathcal H}(x \otimes y^{\rm op}) \xi$, for $x \in M$, $y \in N$, and $\xi \in \mathcal H$. A normal Hilbert $M$-$N$ bimodule $\mathcal H$ is weakly contained in another normal Hilbert $M$-$N$ bimodule $\mathcal K$ (written $\mathcal H \prec \mathcal K$) if the identity map on $M \odot N^{\rm op}$ extends to a $*$-homomorphism from $C^*( \pi_{\mathcal K}(M \odot N^{\rm op}))$ to $C^*(\pi_{\mathcal H}(M \odot N^{\rm op}))$. The trivial bimodule for $M$ is given by $L^2M$ with the bimodule structure $x \xi y = x Jy^*J \xi$, and the coarse $M$-$N$ bimodule is given by $L^2 M \ovt L^2 N$ with the bimodule structure given by the canonical representation of $M \otimes N^{\rm op}$ in $\B(L^2M \ovt L^2 N)$. 

We will say that a von Neumann subalgebra $N \subset M$ is with expectation if there exists a normal faithful conditional expectation $E: M \to N$. Given a normal faithful conditional expectation $E: M \to N$, and a normal faithful semifinite weight $\psi_N$ on $N$, we obtain a normal faithful semifinite weight on $M$ by $\psi = \psi_N \circ E$, and this gives rise to an inclusion $L^2(N, \psi_N) \subset L^2(M, \psi)$. The Jones projection $e_N: L^2(M, \psi) \to L^2(N, \psi_N)$ is the corresponding orthogonal projection. The Jones projection does not depend on the normal faithful semifinite weight $\psi_N$ \cite[Appendix A]{HoIs17} and so we may consider it abstractly as a coisometry $e_N: L^2M \to L^2N$. The Jones projection $e_N$ satisfies 
\begin{enumerate}
\item $e_N^* x e_N = E(x)$, for $x \in M$. 
\item $e_N J_M = J_N e_N$. 
\end{enumerate}

If we have a group action $\Gamma \actson^\sigma M$ by automorphisms, then by \cite[Theorem 3.2]{Ha75} there exists a unique unitary representation $\sigma^0: \Gamma \to \mathcal U(L^2M)$ so that for each $t \in \Gamma$ we have
\begin{enumerate}
\item $\sigma_t(x) = \sigma_t^0 x (\sigma_t^0)^*$, for $x \in M$. 
\item\label{item:koopman2} $J \sigma_t^0 = \sigma_t^0 J$.
\item $\sigma_t^0 ( \mathfrak B) = \mathfrak B$. 
\end{enumerate}
We call the representation $\sigma^0: \Gamma \to \mathcal U(L^2M)$ the Koopman representation associated to the action $\sigma$. 

If we have an action $\Gamma \actson A$ of a group on a ${\rm C}^*$-algebra, then there exists a faithful nondegenerate representation $A \subset \B(\mathcal H)$ and a covariant representation of $\pi: \Gamma \to \mathcal U(\mathcal H)$. The reduced crossed-product $A \rtimes_r \Gamma$ is the ${\rm C}^*$-subalgebra generated by $\{ (a \otimes 1) \pi_t \otimes \lambda_t \mid a \in A, t \in \Gamma \}$, where $\lambda: \Gamma \to \mathcal U(\ell^2 \Gamma)$ denotes the left-regular representation. This ${\rm C}^*$-algebra is independent of the covariant representation, and we identify $A$ with $A \otimes \mathbb C$ so that $A \subset A \rtimes_r \Gamma$.

If we have an $\Gamma \actson M$ of a group on a von Neumann algebra, then we denote by $M \rtimes \Gamma \subset \B(L^2 M \ovt \ell^2 \Gamma)$ the von Neumann crossed-product, which is the von Neumann completion of $M \rtimes_r \Gamma \subset \B(L^2 M \ovt \ell^2 \Gamma)$. We denote by $u_t$ the unitary $\sigma^0_t \otimes \lambda_t \in M \rtimes_r \Gamma$, although if $M = \mathbb C$, then in this case we have $u_t = \lambda_t$ and we will use both notations interchangeably. The von Neumann algebra $M \rtimes \Gamma$ is standardly represented on $L^2 M \ovt \ell^2 \Gamma$, and we may explicitly compute the modular conjugation operator to see that it satisfies 
\[
J_{M \rtimes \Gamma} ( \xi \otimes \delta_t ) 
= \sigma_{t^{-1}}^0 J_M \xi \otimes \delta_{t^{-1}}
\]
for $\xi \in L^2M$ and $t \in \Gamma$. If we use the canonical identification $L^2 M \ovt \ell^2 \Gamma \cong \oplus_{t \in \Gamma} L^2 M$, then for $x \in M$ and $t \in \Gamma$ we have
\begin{equation}\label{eq:modulargp}
J_{M \rtimes \Gamma} \, x  \, J_{M \rtimes \Gamma} 
= \oplus_{t \in \Gamma} J_M \sigma_t(x) J_M, \ \ \ \ \ \ \ \ \ \ 
J_{M \rtimes \Gamma} \, u_t \, J_{M \rtimes \Gamma} 
= 1 \otimes \rho_t,
\end{equation}
where $\rho: \Gamma \to \mathcal U(\ell^2 \Gamma)$ denotes the right-regular representation. 

If $M$ is a von Neumann algebra, $p, q \in \mathcal P(M)$ are nonzero projections and $B \subset pMp$, $Q \subset qMq$ are von Neumann subalgebras, then we say the $B$ embeds with expectation into $Q$ inside of $M$ and write $B \preceq M$ if there exist projections $e \in \mathcal P(B)$, $f \in \mathcal P(Q)$, a nonzero partial isometry $v \in e M f$ and a unital normal $*$-homomorphism $\phi: e B e \to f Q f$ such that the inclusion $\phi( e B e) \subset f Q f$ is with expectation and $b v = v \phi(b)$ for all $b \in B$.  By a fundamental theorem of Popa \cite[Theorem 2.1]{Po06B}, for a tracial von Neumann algebra $M$, we have $B \not\preceq_M Q$ if and only if there exists a net $\{ u_i \}_i \subset \mathcal U(B)$ so that $E_Q( a u_i b) \to 0$ in the ultrastrong-$^*$ topology for all $a, b \in M$. Popa's theorem has been extended by Houdayer and Isono \cite[Theorem 4.3]{HoIs17} to the case when $B$ is finite and both $B$ and $Q$ are with expectation. For further developments of intertwining theory in the nontracial setting, see Isono \cite{Iso22,Iso25}.

\section{A relative topology on ${\rm C}^*$-bimodules}\label{sec:Mtopology}

In this section we collect some needed background information, and we introduce a generalization of Magajna's $M$-$N$-topology on operator bimodules \cite{Ma00}. Many of the results here are adaptations of current techniques in the literature, although a few results (e.g., Lemma~\ref{lem:weakcontinuity}) use new techniques. 

Let $M_0$ and $N_0$ be ${\rm C}^*$-algebras. 
An (concrete) operator $M_0$-$N_0$-bimodule consists of a concrete operator space $X \subset \B(\mathcal H)$, together with two faithful nondegenerate representations $\pi: M_0 \to \B(\HH)$ and $\rho: N_0 \to \B(\HH)$ so that $X$ is an $\pi(M_0)$-$\rho(N_0)$-bimodule, 
	whose bimodule structure is given by composition of operators. 
An (concrete) operator $M_0$-system consists of concrete operator system $X \subset \B(\HH)$, together with a faithful nondegenerate representation $\pi: M_0 \to \B(\HH)$ so that $X$ is also an $\pi(M_0)$-bimodule. Note that we then have $M_0 \cong \pi(M_0) \subset X$. We say that $X$ is a dual operator $M_0$-$N_0$-bimodule (resp.\ dual operator $M_0$-system) if $X$ may be taken as above to be ultraweakly closed. 
	
If $M_0 = M$ and $N_0 = N$ are von Neumann algebras, then the operator $M$-$N$-bimodule (resp.\ operator $M$-system) is normal if this concrete realization can be made so that $\pi$ and $\rho$ (resp.\ $\pi$) is a normal representation. More generally, we will say that an operator $M_0$-$N_0$-bimodule (resp.\ operator $M_0$-system) is $(M_0 \subset M)$-$(N_0 \subset N)$-normal (resp.\ $(M_0 \subset M)$-normal) if there the concrete realization can be made so that $\pi$ and $\rho$ (resp.\ $\pi$) extend to normal representations of $M$ and $N$ (resp.\ $M$).

Let $X$ be a normal operator $M$-$N$-bimodule for von Neumann algebras $M$ and $N$. 
For each pair of positive functionals $\omega \in M_{*}$ and $\rho \in N_*$ we have an associated seminorm on $X$ given by 
$$s^\rho_\omega( x) = \inf \{ \rho(a^* a)^{1/2} \| y \| \omega(b^*b)^{1/2} \mid x = a^*  yb, a, b \in M, y \in X \}.$$ 
Following Magajna \cite{Ma00}, the topology induced by these seminorms will be called the $M$-$N$-topology, or simply the $M$-topology if $M = N$. If $S \subset X$ is a subset, then we denote the closure of $S$ in the $M$-$N$-topology by $\overline{S}^{_{M-N}}$.

A linear functional $\varphi \in X^*$ is $M$-$N$-normal (or simply $M$-normal if $M = N$) 
	if for each $a \in X$ the linear functionals $M \ni x \mapsto \varphi(xa)$ 
	and $N \ni y \mapsto \varphi(ay)$ are normal. 
We let $X^{M\sharp N}$ denote the space of $M$-$N$-normal linear functionals 
	and we call the resulting weak topology on $X$ the weak $M$-$N$-topology (or simply the weak $M$-topology if $M = N$).  
By \cite[Theorem 3.7]{Ma00}, $X^{M\sharp N}$ coincides with the space of $M$-$N$-topology continuous linear functionals on $X$.

More generally, we may consider a similar topology that takes ultraweakly dense ${\rm C}^*$-subalgebras into consideration.
Let $M$ and $N$ be von Neumann algebras, let $M_0 \subset M$ and $N_0 \subset N$ be ultraweakly dense ${\rm C}^*$-subalgebras with $1_M \in M_0$ and $1_N \in N_0$, and let $X$ be an operator $M_0$-$N_0$-bimodule. We may equip $M_0$ with the ultraweak topology from $M$ and we denote by $M_0^\sharp \subset M_0^*$ the space of normal functionals $M_*$, restricted to $M_0$. We similarly define $N_0^\sharp$. 

Given positive functionals $\omega \in M_0^\sharp$ and $\rho \in N_0^\sharp$, we may consider a seminorm on $X$ given by
\[
s_\omega^\rho(x)=\inf\{\rho(a^* a)^{1/2} \|y\| \omega(b^*b)^{1/2}\mid x= a^* y b, a\in C_n(M_0), b \in C_n(N_0), y\in \M_n(X), n\in \N\}.
\]
Here $C_n(A)=\mathbb M_{n,1}(A)$ denotes the space of columns with entries in $A$. The triangle inequality for $s_\omega^\rho$ follows from the standard argument showing that the Haagerup norm on tensor products satisfies the triangle inequality.  

In general, if $M_0$ and $N_0$ do not necessarily contain units, we may use the same formula above to define $s_\omega^\rho$ on the subspace $M_0 X N_0$. Note that if we set $\tilde M_0 = M_0 + \mathbb C 1_M$ and $\tilde N_0 = N_0 + \mathbb C 1_N$, then the seminorm $s_{\omega, N_0}^{\rho, M_0}$ on $M_0 X N_0$ obtained by viewing this as an operator $M_0$-$N_0$ bimodule agrees with the corresponding seminorm $s_{\omega, \tilde N_0}^{\rho, \tilde M_0}$ obtained by viewing $M_0 X N_0$ as a operator $\tilde M_0$-$\tilde N_0$ bimodule. Indeed, we clearly have $s_{\omega, \tilde N_0}^{\rho, \tilde M_0} \leq s_{\omega, N_0}^{\rho, M_0}$, and if $x = a^* y b \in M_0 X N_0$ with $a \in C_n(\tilde M_0)$ and $b \in R_N(\tilde N_0)$, then letting $\{ e_i \}_i$ and $\{ f_j \}_j$ denote approximate identities for $M_0$ and $N_0$ respectively we have $s_{\omega, N_0}^{\rho, M_0}( e_i a^* y b f_j ) \leq \rho(e_i a^* a e_i)^{1/2} \| y \| \omega(f_j b^*b f_j)^{1/2}$. Since $\lim_{i \to \infty} \rho( e_i a^* a e_i ) = \rho(a^*a), \lim_{j \to \infty} \omega(f_j b^*b f_j) = \omega(b^*b)$, and $\lim_{i \to \infty, j \to \infty} \| x - e_i x f_j \| = 0$, it then follows that $s_{\omega, N_0}^{\rho, M_0} \leq  s_{\omega, \tilde N_0}^{\rho, \tilde M_0}$. 

For a general operator $M_0$-$N_0$ bimodule we may then define $s_\omega^\rho$ on all of $X$ by the formula
\[
s_\omega^\rho(x) = \max \left\{ \inf_{z \in M_0 X N_0}  \| x - z \|,  s_{\omega, \tilde N_0}^{\rho, \tilde M_0}(x) \right\}.
\]

We call the topology on $X$ induced by the family of seminorms $\{s^\rho_\omega\mid \omega\in (M_0)^\sharp_+, \rho \in (N_0)^\sharp_+ \}$ the $(M_0 \subset M)$-$(N_0 \subset N)$-topology (or simply the $(M_0 \subset M)$-topology when $M_0 = N_0$ and $M = N$). In the case when $M_0 = M$ and $N_0 = N$, the following lemma is contained in \cite[Lemma 3.1]{Ma98}.

\begin{lem}
Let $M$ and $N$ be von Neumann algebras, let $M_0 \subset M$ and $N_0 \subset N$ be ultraweakly dense ${\rm C}^*$-subalgebras, let $X$ be an operator $M_0$-$N_0$-bimodule, and $Y\subset X$ an operator $M_0$-$N_0$-subbimodule.
Then the restriction of the $(M_0 \subset M)$-$(N_0 \subset N)$-topology on $X$ to $Y$ coincides with the $(M_0 \subset M)$-$(N_0 \subset N)$-topology on $Y$.
\end{lem}
\begin{proof}
Given $\omega \in (M_0)^\sharp_+$ and $\rho \in (N_0)^\sharp_+$, we denote by $s^\rho_{\omega, X}$ and $s^\rho_{\omega, Y}$ the above seminorms on $X$ and $Y$, respectively.
It is clear that $s^\rho_{\omega, X}(y)\leq s^\rho_{\omega, Y}(y)$ for any $y\in Y$.
To see the reverse, suppose first that $1_M \in M_0$ and $1_N \in N_0$ and $y\in Y$ is such that we have a decomposition $y= a^* x b$ with $x \in \M_n(X)$, $a \in C_n(M_0)$, and $b \in C_n(N_0)$.

For $k \geq 1$ we let $f_k: [0, \infty) \to [0, \infty)$ denote the function given by $f_k(t) = \begin{cases}
1/k & 0 \leq t \leq 1/k \\
t & t > 1/k
\end{cases}$, then $y_k = f_k( | a | )^{-1}a^* x b f_k( | b | )^{-1} \in Y$ with $\| y_k \| \leq \| y \|$, and so 
\[
s^\rho_{\omega, Y}( y ) \leq \rho( f_k( | a | )^2 )^{1/2} \| y \| \omega( f_k( | b | )^2 )^{1/2}.
\]
Since $f_k( | a | )^2 \to |a|^2$ and $f_k( | b |)^2 \to | b |^2$ in norm, we then have $s^\rho_{\omega, Y}( y ) \leq s^\rho_{\omega, X}( y )$.

In the non-unital case we notice that if $\{ e_i \}_i$ and $\{ f_j \}_j$ are approximate units for $N_0$ and $M_0$, respectively, then we have $\inf_{z \in M_0 X N_0 } \| x - z \| = \lim_{i \to \infty, j \to \infty} \| x - e_i x f_j \|$ for $x\in X$, and so if $y \in Y$, then we have $\inf_{z \in M_0 X N_0} \| y - z \| = \inf_{z \in M_0 Y N_0} \| y - z \|$. The general result then follows from the unital case. 
\end{proof}

\begin{rem}\label{rem:diagonalseminorm}
If $1_M \in M_0$, $1_N \in N_0$ and we are given an operator $M_0$-$N_0$-bimodule $X$, we may also consider the function 
\[
\tilde s^\rho_\omega(x)=\inf\{\rho(a^2)^{1/2} \|y\| \omega(b^2)^{1/2}\mid x= a y b, a, b\in (M_0)_+, y\in X\}. 
\]
The same argument above shows that if $Y \subset X$ is an $M_0$-$N_0$-subbimodule and $y \in Y$, then the value of $\tilde s^\rho_\omega(y)$ does not depend on whether we consider $y$ as an element in $Y$ or $X$. 

In the case when $X = \B(\mathcal H)$, we may use polar decomposition in $\B(\mathcal H)$ to see that $\tilde s^\rho_\omega = s^\rho_\omega$. Since every operator $M_0$-$N_0$-bimodule is a subbimodule of $\B(\mathcal H)$ for some Hilbert space, it then follows that $s^\rho_\omega = \tilde s^\rho_\omega$ in general. When $M_0 = N_0$, $M = N$, and $\omega = \rho$, we use the notation $s_\omega$ in place of $s_\omega^\omega$.
\end{rem}

Continuing as above with $M$ and $N$ von Neumann algebras, $M_0 \subset M$, $N_0 \subset N$ ultraweakly dense ${\rm C}^*$-subalgebras, and $X$ an operator $M_0$-$N_0$-bimodule, we denote by $X^{M_0\sharp N_0}$ (or simply $X^\sharp$ if no confusion will arise) 
the space of functionals $\varphi\in X^*$ such that for any $x\in X$ the map
	$M_0\times N_0\ni (a,b)\mapsto \varphi(axb)$ is separately ultraweak continuous.
We call the $\sigma(X, X^{M_0\sharp N_0})$-topology on $X$ the weak $(M_0 \subset M)$-$(N_0 \subset N)$-topology (or simply the weak $(M_0 \subset M)$-topology if $M_0 = N_0$ and $M = N$). The following proposition is from \cite[Theorem 3.7]{Ma00}, and we include a proof for convenience.

\begin{prop}\label{prop:A-continuous and topology}
In the above setting, a functional $\varphi\in X^*$ is continuous in the $(M_0 \subset M)$-$(N_0 \subset N)$-topology if and only if $\varphi\in X^{M_0 \sharp N_0}$.
\end{prop}
\begin{proof}
The ``only if'' direction is clear. To see the ``if'' direction, note that 
\[
M_0\times X\times N_0\ni (a, x, b)\mapsto \varphi(axb)\in \C
\]
is a completely bounded multilinear map.
It follows from the Christensen-Paulsen-Sinclair representation theorem (see, e.g., \cite[Theorem 1.5.4]{SinSmi95})
	that there exist $*$-representations $\pi_1: M_0\to \B(\cK_1)$, $\pi_2: N_0 \to \B(\cK_2)$, a complete contraction $\phi: X\to \B(\cH)$,
	 bounded operators $T_1:\cH\to \cK_1$, $T_2: \cK_2\to \cH$, and vectors $\xi_i\in \cK_i$ such that 
	$\varphi(axb)=\langle \pi_1(a) T_1 \phi(x) T_2 \pi_2(b) \xi_2, \xi_1\rangle$.
We let $p_1 \in M_0^{**}$ (resp.\ $p_2 \in N_0^{**}$) denote the support projection for the inclusion $M_0 \to M$ (resp.\ $N_0 \to N$). Since for any $x \in X$ the map $M_0\times N_0\ni (a,b)\mapsto \varphi(axb)$ is separately ultraweak continuous, it then follows that we have $\varphi(axb)=\langle \pi_1(a) T_1 \phi(x) T_2 \pi_2(b) p_2 \xi_2, p_1\xi_1\rangle$ for $x \in X$, $a \in M_0$, and $b \in N_0$.

Then we have $|\varphi(axb)|\leq \|T_1\|\|T_2\| \eta_1(a a^*)^{1/2} \|x\| \eta_2(b^*b)^{1/2}$,
	where $\eta_i(\cdot)=\langle \pi_i(\cdot) p_i\xi_i, p_i\xi_i\rangle$.
Since $\eta_1 \in M_0^\sharp$ and $\eta_2 \in N_0^\sharp$, it is then easy to conclude that $\varphi$ is continuous in the $(M_0 \subset M)$-$(N_0 \subset N)$-topology.
\end{proof}

The following result is in the same spirit as the Noncommutative Egorov Theorem.

\begin{lem}\label{lem:compressnorm}
Let $M$ be a von Neumann algebra, $M_0\subset M$ an ultraweakly dense ${\rm C}^*$-subalgebra and $X$ an operator $M_0$-bimodule.
Suppose for each $1\leq k\leq n$, $\{x_i^k\}_{i\in I}\subset X$ is a net converging to $0$ in the $(M_0\subset M)$-topology.
Then, for any finite subset $F=\{\varphi_j\mid 1\leq j\leq m\}\subset (M_0^\sharp)_+$, we may find a net of positive contractions $1 - z_i \in M_0$, 
	such that $\| z_i x_i^k z_i \|\to 0$ and $\varphi_j(1 - z_i)\to 0$ for any $1\leq j\leq m$.
When $M_0=M$, we may choose $\{z_i\}$ to be projections.
\end{lem}
\begin{proof}
Note first that it suffices to consider the case when $\{x_i^k\}_{i\in I}\subset M_0XM_0$. Put $\varphi=\sum_{j=1}^m\varphi_j$. For each $i$ and $1\leq k\leq n$, after replacing $I$ by $I\times\mathbb N$ and relabelling, we may then find a decomposition $x_i^k=a_{i,k}^*y_{i,k}b_{i,k}$, with $a_{i,k},b_{i,k}\in M_0$, $y_{i,k}\in X$, and $\varphi(a_{i,k}^*a_{i,k})^{1/2}+\|y_{i,k}\|+\varphi(b_{i,k}^*b_{i,k})^{1/2}\to0$.

We fix $N>1$ and let $f_N:[0,\infty)\to[0,1]$ denote a piecewise linear continuous function such that $f_N(t)=1$ if $t\leq1/N$, and $f_N(t)=0$ if $t\geq2/N$. We then define
\[
z_i=f_N\left(\sum_{k=1}^n(a_{i,k}^*a_{i,k}+b_{i,k}^*b_{i,k})\right)\in M_0+\mathbb C1_M.
\]
Note that $0\leq z_i\leq1$, $1-z_i\in M_0$, and
\[
\|z_i a_{i,k}^*\|^2
\leq\left\|z_i\left(\sum_{k=1}^n(a_{i,k}^*a_{i,k}+b_{i,k}^*b_{i,k})\right)z_i\right\|
\leq2/N.
\]
Similarly, $\|b_{i,k}z_i\|^2\leq2/N$, and it follows that $\|z_ix_i^kz_i\|\to0$ for every $1\leq k\leq n$.

Also, we have $1-z_i\leq N\sum_{k=1}^n(a_{i,k}^*a_{i,k}+b_{i,k}^*b_{i,k})$, and hence $\varphi(1-z_i)\to0$. Since $\varphi_j\leq\varphi$, we obtain $\varphi_j(1-z_i)\to0$ for every $j$.

When $M_0=M$, replace $f_N$ with $1_{[0,1/N]}$, so that $z_i$ is a projection in $M$.
\end{proof}

The next lemma will allow us to adapt the standard c.c.\ to u.c.p.\ perturbation result to the setting of $M$-topology.

\begin{lem}\label{lem:norm converge on unit}
Let $M$ be a von Neumann algebra, and let $M_0\subset M$ be an ultraweakly dense unital ${\rm C}^*$-subalgebra.
Let $E$ be an operator $M_0$-system, and for $i\in I$ suppose $E_i$ is an operator system and $\phi_i:E_i\to E$ is a net of c.c.\ maps such that $1_E-\phi_i(1_{E_i})\to0$ in the $(M_0\subset M)$-topology.
Then, for any finite subset of states $F\subset M_0^\sharp$, there exist c.c.\ maps $\psi_i:E_i\to E$ such that $\|1_E-\psi_i(1_{E_i})\|\to0$ and
\[
\lim_i\sup_{x\in(E_i)_1}s_\omega^\rho((\psi_i-\phi_i)(x))=0
\]
for any $\omega,\rho\in F$.
\end{lem}
\begin{proof}
Since $1_E-\phi_i(1_{E_i})\to0$ in the $(M_0\subset M)$-topology, it follows as in the proof of Lemma~\ref{lem:compressnorm} that there exist $1-z_i\in M_0$, with $0\leq z_i\leq1$, such that $\omega(1-z_i^2)\to0$ for every $\omega\in F$ and $\|z_i(1_E-\phi_i(1_{E_i}))z_i\|\to0$.

For each $i$, fix a state $\eta_i\in E_i^*$ and define $\psi_i:E_i\to E$ by
\[
\begin{aligned}
\psi_i(x)&=z_i\phi_i(x)z_i+(1-z_i^2)\eta_i(x)\\
&=\begin{pmatrix}z_i\\(1-z_i^2)^{1/2}\end{pmatrix}^{\!*}
\begin{pmatrix}\phi_i(x)&0\\0&\eta_i(x)\end{pmatrix}
\begin{pmatrix}z_i\\(1-z_i^2)^{1/2}\end{pmatrix}.
\end{aligned}
\]
Then, $\psi_i$ is c.c.\ and satisfies $\|1_E-\psi_i(1_{E_i})\|\to0$. Moreover, for $x\in E_i$ and $\omega,\rho\in F$ we have
\[
\begin{aligned}
s_\omega^\rho(\phi_i(x)-\psi_i(x))
&\leq\rho(1-z_i^2)^{1/2}\omega(1-z_i^2)^{1/2}\|x\|+s_\omega^\rho\bigl((1-z_i)\phi_i(x)+z_i\phi_i(x)(1-z_i)\bigr)\\
&\leq\Bigl(\rho(1-z_i^2)^{1/2}+\omega(1-z_i^2)^{1/2}+\rho(1-z_i^2)^{1/2}\omega(1-z_i^2)^{1/2}\Bigr)\|x\|.
\end{aligned}
\]
\end{proof}

\begin{lem}\label{lem:perturb}
Let $M\subset \B(\cH)$ be a von Neumann algebra, $M_0\subset M$ an ultraweakly dense unital ${\rm C}^*$-subalgebra 
	and $\psi_i: E_i\to \B(\cH)$ a net of c.c.\ maps from operator systems $E_i$ to $\B(\cH)$.
Suppose $\psi_i(1_{E_i})\to 1$ in the $(M_0\subset M)$-topology.
Then, for any finite set of states $F\subset M_0^\sharp$, 
	there exists a net of u.c.p.\ maps $\phi_i: E_i\to \B(\cH)$ such that 
	$\lim_i\sup_{x\in (E_{i})_1}s_\omega^\rho((\phi_i-\psi_i)(x))= 0$ for any $\omega, \rho\in F$.
\end{lem}
\begin{proof}
By Lemma~\ref{lem:norm converge on unit}, we have c.c.\ maps $\psi_i':E_i\to\B(\cH)$ such that $\psi_i'(1_{E_i})\to1$ in norm and $\lim_i\sup_{x\in(E_i)_1}s_\omega^\rho((\psi_i-\psi_i')(x))=0$ for any $\omega,\rho\in F$.

By \cite[Theorem B.7]{BO08}, write $\psi_i'(x)=V_i^*\pi_i(x)W_i$, where $\pi_i$ is a unital representation and $V_i,W_i$ are isometries. Define the u.c.p.\ map $\phi_i(x)=V_i^*\pi_i(x)V_i$. Then
\[
\|\phi_i-\psi_i'\|_{\mathrm{cb}}\leq\|V_i-W_i\|=\bigl\|2-\psi_i'(1_{E_i})-\psi_i'(1_{E_i})^*\bigr\|^{1/2}\leq\sqrt{2\|1-\psi_i'(1_{E_i})\|}.
\]
It follows that, for any $\omega,\rho\in F$,
\[
\begin{aligned}
\sup_{x\in(E_i)_1}s_\omega^\rho((\phi_i-\psi_i)(x))
&\leq\|\phi_i-\psi_i'\|_{\mathrm{cb}}+\sup_{x\in(E_i)_1}s_\omega^\rho((\psi_i'-\psi_i)(x))\\
&\leq\sqrt{2\|1-\psi_i'(1_{E_i})\|}+\sup_{x\in(E_i)_1}s_\omega^\rho((\psi_i'-\psi_i)(x))\longrightarrow0.
\end{aligned}
\]
\end{proof}

\subsection{$M$-${\rm C}^*$-algebras and normal biduals}

If $M_0$ is a ${\rm C}^*$-algebra, then an $M_0$-${\rm C}^*$-algebra consists of a ${\rm C}^*$-algebra $A$, together with a faithful nondegenerate $*$-homomorphism from $M_0$ into the multiplier algebra $M(A)$.  If $M_0 = M$ is a von Neumann algebra, then an $M$-${\rm C}^*$-algebra $A$ is normal if $M(A)$ is a normal operator $M$-system. More generally, if we have an ultraweakly dense ${\rm C}^*$-algebra $M_0\subset M$,
then an $M_0$-${\rm C}^*$-algebra $A$ is $(M_0 \subset M)$-normal if $M(A)$ is $(M_0 \subset M)$-normal as an operator $M_0$-system. 

If $M$ is a von Neumann algebra and $M_0 \subset M$ is an ultraweakly dense ${\rm C}^*$-subalgebra, and if we have a $(M_0 \subset M)$-normal $M_0$-${\rm C}^*$-algebra $A$, then we let $p_{\rm nor} \in M_0^{**} \subset A^{**}$ denote the projection corresponding to the support of the identity representation $M_0 \to M$. On the predual $( p_{\rm nor} A^{**} p_{\rm nor} )_*$ we may consider the restriction map to $A$, and this gives rise to an operator space isomorphism $( p_{\rm nor} A^{**} p_{\rm nor} )_* \cong A^{ \sharp }$ (see \cite[Section 2]{DKEP22}). The dual map then allows us to equip $A^{ \sharp *}$ with a von Neumann algebraic structure so that 
\[
A^{ \sharp *} \cong p_{\rm nor} A^{**} p_{\rm nor}.
\]
Note that $p_{\rm nor}$ commutes with $M_0 \subset M(A) \subset A^{**}$, and so this isomorphism preserves the natural $M_0$-bimodule structures on $A^{\sharp *}$ and $p_{\rm nor} A^{**} p_{\rm nor}$, so we will view $A^{\sharp *}$ as a von Neumann algebra that contains $M$ as a von Neumann subalgebra.

It is worth noting that the canonical mapping of $A$ into $A^{\sharp *}$ is only a complete order isomorphism, rather than a $*$-isomorphism, since $p_{\rm nor}$ need not to be central in $A^{**}$ if $A$ is not generated as a ${\rm C}^*$-algebra by $M_0$ and $M_0' \cap A$. The von Neumann algebra $A^{\sharp *}$ may be seen as the universal version of the construction considered in \cite[Proposition 3.1]{BoCa15}.

If $A$ is unital and $E \subset A$ a normal operator $M_0$-subsystem, then we may equip $E^{\sharp *}$ with a dual normal operator $M$-system structure, 
	by identifying it as an operator $M$-subsystem of $A^{\sharp *}$. We will denote by $i_E: E \to  E^{\sharp *}$ 
	the canonical complete order isomorphism.

If we have an inclusion of if $(M_0 \subset M)$-normal $M_0$-${\rm C}^*$-algebras $J \subset A$, then we have a canonical (not necessarily unital) embedding of von Neumann algebras $J^{**} \subset A^{**}$, and, by considering an $M_0$-quasi-central approximate unit in $J$, we see that the support projection $p_0$ of $J^{**}$ in $A^{**}$ commutes $p_{\rm nor} \in A^{**}$ and the corresponding projection $p_0 p_{\rm nor}$ gives the corresponding normal projection in $J^{**}$. 
Thus, we also have a canonical (not necessarily unital) embedding of von Neumann algebras $J^{\sharp *} \subset A^{\sharp *}$. In particular, if $J \subset A$ is an ideal, then the support projection $p$ for $J^{\sharp *}$ is central and gives rise to the decomposition
\[
A^{\sharp *} = J^{\sharp *} \oplus p^\perp A^{\sharp *}.
\]
If $J$ is not an ideal, but rather a hereditary ${\rm C}^*$-subalgebra, then $p$ is no longer central and so in this case we obtain the matrix decomposition
\[
A^{\sharp *} = 
\begin{pmatrix}
{J^{\sharp *}} & {p A^{\sharp *} p^\perp} \\
{p^\perp A^{\sharp *} p} & {p^\perp A^{\sharp *} p^\perp} 
\end{pmatrix}.
\]

If $M$ is a von Neumann algebra and $E$ is an operator $M$-system, then the Cauchy-Schwarz inequality easily shows that a positive linear functional $\varphi \in E^*$ is contained in $E^{M \sharp M}$ if and only if $\varphi_{| M}$ is normal. This observation can be used to give the following useful lemma, giving a criterion for when a c.p.\ map is continuous in the $M$-topology.

\begin{lem}\label{lem:weakcontinuity}
Let $M$ and $N$ be von Neumann algebras and let $E$ and $F$ be operator $M$ and $N$-systems, respectively. If $\phi: E \to F$ is completely positive such that the restriction of $\phi$ to $M$ defines a normal map from $M$ to $N$, then $\phi$ is a continuous map from $E$ with the weak $M$-topology, to $F$ with the weak $N$-topology. 
\end{lem}
\begin{proof}
To prove that $\phi$ is continuous from the weak $M$-topology to the weak $N$-topology, we need to check that if $\eta \in F^{N\sharp N}$, then $\eta \circ \phi \in E^{M \sharp M}$. Moreover, by viewing $F$ as an operator subsystem of $(F^{N \sharp N})^*$, we see that every linear functional $\eta \in F^{N\sharp N}$ is implemented by a vector linear functional in some $N$-bimodular c.p.\ representation of $F$ where $N$ is normally represented. Hence, by the polarization identity, we may write $\eta$ as a span of states in $F^{N\sharp N}$. Thus, it suffices to check that $\eta \circ \phi \in E^{M\sharp M}$ whenever $\eta \in F^{N \sharp N}$ is a state. If $\eta \in F^{N \sharp N}$ is a state, then $\eta_{| N}$ is normal and hence $(\eta \circ \phi)_{| M}$ is normal, from which it follows that $\eta \circ \phi \in E^{M \sharp M}$. 
\end{proof}

\begin{cor}\label{cor:condexpcont}
Using the notation above, 
	the map ${\Ad (e_N)}: \B(L^2M) \to \B(L^2N)$ is continuous from the weak $M$-topology to the weak $N$-topology, and also from the weak $M'$-topology to the weak $N'$-topology. 
\end{cor}

\subsection{The small-at-infinity boundary and boundary pieces}\label{subsec:smallboundary}

If $M \subset \B(\mathcal H)$ is a von Neumann algebra, then we will have occasion to use not only the $M$-topology, but also the $M'$-topology on $\B(\mathcal H)$. We therefore introduce the seminorms
$r_\omega$ on $\B(L^2M)$ given by 
\[
r_\omega(T) = \inf \{ ( \omega(J a^*a J) + \omega( b^*b ) )^{1/2} \| Z \| (\omega(J c^*c J) + \omega(d^*d) )^{1/2} \},
\] 
where $\omega$ is a normal state on $M$, $J$ is the modular conjugation operator, and the infimum is taken over all decompositions $T = \left( \begin{smallmatrix}
 a  \\
b
\end{smallmatrix}
\right)^*  Z \left( \begin{smallmatrix}
 c  \\ 
d
\end{smallmatrix} \right)$ where $a, c \in M'$, $b, d \in M$ and $Z \in \mathbb M_2(\B(L^2M))$. 
These seminorms on bounded sets describe the coarsest locally convex topological vector space topology that contains both the $M'$-topology and the $M$-topology.

 If $X$ is a left operator $M$-module (i.e., an operator $M$-$\mathbb C$-bimodule), then for a positive linear functional $\omega \in M^\sharp$ we denote by $s_\omega^\ell$ the seminorm $s_\omega^\rho$ where $\rho: \mathbb C \to \mathbb C$ is the identity map. Given a right operator $M_0$-module, we similarly denote by $s_\omega^r$ the seminorm $s_\rho^\omega$. We will also use similar notation $r_\omega^\ell$ and $r_\omega^r$ for the corresponding seminorms on $\B(L^2 M)$. For easy reference, we summarize these seminorms on an operator $T \in \B(L^2 M)$ here: 
 
 \[
 s_\omega(T) = \inf \{ \omega(a^*a) \| S \| \omega(b^*b) \mid T = a^* S b, a, b \in M, S \in \B(L^2M) \};
 \]
 
\[
s_\omega^\ell(T) = \inf \{ \omega(a^*a) \| S \| \mid T = a^* S, a \in M, S \in \B(L^2M) \};
\]

\[
s_\omega^r(T) = \inf \{ \| S \| \omega(b^*b) \mid T = Sb, b \in M, S \in \B(L^2M) \};
\]

\begin{align}
r_\omega(T) = \inf \{ ( \omega(J a^*a J) + & \omega( b^*b ) )^{1/2} \| Z \| (\omega(J c^*c J) + \omega(d^*d) )^{1/2} \nonumber \\
& \mid T = \left( \begin{smallmatrix}
 a  \\
b
\end{smallmatrix}
\right)^*  Z \left( \begin{smallmatrix}
 c  \\ 
d
\end{smallmatrix} \right), a, c \in M', b, d \in M, Z \in \mathbb M_2(\B(L^2M)) \}; \nonumber
\end{align}

\[
r_\omega^\ell(T) = \inf \{ ( \omega(J a^*a J) + \omega( b^*b ) )^{1/2} \| Z \| 
\mid T = \left( \begin{smallmatrix}
 a  \\
b
\end{smallmatrix}
\right)^*  Z, a \in M', b \in M, Z \in C_2(\B(L^2M)) \};
\]

\[
r_\omega^r(T) = \inf \{ \| Z \| (\omega(J c^*c J) + \omega(d^*d) )^{1/2} 
 \mid T =  Z \left( \begin{smallmatrix}
 c  \\ 
d
\end{smallmatrix} \right), c \in M', d \in M, Z \in R_2(\B(L^2M)) \}.
\]

In the case when $M$ is finite with a normal faithful trace $\tau$, the norm $r_\tau^r(T)$ is equivalent to the norm $\| T \|_{\infty, 2} = \sup_{a \in (M)_1} \| T \hat{a} \|$, which is nothing but the operator norm when thinking of $T$ as an operator from $M \subset L^2(M, \tau)$ with the uniform norm into $L^2(M, \tau)$ \cite{Oza10}. This result was generalized in \cite[Proposition 3.1]{DKEP22} to show that the norm $r_\tau(T)$ is equivalent to the norm $\| T \|_{\infty, 1} = \sup_{a, b \in (M)_1} | \langle T \hat{a}, \hat{b} \rangle |$ where we think of $T$ as an operator from $M \subset L^2(M, \tau)$ to $L^1(M, \tau) \supset L^2(M, \tau)$. 

The relationship between the seminorms is given by the following result.

\begin{prop}\label{prop:leftequality}
Let $M$ be a von Neumann algebra and $A$ an $M$-${\rm C}^*$-algebra. Then, for $\mu$ a normal positive linear functional and $x \in A$, we have $s_\mu^\ell(x^*) = s_\mu^r(x) = s_\mu(x^* x)^{1/2}$.  Also, if $T \in \B(L^2M)$, then we have $r_\mu^\ell(T^*) = r_\mu^r(T) = r_\mu(T^*T)^{1/2}$. 
\end{prop}
\begin{proof}
First note that the equality $s_\mu^\ell(x^*) = s_\mu^r(x)$ is obvious. Also, if $x = yb$ with $y \in A$ and $b \in M$, then $x^*x = b^* y^* y b$, and so $s_\mu(x^* x)^{1/2} \leq \mu(b^*b)^{1/2} \| y^* y \|^{1/2}$. Taking the infimum over all such decompositions gives $s_\mu(x^* x )^{1/2} \leq s_\mu^r(x)$. 

To see the reverse inequality, we suppose $x^*x = a^* y b$ for some $y \in A$, and $a, b \in M$ and set $\kappa = \mu(a^*a)^{1/2} \| y \| \mu(b^*b)^{1/2}$. By rescaling, we will assume that $\mu(a^*a) = \mu(b^*b)$. We then have 
\[
x^* x = \frac{1}{2} ( a^* y b + b^* y^* a) = \frac{1}{2} \left( \begin{smallmatrix}
 a  \\
b
\end{smallmatrix} \right)^*
\left( \begin{smallmatrix}
 0 &  y \\
y^* & 0
\end{smallmatrix} \right)
\left( \begin{smallmatrix}
 a  \\
b
\end{smallmatrix} \right),
\]
and note that $\frac{1}{2} \mu(a^*a + b^* b) = \mu(a^*a) = \mu(a^*a)^{1/2} \mu(b^*b)^{1/2}$. We consider the polar decomposition $\left( \begin{smallmatrix}
 a  \\
b
\end{smallmatrix} \right) = u (a^*a + b^*b)^{1/2}$, and we set $z = u^* \left( \begin{smallmatrix}
 0 &  y \\
y^* & 0
\end{smallmatrix} \right) u$ and $a_0 = \frac{1}{\sqrt{2}} (a^*a + b^*b)^{1/2}$ so that we have $x^*x = a_0 z a_0$ with $\mu(a_0^2) \| z \| \leq \kappa$. Note that by replacing $z$ with $p z p$ where $p$ is the support of $a_0 \in M$ we may assume also that $z \geq 0$. 

We then have $| x | = | z^{1/2} a_0 |$ and so if we consider the polar decompositions $x = v | x |$ and $z^{1/2} a_0 = w | z^{1/2} a_0 | = w | x |$, then we have $x = v | x | = v w^* z^{1/2} a_0$, and hence 
\[
s_\mu^r(x) \leq \| v w^* z^{1/2} \| \mu(a_0^2)^{1/2} \leq \kappa^{1/2}.
\] 
Taking the infimum over all such decompositions then gives $s_\mu^r(x) \leq s_\mu(x^* x)^{1/2}$. 

The result for the seminorms $r_\mu^\ell$, $r_\mu^r$, and $r_\mu$ follows similarly.
\end{proof}

\begin{cor}\label{cor:productcont}
Let $M$ be a von Neumann algebra, let $A$ be an $M$-${\rm C}^*$-algebra, and let $\mu$ be a normal positive linear functional on $M$. Then for $x, y, z \in A$ we have 
\[
s_\mu(x^* y z) \leq s_\mu(x^* x)^{1/2} \| y \| s_\mu(z^*z)^{1/2}.
\]
Also, if $x, y, z \in \B(L^2 M)$, then we have
\[
r_\mu(x^* y z) \leq r_\mu(x^* x)^{1/2} \| y \| r_\mu(z^*z)^{1/2}.
\]
\end{cor}
\begin{proof}
We clearly have 
$s_\mu(x^* y z) \leq s_\mu^\ell(x^*) \| y \| s_\mu^r(z)$, 
and so the result follows directly from Proposition~\ref{prop:leftequality}. The case for $r_\mu$ follows similarly. 
\end{proof}

Let $M$ be a von Neumann algebra. Generalizing the group setting from \cite{BIP21}, an $M$-boundary piece was defined in \cite{DKEP22} to be a hereditary ${\rm C}^*$-subalgebra $\X\subset \B(L^2M)$ 
	such that $M\cap M(\X)\subset M$ and $M'\cap M(\X)\subset M'$ are weakly dense.
Denote by 
\[
\K_\X^L(M)=\overline{\overline{\B(L^2M) \X }^{_{\C-M}}}^{_{\C-M'}},
\] 
which is a left ideal containing $M$ and $M'$ in its space of right multipliers. If $\mathcal H$ is a Hilbert space, then we set
\[
\K_\X^L(M, \mathcal H) = \overline{\overline{\B(L^2M, \mathcal H) \X }^{_{\C-M}}}^{_{\C-M'}}.
\]
	 A consequence of \cite[Proposition 2.2]{Ma98} (see \cite[Proposition 2.3]{DKEP22}) is that an operator $T \in \B(L^2 M)$ is contained in $\K_\X^L(M)$ if and only if there exist orthogonal families of projections $\{ f_i \}_{i \in I}, \{ \tilde f_j \}_{j \in J} \subset M$ such that $T J \tilde f_i J f_j \in \X$ for each $i \in I$, $j \in J$. 
	
Let 
\[
\K_\X(M)=(\K_\X^L(M))^* \cap \K_\X^L(M) = (\K_\X^L(M))^* \K_\X^L(M) \subset \B(L^2M)
\] 
be the hereditary ${\rm C}^*$-subalgebra associated to $\K_\X^L(M)$, and note that both $M$ and $M'$
	are in its multiplier algebra.
We also define 
\[
\K^{\infty,1}_\X(M)=\overline{\overline{\K_\X(M)}^{_{M-M}}}^{_{M'-M'}},
\] 
which agrees with the closure in the topology given by the seminorms $r_\mu$ defined above. 

By Lemma~\ref{lem:compressnorm} if $T \in \K^{\infty, 1}_\X(M)_+$, then there exists a net of positive contractions $z_i \in M$ converging ultrastrongly to $1$ so that in uniform norm we have $d(z_i T z_i, \X) \to 0$, and hence also $d(( z_i T z_i )^{1/2}, \X) \to 0$. By considering the polar decomposition we then have $d( T^{1/2} z_i, \B(L^2 M) \X ) \to 0$, and hence $T^{1/2} \in \K_\X^L(M)$, which shows that $T \in ( \K_\X^L(M) )^* \cap \K_\X^L(M) = \K_\X(M)$. Thus, we have the equality of positive cones 
\[
\K^{\infty, 1}_\X(M)_+ = \K_\X(M)_+.
\]
In particular, for a Hilbert space $\mathcal H$, an operator $T \in \B(L^2 M, \mathcal H)$ is contained in $\K_\X^L(M, \mathcal H)$ if and only if $| T | \in \K_\X^{\infty, 1}(M)$. 

The small-at-infinity boundary of $M$ with respect to the $M$-boundary piece $\X$ is defined as
\[
\bS_\X(M)=\{T\in \B(L^2M)\mid [T,x]\in \K_\X^{\infty,1}(M), \forall x\in M'\}.
\]
Since the operator space $\K_\X^{\infty, 1}(M)$ is typically not a ${\rm C}^*$-algebra, the space $\bS_\X(M)$ is also typically not a ${\rm C}^*$-algebra, though it is a normal operator $M$-system. 

In the case when $\X = \K(L^2M)$, we will denote $\K_\X^L(M)$, $\K_\X(M)$, $\K_\X^{\infty, 1}(M)$, and $\bS_\X(M)$ by $\K^L(M)$, $\K(M)$, $\K^{\infty, 1}(M)$, and $\bS(M)$, respectively.

One reason for the utility of this definition for the small-at-infinity boundary is that if $T \in \B(L^2M)$ and we consider the set 
\[
P = \{ x \in M' \mid [T, x] \in \K_\X^{\infty, 1}(M) \},
\]
then this set is always a von Neumann subalgebra of $M'$, and hence to show that $T \in \bS_\X(M)$ it suffices to check that $[T, x] \in \K_\X^{\infty, 1}(M)$ for $x$ in a set of generators for $M'$. When $M$ is a tracial von Neumann algebra this is Lemma 6.1 in \cite{DKEP22}, but the proof works in general. Indeed, it is easy to see that $P$ is a ${\rm C}^*$-algebra containing the unit, and if $x$ is in the ultrastrong$^*$ closure of $P$ and $\omega \in M_*$ is a positive linear functional, then taking a net $x_i \in P$ so that $x_i$ converges to $x$ in the ultrastrong$^*$-topology, we have $\omega( (x - x_i)^*(x - x_i) ), \omega( (x - x_i)(x - x_i)^*) \to 0$ and hence
\[
d_{r_\omega}([T, x], \K_\X^{\infty, 1}(M) )
\leq r_\omega( T (x - x_i ) ) + r_\omega( (x - x_i) T ) \to 0.
\]

Just as we could associate to any $M$-${\rm C}^*$-algebra $A$ a universal von Neumann algebra $A^{\sharp *}$ that contains $M$ as a von Neumann subalgebra, we can also associate to $\B(L^2 M)$ a von Neumann algebra $\B(L^2 M)_J^{\sharp *}$ containing both $M$ and $JMJ$ as von Neumann subalgebras. Specifically, we let $\B(L^2 M)_J^{\sharp *}$ be the corner $p_{\rm nor} q_{\rm nor} \B(L^2M)^{**} q_{\rm nor} p_{\rm nor}$, where $p_{\rm nor} \in M^{**}$ denotes the projection corresponding to the support of the identity representation $M \to M$, and $q_{\rm nor} \in (JMJ)^{**}$ denotes the projection corresponding to the support of the identity representation $JMJ \to JMJ$. 

If $\X \subset \B(L^2M)$ is a boundary piece, then we will also let $\K_\X(M)_J^{\sharp *}$ denote $p_{\rm nor} q_{\rm nor} \K_X(M)^{**} q_{\rm nor} p_{\rm nor}$ so that we have a natural identification $\K_\X(M)_J^{\sharp *} = q_{\X} \B(L^2M)_J^{\sharp *} q_{\X}$ where $q_\X \in \B(L^2M)_J^{\sharp *}$ denotes the support projection of $\K_\X(M)_J^{\sharp *}$.

\section{$M$-nuclearity}\label{sec:nuclearity}
Let $M$ be a von Neumann algebra, $F$ an operator system or ${\rm C}^*$-algebra, and $E$ an operator $M$-system or an $M$-${\rm C}^*$-algebra. 
We say a c.c.p.\ map $\phi: F\to E$ is $M$-nuclear 
	if there exist nets of c.c.p.\ maps $\phi_i:F\to\mathbb M_{n(i)}(\mathbb C)$ 
	and $\psi_i:\mathbb M_{n(i)}(\mathbb C)\to E$ 
	such that $\psi_i\circ\phi_i(x)$ converges to $\phi(x)$ in the $M$-topology for any $x\in F$.
By \cite[Theorem 3.7]{Ma00} and a standard convexity argument, this is equivalent to the existence of such maps $\phi_i$ and $\psi_i$ so that 
	$\psi_i\circ\phi_i(x)$ converges to $\phi(x)$ in the weak $M$-topology for any $x\in F$.

More generally, if $M_0\subset M$ is an ultraweakly dense ${\rm C}^*$-subalgebra and $E$ is an $M_0$-system or $M_0$-${\rm C}^*$-algebra,
	we then say a c.c.p.\ map $\phi: F\to E$ is $(M_0 \subset M)$-nuclear 
	if there exist nets of c.c.p.\ maps $\phi_i:F\to\mathbb M_{n(i)}(\mathbb C)$ 
	and $\psi_i:\mathbb M_{n(i)}(\mathbb C)\to E$ 
	such that $\psi_i\circ\phi_i(x)$ converges to $\phi(x)$ in the $(M_0 \subset M)$-topology for any $x\in F$. Equivalently, by Proposition~\ref{prop:A-continuous and topology}, there exist a net of c.c.p.\ maps $\phi_i:F\to\mathbb M_{n(i)}(\mathbb C)$ 
	and $\psi_i:\mathbb M_{n(i)}(\mathbb C)\to E$ 
	such that $\psi_i\circ\phi_i(x)$ converges to $\phi(x)$ in the weak $(M_0 \subset M)$-topology.
	
We remark that if $F \subset \B(\mathcal H)$, then in the definition of $(M_0 \subset M)$-nuclearity we may always take the map $\phi_i$ to be a compression onto a finite-dimensional subspace $\mathcal K \subset \mathcal H$, i.e., we may take $\phi_i$ to be of the form $F \ni x \mapsto P_{\mathcal K} x P_{\mathcal K} \in \B(L^2 \mathcal K) \cong \mathbb M_{\dim(\mathcal K)}(\mathbb C)$. This can be seen easily from the following general lemma.

\begin{lem}\label{lem:compressionapprox}
Let $F \subset \B(\mathcal H)$ be a ${\rm C}^*$-algebra or operator system and $\phi: F \to \mathbb M_k(\mathbb C)$ a c.c.\ (resp.\ c.p., u.c.p.) map. There exists a net of finite-dimensional subspaces $\mathcal K_i \subset \mathcal H$ and c.c.\ (resp.\ c.p., u.c.p.) maps $\phi_i: \B(\mathcal K_i) \to \mathbb M_k(\mathbb C)$ so that $( \phi_i \circ {\rm Ad}(P_{\mathcal K_i}) )_{|F}$ converges pointwise in norm to $\phi$. 
\end{lem} 
\begin{proof}
Let $\tilde \phi: \B(\mathcal H) \to \mathbb M_k(\mathbb C)$ be a c.c.\ extension of $\phi$, and let $\phi_i: \B(\mathcal H) \to \mathbb M_k(\mathbb C)$ be normal c.c.\ maps so that $\phi_i \to \tilde \phi$ in the point-norm topology. Since each $\phi_i$ is normal, if we consider the family of finite-dimensional subspaces of $\mathcal H$ to be a net ordered by inclusion, then for each $x \in F$ we have norm convergence $\phi_i(x) = \lim_{\mathcal K \to \infty} \phi_i \circ {\rm Ad}(P_{\mathcal K})(x)$, and hence $\phi(x) = \lim_{\mathcal K \to \infty} \lim_{i \to \infty} \phi_i \circ {\rm Ad}(P_{\mathcal K})(x)$. The proof in the c.p.\ case or the u.c.p.\ case, when $F$ is an operator system, is the same. 
\end{proof}

Similar to the usual notion of nuclear maps, in the unital case we may consider u.c.p.\ maps instead of c.c.p.\ maps.

\begin{lem}
Let $M$ be a von Neumann algebra with $M_0 \subset M$ an ultraweakly dense ${\rm C}^*$-subalgebra. Suppose $F$ is an operator system, $E$ is an $M_0$-system, and $\phi: F \to E$ is a u.c.p.\ $(M_0 \subset M)$-nuclear map. There exists a net of u.c.p.\ maps $\phi_i:F\to\mathbb M_{n(i)}(\mathbb C)$ 
	and $\psi_i:\mathbb M_{n(i)}(\mathbb C)\to E$ 
	such that $\psi_i\circ\phi_i(x)$ converges to $\phi(x)$ in the $(M_0 \subset M)$-topology for any $x\in F$	
\end{lem}
\begin{proof}
Let $\tilde \phi_i: F \to \mathbb M_{n(i)}(\mathbb C)$ and $\tilde \psi_i: \mathbb M_{n(i)}(\mathbb C) \to E$ be c.c.p.\ maps such that $\tilde \psi_i \circ \tilde \phi_i(x) - \phi(x) \to 0$ in the $(M_0 \subset M)$-topology for any $x \in F$. By Lemma 2.2.5 in \cite{BO08} there exist u.c.p.\ maps $\phi_i: F \to \mathbb M_{n(i)}(\mathbb C)$ so that $\tilde \phi_i(x) = \tilde \phi_i( 1)^{1/2} \phi_i(x) \tilde \phi_i(1)^{1/2}$. 

Note that $\tilde \psi_i (\tilde \phi_i(1))$ is a net of positive contractions that converge to $1$ in the $(M_0 \subset M)$-topology. By Lemma~\ref{lem:compressnorm} we may find a net of positive contractions $1 - z_i \in M_0$, so that $\| z_i (1 - \tilde \psi_i(\tilde \phi_i(1) ) )z_i \| \to 0$, and $z_i \to 1$ ultrastrongly in $M$. It therefore follows that $1 - z_i  \tilde \psi_i(\tilde \phi_i(1) ) z_i $ converges to $0$ in the $(M_0 \subset M)$-topology. Hence, if we fix states $\eta_i \in \mathbb M_{n(i)}(\mathbb C)$ and set \[
\psi_i(T) = z_i \tilde \psi_i( \tilde \phi_i(1)^{1/2} T \tilde \phi_i(1)^{1/2}) z_i +  (1 - z_i  \tilde \psi_i(\tilde \phi_i(1) ) z_i )\eta_i(T),
\] 
then $\psi_i$ are u.c.p.\ and we have $\psi_i \circ \phi_i$ converges to $\phi$ in the point-$(M_0 \subset M)$-topology. 
\end{proof}

We also remark that for a bimodular map between $M_0$ systems $E$ and $F$, the convergence in the definition of $(M_0 \subset M)$-nuclearity can be strengthened in the following sense.

\begin{lem}\label{lem:approxunifconv}
Let $M$ be a von Neumann algebra, $M_0 \subset M$ an ultraweakly dense ${\rm C}^*$-subalgebra. Suppose $M_0 \subset E$ and $M_0 \subset F$ are two $M_0$-systems, and $\phi: E \to F$ is $M_0$-bimodular and u.c.p. Then $\phi$ is $(M_0 \subset M)$-nuclear if and only if there exist nets of u.c.p.\ maps $\phi_i: E \to \mathbb M_{n(i)}(\mathbb C)$ and $\psi_i: \mathbb M_{n(i)}(\mathbb C) \to F$ so that $\psi_i \circ \phi_i$ converges in the point-$(M_0 \subset M)$-topology to $\phi$, and such that for each $x \in M_0$ there exist $T_i \in F$ and $a_i \in M_0$ so that $\psi_i \circ \phi_i(x) - x = T_i + a_i$ and we have
\begin{enumerate}
\item $\| T_i \| \to 0$. 
\item $\sup_i \| a_i \| <\infty$.
\item $a_i \to 0$ ultrastrongly. 
\end{enumerate}
\end{lem}
\begin{proof}
We let $\phi^0_i: E \to \mathbb M_{n(i)}(\mathbb C)$ and $\psi^0_i: \mathbb M_{n(i)}(\mathbb C) \to F$ denote u.c.p.\ maps so that $\psi^0_i \circ \phi_i^0$ converges to the identity in the point-$(M_0 \subset M)$-topology. Fix a finite set $E = \{ x_1, \ldots, x_m \} \subset M_0$ and a finite set of states $S = \{ \omega_1, \ldots, \omega_n \} \subset M_*$. 
Then, we may apply Lemma~\ref{lem:compressnorm} to $\{\psi_i^0\circ\phi_i^0(x_j)-x_j\}_i$ and $S$,	
	and we obtain positive contractions $\{1 - z_i\}_i\subset M_0$	
	such that $\lim_i\max_{1\leq k\leq n}\omega_k(1-z_i)= 0$ and $\lim_i \max_{1\leq j\leq m}\|z_i (\psi_i^0\circ\phi_i^0(x_j)-x_j)z_i\|=0$.

We now fix a state $\eta \in F^{ M_0 \sharp M_0 }$ and define the u.c.p.\ map $\psi_i: \mathbb M_{n(i)}(\mathbb C) \to F$ by $\psi_i(T) = z_i \psi_i^0(T) z_i + (1 - z_i^2) \eta(T)$. We then have $\psi_i \circ \phi_i^0(x_j) - x_j = z_i (\psi_i^0\circ\phi_i^0(x_j)-x_j ) z_i + (1 - z_i^2) \eta(\phi_i^0(x_j)) - ( x_j - z_i x_j z_i )$.

Note that $c_{i, j} := (1 - z_i^2) \eta(\phi_i^0(x_j)) - ( x_j - z_i x_j z_i ) \in M_0$ satisfies $\| c_{i, j} \| \leq 1 + 2 \| x_j \|$, and $\lim_{i \to \infty} \omega_k( c_{i, j}^* c_{i, j} ) = 0$ for each $1 \leq j \leq m$ and $1 \leq k \leq n$. Since $F$ and $S$ were arbitrary finite sets, the result then follows. 
\end{proof}

Nuclearity of a map with respect to the $(M_0 \subset M)$-topology can be reformulated as a weak nuclearity property by adapting the argument from \cite[Lemma 2.8(i)]{Kir95B} and using the appropriate normal bidual.

\begin{lem}\label{lem:weakly nuclear bidual}
Let $M_0$ be an ultraweakly dense ${\rm C}^*$-subalgebra of a von Neumann algebra $M$, and let $F$ be a normal $M_0$-system, or a normal $M_0$-${\rm C}^*$-algebra. If $E$ is an operator system or ${\rm C}^*$-algebra and $\theta: E \to F$ is a c.c.p.\ map, then the following conditions are equivalent:
\begin{enumerate}
\item\label{item:wknuc1} The c.c.p.\ map $\theta$ is $(M_0 \subset M)$-nuclear.
\item\label{item:wknuc2} The c.c.p.\ map $i_F \circ \theta: E \to (F^{M_0 \sharp M_0})^*$ is weakly nuclear.
\item\label{item:wknuc3} The c.c.p.\ map $i_F \circ \theta: E  \to (F^{M_0 \sharp M_0})^*$ is $(M_0 \subset M)$-nuclear.
\end{enumerate}
\end{lem}
\begin{proof}
Clearly (\ref{item:wknuc1}) $\implies$ (\ref{item:wknuc3}) $\implies$ (\ref{item:wknuc2}), so we need only show (\ref{item:wknuc2}) $\implies$ (\ref{item:wknuc1}). Notice that for each c.c.p.\ map $\phi: \mathbb M_n(\mathbb C)\to E^{\sharp *}$, there exists a net of c.c.p. maps $\phi_i: \mathbb M_n(\mathbb C)\to E$ such that $\phi_i\to \phi$ in the point-weak$^*$ topology.
Indeed, since $(\M_n(\C)\otimes E)^{* *}$ is completely isometrically isomorphic to $\M_n (\C)\otimes E^{**}$ \cite[Proposition B.16]{BO08},
	we have a completely isometrically isomorphism $(\M_n(\C)\otimes E)^{\sharp *}\cong\M_n(\C)\otimes E^{\sharp *}$.
Moreover, since c.p.\ maps from $\M_n(\C)$ to $E$ (resp.\ $E^{\sharp*}$) are one-to-one correspondent to elements in the positive cone of $\M_n(E)$ (resp.\ $\M_n( E^{\sharp*})$), 	
	it then follows from the density of $\M_n(E)\subset \M_n( E)^{\sharp*}$
	that we may approximate $\phi$ with c.c.p.\ maps $\phi_i':\M_n(\C)\to E$ in the point-weak$^*$ topology.
	Since the weak$^*$ topology on $i_E(E) \subset E^{\sharp *}$ agrees with the weak $(M_0 \subset M)$-topology on $E$, the result then follows. 
\end{proof}

\begin{lem}\label{lem:almostbimod}
Let $M$ be a von Neumann algebra and $M_0\subset M$ a weakly dense ${\rm C}^*$-subalgebra.
Let $E$ and $F$ be operator $M_0$-systems, and suppose $\theta_i: E \to F$ are u.c.p.\ maps such that ${\theta_i}_{|M_0}$ converges to the identity in the point-weak $(M_0 \subset M)$-topology, then for each $a, b \in M_0$ and $x \in E$ we have $\theta_i(axb) - a\theta_i(x) b \to 0$ in the weak $(M_0 \subset M)$-topology.
\end{lem}
\begin{proof}
Let $\theta: E \to (F^{_{M_0}\sharp_{M_0}})^{*}$ be any point-weak$^*$-limit point of $\{ \theta_i \}$. Since ${\theta_i}_{|M_0}$ converges to the identity in the point-weak $(M_0 \subset M)$-topology we have $\theta_{|M_0} = {\rm id}$, and since $\theta$ is u.c.p., we then have that $\theta$ is $M_0$-bimodular. Thus, for all $x \in F$ and $a, b \in M_0$ we have $\theta(axb) = a \theta(x) b$. Since $\theta$ was an arbitrary point-weak$^*$-limit point, the result follows. 
\end{proof}

The following theorem uses ideas as in the same spirit as in Section 2 of \cite{Kir95B}, Theorem 4.7 in \cite{EfOzRu01}, and Theorem 1 in \cite{Oza07}.

\begin{thm}\label{thm:nuclear inclusion density}
Let $M$ be a separable von Neumann algebra, $M_0\subset M$ a unital weakly dense ${\rm C}^*$-subalgebra, and $E$ a normal $M_0$-system. The following conditions are equivalent:
\begin{enumerate}
\item\label{item:nucinclusion1} The inclusion $M_0 \subset E$ is $(M_0 \subset M)$-nuclear.
\item\label{item:nucinclusion2}  There exists an amenable von Neumann algebra $R$, and normal u.c.p.\ maps $\Phi: M \to R$ and $\Psi: R \to (E^{M_0 \sharp M_0})^*$ so that $\Psi \circ \Phi(x) = i_E(x)$ for all $x \in M_0$.  
\end{enumerate}
\end{thm}

\begin{proof}
Since injectivity for a von Neumann algebra is equivalent to semi-discreteness, the implication (\ref{item:nucinclusion2}) $\implies$ (\ref{item:nucinclusion1}) then follows from Lemma~\ref{lem:weakly nuclear bidual}. We now suppose that (\ref{item:nucinclusion1}) holds. 
Since $M$ is separable, we may choose an increasing sequence of finite-dimensional operator systems $\{E_n\}_n$ inside $M_0$ so that $M_{00} := \cup E_n\subset M_0$ is ultraweakly dense.
By $(M_0 \subset M)$-nuclearity of the inclusion map $j_{M_0}:M_0\to E$, we have nets of u.c.p.\ maps $\phi_i: M_0\to \M_{k(i)}(\C)$
	and $\psi_i: \M_{k(i)}(\C)\to E$ so that $\psi_i\circ \phi_i \to j_{M_0}$ in the point-$(M_0 \subset M)$-topology.
By Arveson's Extension Theorem we may assume that each $\phi_i$ is defined on $E$. 

Denote by $\mu$ a normal faithful state on $M$ and $\mu_0$ the restriction of $\mu$ to $M_0$. Recall from Remark~\ref{rem:diagonalseminorm} that $s_{\mu_0}=s_{\mu_0}^{\mu_0}$.
For each $n\geq 1$, we inductively choose u.c.p.\ maps $\phi_n: E\to \M_{k(n)} (\C)$ and $\psi_n: \M_{k(n)}(\C)\to E$ so that 
	setting $\theta_n=\psi_n\circ \phi_n$,
	we have
\begin{equation}\label{eq:condition 1}
s_{\mu_0}(\theta_n(a)-a)< 2^{-n+3} \|a\|, a\in E_n,
\end{equation}
and 
\begin{equation}\label{eq:condition 2}
s_{\mu_0}(\theta_n\circ \cdots \circ \theta_{n-j}(\theta_{n-j-1}(a)-a))<2^{-n-j} \|a\|, a\in E_{n-j-1}, 0\leq j \leq n-2.
\end{equation}
Indeed, (\ref{eq:condition 1}) follows directly from $(M_0 \subset M)$-nuclearity. 
We now verify (\ref{eq:condition 2}) in the case of $j=0$, the general case follows similarly.
Notice that since $s_{\mu_0}(\theta_{n-1}(a)-a)<2^{-n+2}\|a\|$ for $a\in E_{n-1}$,
	we may find $a_1, a_2\in A$ and $b\in E$ such that $a=a_1^* b a_2$ and 
	$\mu(a_1^*a_1)^{1/2}\| b \| \mu(a_2^* a_2)^{1/2}< 2^{-n+1} \|a\|$.
Thus by Lemma~\ref{lem:almostbimod} and Proposition~\ref{prop:A-continuous and topology},
	we may, after passing to convex combinations, choose $\theta_n$ that satisfies $s_{\mu_0}(\theta_n( a_1^* b a_2)-a_1^*\theta_n(b)a_2)< 2^{-n-1} \|a\|$ for all $a\in E_{n-1}$,
	and it follows that 
\[
\begin{aligned}
s_{\mu_0}(\theta_n(\theta_{n-1}(a)-a))&\leq s_{\mu_0}(\theta_n( a_1^* b a_2)-a_1^*\theta_n(b)a_2)+ s_{\mu_0}(a_1^*\theta_n(b)a_2)\\
&\leq 2^{-n-1}\|a\|+ \mu(a_1^*a_1)^{1/2} \| b \| \mu(a_2^* a_2)^{1/2} \leq 2^{-n}\|a\|.
\end{aligned}
\]

Therefore, we have the following commutative diagram:
\[
\begin{tikzcd}
M_0 \arrow[dr, "\phi_1"] \arrow[rr, "\theta_1"]  & & 
E \arrow[dr, "\phi_2"] \arrow[rr, "\theta_2"]  & & 
E \arrow[rr, "\theta_3"] \arrow[dr, "\phi_3"]  & & 
\cdots & & \\
& \mathbb M_{k(1)}(\mathbb C) \arrow[ur, "\psi_1"] \arrow[rr, "\sigma_1"]& 
& \mathbb M_{k(2)}(\mathbb C) \arrow[ur, "\psi_2"]  \arrow[rr, "\sigma_2"]& 
& \mathbb M_{k(3)}(\mathbb C) \arrow[rr, "\sigma_3"] \arrow[ur, "\psi_3"]&
& \cdots &
\end{tikzcd}
\]
Let $S$ be the operator system inductive limit of $(\mathbb M_{k(n)}(\mathbb C), \sigma_n)$ in the sense of \cite[Section 2]{Kir95B}.
More precisely, if $\tilde S \subset \prod_{n\geq 1} \M_{k(n)}(\C)$ is the norm closure of 
\[
\left\{(x_n)_n\in \prod_{n\geq 1} \M_{k(n)}(\C)\mid{\rm there\ exists\ } m \geq 1 {\rm\ s.t.\ } x_{n+1} = \sigma_n(x_n){\rm \ for\ all\ }n\geq m \right\},
\]
	then $S=Q(\tilde S)$, where 
	$$Q: \prod_{n\geq 1} \M_{k(n)}(\C)\to \prod_{n\geq 1} \M_{k(n)}(\C)/ \sum_{n\geq 1}\M_{k(n)}(\C)$$ 
	is the quotient map.

It is easy to see that $\tilde S$ is nuclear.
As operator spaces, we may identify $S^{**} = p^\perp \tilde S^{**}$ where $p$ is the central projection in the bidual of $\prod_{n\geq 1} \M_{k(n)}(\C)$ given by the support of the ideal $\sum_{n\geq 1} \M_{k(n)}(\C)$.
We therefore have that $S^{**}$ is unitally completely order isomorphic to an amenable von Neumann algebra \cite[Theorem 4.5]{EfOzRu01}, \cite[Theorem 3.5]{HaPa11}.

Define $\Phi: E\to S^{**}$ to be a point-weak$^*$ limit point of the sequence
\[
\phi_m: E\ni x\mapsto \phi_m(x)\in S\subset S^{**},
\]
where we view $\phi_m(x)\in \M_{k(m)}(\C) \to S$ is the canonical map.

Let $\tilde \Psi: \tilde S\to (E^{M_0\sharp M_0})^*$ be the point-weak$^*$ limit point of $\tilde \psi_m:\tilde S\ni (x_n)_n \to \psi_m(x_m)\in E\subset (E^{M_0\sharp M_0})^*$.
	It is not hard to see that $\tilde \Psi = \Psi\circ Q$ for some u.c.p.\ map $\Psi: S\to (E^{M_0\sharp M_0})^*$.
We still denote by $\Psi$ the unique weak$^*$ continuous extension $\Psi: S^{**}\to (E^{M_0\sharp M_0})^*$.

Then for any $a\in M_{00}$ with norm $1$ and $\varphi\in (E^{M_0\sharp M_0})^*$, we have that
\[
\begin{aligned}
|\varphi(a-\Psi\circ \Phi(a))| & = |\varphi(a-\lim_m \lim_n\tilde \psi_n\circ\phi_m(a))|\\
&=\lim_m\lim_n|\varphi(a-(\theta_n\circ\cdots\theta_m)(a))| \\
& \leq \lim_m\lim_n\sum_{j=1}^{n-m} |\varphi((\theta_{n}\circ\cdots\circ\theta_{n-j+1} )(a-\theta_{n-j}(a)))|+ |\varphi(a-\theta_n(a))|.\\
\end{aligned}
\]
Since $\varphi\in E^{M_0\sharp M_0}$, there exists some $\kappa >0$ such that $|\varphi(x)|\leq\kappa s_{\mu_0}(x)$ for $x\in (E)_2$.
It then follows that for $m$ and $n$ large enough, we have
\[
|\varphi((\theta_{n}\circ\cdots\circ\theta_{n-j+1} )(a-\theta_{n-j}(a)))|\leq 2^{-n-j} \kappa,
\]
and hence 
\[
|\varphi(a-\Psi\circ \Phi(a))|\leq \lim_m \lim_n \sum_{j=0}^{n-m} 2^{-n-j}\kappa \leq \lim_m  2^{-m+1}\kappa =0.
\]
From the above argument, we obtain a u.c.p\ map 
	$\Phi_{M_{00}}: M_{00} \to S^{**}$
	and a weak$^*$-continuous u.c.p.\ map 
	$\Psi: S^{**}\to (E^{M_0\sharp M_0})^*$
	such that $\Psi\circ\Phi_{\mid{M_{00}}}=i_{M_{00}}$, where $i_{M_{00}}: M_{00} \to E^{\sharp *}$ is the canonical embedding.

Denote by $\bar \Phi: (M_{00})^{**}\to S^{**}$ the unique weak$^*$ continuous extension of $\Phi_{\mid M_{00}}$, and
 realize $M = p (M_{00})^{**}$, where $p\in (M_{00})^{**}$ is the support projection of the identity representation from $M_{00}$ into $M$. Let $\{e_i\}_i\subset M_{00}$ be a net that converges to $p$ in the weak$^*$ topology.
Then for any linear functional $\varphi\in E^{M_0\sharp M_0}$,
	we have $\varphi(\Psi\circ \bar\Phi(p))=\varphi(\Psi(\lim_i \Phi(e_i)))=\lim_i \varphi(e_i)=\varphi(1)$, and hence $\Psi \circ \bar\Phi(p) = 1$. Thus, ${\bar\Phi_{|M}}: M \to S^{**}$ is a normal u.c.p.\ map such that $\Psi \circ \bar\Phi(x) = i_E(x)$ for all $x \in M_0$. 
\end{proof}

To move beyond the separable case, we will use the following technical version of the previous theorem. The proof is essentially the same and so we will leave it to the reader.

\begin{thm}\label{thm:nonsepversion}
Let $M$ be a von Neumann algebra, $M_{0} \subset M$ an ultraweakly dense ${\rm C}^*$-subalgebra, and $E$ a normal $M_0$-system. If the inclusion $M_0 \subset E$ is $(M_0 \subset M)$-nuclear, then for every countably generated ${\rm C}^*$-subalgebra $M_{00} \subset M_0$, and $\sigma$-finite projection $p \in \mathcal P(M)$, there exists an amenable von Neumann algebra $R$, and normal c.c.p.\ maps $\Phi: M \to R$ and $\Psi: R \to (E^{M_0 \sharp M_0})^*$ so that $\Psi \circ \Phi(p) = p$ and $p \Psi \circ \Phi(x) = p i_{E}(x) p$ for all $x \in M_0$. 
\end{thm}

\begin{prop}\label{prop:densenuclear}
Let $M$ be a von Neumann algebra, $M_0, M_1 \subset M$ weakly dense ${\rm C}^*$-subalgebras with $M_0 \subset M_1$, and let $E$ be a normal $M_1$-system. Consider the following conditions:
\begin{enumerate}
\item\label{item:ninclusion1} The inclusion $M_0 \subset E$ is $(M_0 \subset M)$-nuclear.
\item\label{item:ninclusion2} The inclusion $M_1 \subset E$ is $(M_1 \subset M)$-nuclear. 
\item\label{item:ninclusion3} The inclusion $M_0 \subset E$ is $(M_1 \subset M)$-nuclear. 
\end{enumerate}
Then the implications (\ref{item:ninclusion1}) $\implies$ (\ref{item:ninclusion2}) $\implies$ (\ref{item:ninclusion3}) hold. 

Moreover, if $M_0$ is locally reflexive, then conditions (\ref{item:ninclusion2}) and (\ref{item:ninclusion3}) are equivalent. 
\end{prop}
\begin{proof}
If $M$ is separable, and if condition (\ref{item:ninclusion1}) holds, then by Theorem~\ref{thm:nuclear inclusion density} the canonical map $i_0: M \to (E^{M_0 \sharp M_0})^*$ is weakly nuclear. Since $M_0 \subset M_1$, the identity map ${\rm id}: E \to E$ induces a normal completely positive contraction $i_{1, 0}: (E^{M_0 \sharp M_0})^* \to (E^{M_1 \sharp M_1})^*$ so that the canonical embedding $i_1:M \to (E^{M_1 \sharp M_1})^*$ is given by the composition $i_1 = i_{1, 0} \circ i_0$. Hence, the map $i_1$ is weakly nuclear and it follows from Lemma~\ref{lem:weakly nuclear bidual} that condition (\ref{item:ninclusion2}) holds. 

For the general case, fix $x_1, \ldots, x_n \in M_1$, $\varepsilon > 0$, and $\eta_1, \ldots, \eta_m \in E^{M_1 \sharp M_1} \subset E^{M_0 \sharp M_0}$. Since each $\eta_i$ may be implemented by vector linear functional in some representation of $E$ where $M$ is normally represented, we may use the polarization identity to assume that each $\eta_i$ is a state. Let $p$ denote the supremum of the supports of ${\eta_i}_{|M}$, so that $p$ is a $\sigma$-finite. Since $M_0$ is ultraweakly dense in $M$ and since $p$ is $\sigma$-finite, we may choose for each $1 \leq i \leq n$ a sequence $a_{i, k} \in M_0$ so that $p a_{i, k} p \to p x_i p$ in the weak operator topology. 

We let $M_{00}$ denote the ${\rm C}^*$-algebra generated by $\{ a_{i, k} \}_{1 \leq i \leq n, k \geq 1}$, and by Theorem~\ref{thm:nonsepversion} there exists an amenable von Neumann algebra $R$, and normal c.c.p.\ maps $\Phi: M \to R$ and $\Psi: R \to (E^{M_0 \sharp M_0})^*$ so that $\Psi \circ \Phi(p) = p$ and $p \Psi \circ \Phi(x) = p i_{E}(x) p$ for all $x \in M_0$. Since $p$ is in the multiplicative domain of the normal c.c.p.\ map $\Psi \circ \Phi$, we then have, after compressing down to $(E^{M_1\sharp M_1})^*$, a weak operator topology limit 
\[
p \Psi \circ \Phi( x_i ) p = p \Psi \circ \Phi( p x_i p ) p = \lim_{k \to \infty} p \Psi \circ \Phi(  a_{i, k} ) p = p x_i p \in (E^{M_1\sharp M_1})^*.
\]
In particular, for each $1 \leq i \leq n$ and $1 \leq j \leq m$ we have $\eta_i( \Psi \circ \Phi(x_i)  - x_i ) = \eta_i( p ( \Psi \circ \Phi(x_i)  - x_i ) p) = 0$. Since $R$ is semi-discrete, there then exist c.c.p.\ maps $\phi_0: R \to \mathbb M_{d}(\mathbb C)$ and $\psi_0: \mathbb M_d(\mathbb C) \to R$ so that for each $1 \leq i \leq n$ and $1 \leq j \leq m$ we have $\eta_j( \Psi \circ \psi_0 \circ \phi_0 \circ \Phi(x_i) - x_i ) < \varepsilon$. Hence $M_1 \subset E$ is $(M_1 \subset M)$-nuclear.

The implication (\ref{item:ninclusion2}) $\implies$ (\ref{item:ninclusion3}) is trivial. Since weak nuclearity of the inclusion $i_1: M \to (E^{M_1 \sharp M_1})^*$ is a local property and since $i_1$ is normal, if $M_0$ is locally reflexive then we see that weak nuclearity of the inclusion $M_0$ into $(E^{M_1 \sharp M_1})^*$ implies that $i_1$ is weakly nuclear and hence another application of Lemma~\ref{lem:weakly nuclear bidual} shows that (\ref{item:ninclusion2}) and (\ref{item:ninclusion3}) are equivalent in this case.
\end{proof}

The following special case of Proposition~\ref{prop:densenuclear} will be used a number of times in the sequel, and so we explicitly state it as a corollary. 

\begin{cor}\label{cor:densenuclear}
Let $M$ be a von Neumann algebra with a dense ${\rm C}^*$-algebra $M_0 \subset M$. If $E$ is a normal $M$-system and if the inclusion $M_0 \subset E$ is $(M_0 \subset M)$-nuclear, then the inclusion $M \subset E$ is $M$-nuclear. 
\end{cor}

\begin{examp}\label{examp:exandwkex}
Let $A$ be a ${\rm C}^*$-algebra, and let $E$ be an $A$-system. Then the weak $(A \subset A^{**})$-topology on $E$ is simply the weak topology on $E$, and hence the inclusion $A \subset E$ is $(A \subset A^{**})$-nuclear if and only if the inclusion $A \subset E$ is nuclear. 

In particular, if we consider a universal representation $\pi: A \to \B(\mathcal H)$, then the inclusion $\pi: A \to \B(\mathcal H)$ is $(A \subset A^{**})$-nuclear if and only if $A$ is exact. However, we will see from Proposition~\ref{prop:wk exact unital} below that $A^{**}$ is weakly exact if and only if the identity map from $A^{**} = \pi(A)''$ to $\B(\mathcal H)$ is $A^{**}$-nuclear. Hence, the converse of Corollary~\ref{cor:densenuclear} does not hold in general (see, e.g., Remark 14.1.3 in \cite{BO08}).
\end{examp}

While the converse of Corollary~\ref{cor:densenuclear} does not hold in general, we do have the following partial converse, which should be compared with Corollary 3.1.6 in \cite{Iso13}. 

\begin{prop}\label{prop:sepnuc}
Let $M$ be a separable von Neumann algebra and $E$ a normal $M$-system.
If $M\subset E$ is $M$-nuclear, then for any separable ${\rm C}^*$-subalgebra $M_0\subset M$ 
	there exists a separable ${\rm C}^*$-subalgebra $M_1\subset M$ containing $M_0$
	such that $M_1\subset E$ is $(M_1\subset M)$-nuclear.
\end{prop}
\begin{proof}
Choose $\{a_n\}_n\subset M_0$ a countable norm dense subset. 
Since $M\subset E$ is $M$-nuclear, there exist sequences of u.c.p.\ maps $\phi_k: M_0\to \M_{n(k)}(\C)$ and $\psi_k: \M_{n(k)}(\C)\to E$
	such that $\psi_k\circ \phi_k(a_n)-a_n= x_{n,k} T_{n,k} y_{n,k}$
	with $x_{n,k}, y_{n,k}\in M_+$ and $T_{n,k}\in E$,
	such that $\lim_k s_{\rho}^\omega(x_{n,k} T_{n,k} y_{n,k})=0$ for any $\rho, \omega\in M_{*,+}$.

Set $A_1= C^*(M_0, \{x_{n,k}, y_{n,k}\mid n\in \N, k\in \N\})$ and note that $M_0\subset E$ is $(A_1\subset M)$-nuclear,
	implemented by $\phi_k$ and $\psi_k$.
Repeating this process inductively, we then have a sequence of separable ${\rm C}^*$-subalgebras $M_0\subset A_1\subset A_2\cdots$,
	such that $A_n\subset E$ is $(A_{n+1}\subset M)$-nuclear.
Let $M_1=\overline{\cup_n A_n}^{\|\cdot\|}$ and one checks that $M_1\subset E$ is $(M_1\subset M)$-nuclear.
\end{proof}

We end this section with a result showing the relationship between $M$-nuclearity and topologically amenable actions. To describe this we recall that if $E \subset \B(\mathcal H)$ is an operator space and $I$ is a set, then the operator space $\mathbb M_I(E) \subset \mathcal B(\mathcal H \ovt \ell^2 I)$ is the operator system that we may identify as those $I \times I$ matrices with entries in $E$ that correspond to bounded operators. This operator space is independent of the representation $E \subset \B(\mathcal H)$. Moreover, if $\Gamma \actson^\sigma E$ by completely isometric isomorphisms, then we may choose a covariant representation $E \subset \B(\mathcal H)$ so that the equation $\sigma_t(x) = \pi_t  x \pi_t^*$ holds for $x \in E$ and $t \in \Gamma$ for some unitary representation $\pi$. Then conjugation by $\pi_t \otimes \rho_t$ gives an action of $\Gamma$ on $\mathbb M_\Gamma(E)$, and we let $\mathbb M_\Gamma(E)^\Gamma$ denote the fixed points under this action. This is then a normal operator $L\Gamma$-bimodule, and if $E$ is an operator system, then $\mathbb M_\Gamma(E)^\Gamma$ will be a normal operator $L\Gamma$-system. Note however that $\mathbb M_\Gamma(E)^\Gamma$ need not be a ${\rm C}^*$-algebra, even when $E$ is a ${\rm C}^*$-algebra. 

\begin{thm}\label{thm:amenaction}
Let $\Gamma$ be a group and suppose $\Gamma \actson K$ is an action by homeomorphisms on a compact Hausdorff space. The action $\Gamma \actson K$ is topologically amenable if and only if the inclusion $L\Gamma \subset \mathbb M_{\Gamma}(C(K))^\Gamma$ is $L\Gamma$-nuclear. 
\end{thm}
\begin{proof}
Note that by \cite[Section 4.1]{BO08} we have an operator space embedding $C(K) \rtimes_r \Gamma \subset \mathbb M_\Gamma(C(K))^\Gamma$ that takes $C^*_\lambda \Gamma$ canonically onto $C^*_\lambda \Gamma \subset L\Gamma$ and takes $C(K)$ to diagonal matrices $C(K) \ni f \mapsto \oplus_{t \in \Gamma} \sigma_{t^{-1}}(f) \in \mathbb M_\Gamma(C(K))^\Gamma$, where $\sigma: \Gamma \to {\rm Aut}(C(K))$ is the corresponding action $\sigma_s(f) = f \circ s^{-1}$. If the action $\Gamma \actson K$ is amenable, then the inclusion $C^*_\lambda \Gamma \subset C(K) \rtimes_r \Gamma$ is nuclear \cite[Section 4.3]{BO08}, and hence by Corollary~\ref{cor:densenuclear} the inclusion $L\Gamma \subset  \mathbb M_{\Gamma}(C(K))^\Gamma$ is $L\Gamma$-nuclear.

For the converse, note that we have a $\Gamma$-equivariant conditional expectation onto the diagonal matrices $E_{\ell^\infty \Gamma}: \mathbb M_\Gamma( C(K) )^\Gamma \to \ell^\infty(\Gamma; C(K) )^\Gamma$ where the later space we may identify, in a $\Gamma$-equivariant way, with $C(K)$ by evaluating at the identity. Also note that $E_{\ell^\infty \Gamma}$ restricts to the trace on $L\Gamma$, and so if $a, b \in L\Gamma$, $T \in \mathbb M_\Gamma( C(K) )^\Gamma$, and $\varphi \in C(K)^*$ is a state, then we have $| \varphi \circ E_{\ell^\infty \Gamma}( a^* T b ) | \leq \| T \|  (\varphi \circ E_{\ell^\infty \Gamma}(a^*a) )^{1/2} ( \varphi \circ E_{\ell^\infty \Gamma}(b^*b) )^{1/2} = \| T \| \tau( a^* a)^{1/2} \tau(b^* b)^{1/2}$. It therefore follows that $E_{\ell^\infty \Gamma}: \mathbb M_\Gamma( C(K) )^\Gamma \to C(K)$ is continuous from the $L\Gamma$-topology into the norm topology. 

If $L\Gamma \subset \mathbb M_\Gamma( C(K) )^\Gamma$ is $L\Gamma$-nuclear, $F \subset \Gamma$ is finite, and $\varepsilon > 0$, then using Lemma~\ref{lem:compressionapprox} and the observations above there exists a finite subset $E \subset \Gamma$ and a u.c.p.\ map $\psi: \mathbb M_E(\mathbb C) \to \mathbb M_\Gamma( C(K) )^\Gamma$ so that setting $h(t) = E_{\ell^\infty \Gamma}( \psi \circ {\rm Ad}(P_{\ell^2 E}) (\lambda_t) \lambda_t^* ) \in C(K)$ we have $\| 1 - h(t) \|_\infty < \varepsilon$ for each $t \in F$. The function $\Gamma \ni t \mapsto h(t) \in C(K)$ is easily seen to be of positive type (see the implication (4) $\implies$ (2) in \cite[Theorem 4.4.3]{BO08}), and hence it follows that the action $\Gamma \actson K$ is amenable. 
\end{proof}

\section{Weakly exact von Neumann algebras}\label{sec:wkexact}

Recall from \cite[Definition 3.1.1]{Iso13} that for a von Neumann algebra $M$ and an ultraweakly dense ${\rm C}^*$-algebra $M_0 \subset M$, the ${\rm C}^*$-algebra $M_0$ is said to be weakly exact in $M$ if for any unital ${\rm C}^*$-algebra $B$ with a closed two-sided ideal $J$ and 
	any representation $\pi: M_0\otimes B\to \B(\cK)$ with $M_0\otimes J\subset \ker \pi$ that is ultraweakly continuous on $M\otimes \C$,
	the induced representation $\tilde\pi:M_0\odot B/J\to\mathbb B(\cH)$ is min-continuous.
This notion generalizes weak exactness for von Neumann algebras and exactness for ${\rm C}^*$-algebras simultaneously: 
	if $M_0=M$ is a von Neumann algebra, 
	then $M$ is weakly exact in $M$ if and only if $M$ is weakly exact in the sense of Kirchberg (\cite{Kir95}, see also \cite[Chapter 14]{BO08});
	at the other extreme, if $M = (M_0)^{**}$, then $M_0$ is weakly exact in $(M_0)^{**}$ if and only if $M_0$ is exact.

In this section, we characterize weak exactness via $(M_0 \subset M)$-nuclearity, 
	analogous to the case of exactness for ${\rm C}^*$-algebras.

\begin{thm}\label{thm:wk exact}
Let $M \subset \B(\mathcal H)$ be a von Neumann algebra and $M_0 \subset M$ an ultraweakly dense ${\rm C}^*$-subalgebra.  
The following conditions are equivalent.
\begin{enumerate}
\item The ${\rm C}^*$-algebra $M_0$ is weakly exact in $M$. \label{item:wk-exact}
\item There exist nets of c.c.\ maps $\phi_i: M_0\to \M_{n(i)}(\C)$ and
	$\psi_i: \M_{n(i)}(\C)\to \B(\cH)$ such that $\psi_i\circ \phi_i(a)\to a$ 
	in the $(M_0\subset M)$-topology for any $a\in M_0$.\label{item:cb-nuclear inclusion}
\item The inclusion $M_0\subset \B(\cH)$ is $(M_0\subset M)$-nuclear, i.e., there exist nets of c.c.p.\ maps $\phi_i: M_0\to \M_{n(i)}(\C)$ and
	$\psi_i: \M_{n(i)}(\C)\to \B(\cH)$ such that $\psi_i\circ \phi_i(a)\to a$ 
	in the $(M_0\subset M)$-topology for any $a\in M_0$.\label{item:nuclear inclusion}
\end{enumerate}
\end{thm}

Before proceeding to the proof, we collect a few lemmas.

\begin{lem}\label{lem:exact kernel}
Let $M$ be a von Neumann algebra and $M_0 \subset M$ an ultraweakly dense ${\rm C}^*$-subalgebra. 
Then $M_0$ is weakly exact in $M$ if and only if for any closed two-sided ideal $J\lhd B$ and any $\pi: M_0\otimes B\to \B(\cK)$ such that $\pi_{\mid M_0\otimes \C}$ is ultraweakly continuous and $M_0\otimes J\subset \ker \pi$,
	we have $\ker q\subset \ker \pi$, where $q$ is the quotient map $q:M_0\otimes B\to M_0\otimes B/J$.
\end{lem}
\begin{proof}
It is clear that $\ker q\subset \ker \pi$ implies that $\tilde \pi: A\odot B/J\to \B(\cK)$ is min-continuous.
Conversely, suppose there is some $x\in \ker q\setminus \ker \pi$.
	Then, we may find a sequence $x_n\in A\odot B$ such that $x_n\to x$ in norm.
Note that $\|q(x_n)\|\to 0$ while $\liminf_n\|\pi(x_n)\|\geq \delta$ for some $\delta>0$.
However, since $x_n\in A\odot B$ and $A\otimes J\subset \ker \pi$, 
	we have $\tilde \pi(q(x_n))=\pi(x_n)$, which is a contradiction.
\end{proof}

\begin{lem}\label{lem:unitization exact}
Let $M$ be a von Neumann algebra and $M_0 \subset M$ an ultraweakly dense ${\rm C}^*$-subalgebra with $1 \not\in M_0$.  Let $\tilde M_0=M_0+\C 1 \subset M$.
If $M_0$ is weakly exact in $M$, then $\tilde M_0$ is also weakly exact in $M$.
\end{lem}
\begin{proof}
Let $\pi:\tilde M_0\otimes B\to \B(\cK)$ and $J\lhd B$ be given as in the definition of weak exactness.
Denote by $q:\tilde M_0\otimes B\to \tilde M_0\otimes B/J$ the quotient map.
By Lemma~\ref{lem:exact kernel}, it suffices to show that $\ker q\subset \ker \pi$.
Let $x\in \ker q$. Consider the following diagram:
\[
\begin{tikzcd}
& 0 \arrow[d] & 0 \arrow[d] & 0 \arrow[d] & \\
0\arrow[r] & M_0\otimes J\arrow[r] \arrow[d] & M_0\otimes B \arrow[r] \arrow[d] & M_0\otimes B/J\arrow[r]\arrow[d] & 0\\
0\arrow[r] & \tilde M_0 \otimes J\arrow[r] \arrow[d] & \tilde M_0 \otimes B \arrow[r] \arrow[d] & \tilde M_0\otimes B/J\arrow[r] \arrow[d] & 0\\
0\arrow[r] & \C \otimes J\arrow[r] \arrow[d] & \C \otimes B \arrow[r] \arrow[d] & \C \otimes B/J\arrow[r] \arrow[d] & 0\\
& 0  & 0  & 0. &\\
\end{tikzcd}
\]
The $3\times 3$ lemma shows that we may find $y\in \tilde M_0\otimes J$
	such that $x-y\in \ker q_{\mid M_0\otimes B}$.
Since $M_0\subset \B(\cH)$ is weakly exact, we have $\ker q_{\mid M_0\otimes B}\subset \ker \pi_{\mid M_0\otimes B}$,
	and it follows that $x\in \ker \pi$.
\end{proof}

\begin{prop}\label{prop:approx and exact}
Let $M \subset \B(\mathcal H)$ be a von Neumann algebra and $M_0 \subset M$ an ultraweakly dense ${\rm C}^*$-subalgebra. 
Then, $M_0$ is weakly exact in $M$ if and only if
there exist nets of c.c.\ maps $\phi_i: M_0\to \M_{n(i)}(\C)$ and
	$\psi_i: \M_{n(i)}(\C)\to \B(\cH)$ such that $\psi_i\circ \phi_i(a)\to a$ 
	in the $(M_0\subset M)$-topology for any $a\in M_0$.
\end{prop}
\begin{proof}
The ``only if'' direction is a direct consequence of \cite[Theorem 3.1.3, (1), (v)]{Iso13}, Wittstock's extension theorem, 
	and the fact that the restriction of the weak $(M_0\subset M)$-topology on $M_0$ coincides with 
	the ultraweak topology on $M_0$ endowed from $M$.
By a standard convexity argument we obtain that the convergence is in the 
	$(M_0\subset M)$-topology instead of the weak $(M_0\subset M)$-topology.

For the ``if'' direction, we follow the idea of \cite[Proposition 3.7.8]{BO08}.
Let $\pi$ and $J\lhd B$ be given as in the definition.
Denote by $q: M_0\otimes B\to M_0\otimes B/J$ the quotient map, and it suffices to show $\ker q\subset \ker \pi$.

Let $x\in \ker q$ be a contraction. Consider the following diagram:
\[
\begin{tikzcd}
0\arrow[r] & M_0\otimes J\arrow[r] \arrow[d, "\phi_i\otimes \id"] & M_0\otimes B \arrow[r] \arrow[d, "\phi_i\otimes \id"] & M_0\otimes B/J\arrow[r]\arrow[d, "\phi_i\otimes \id"] & 0\\
0\arrow[r] & D_i \otimes J\arrow[r] \arrow[d, "\psi_i\otimes \id"] & D_i \otimes B \arrow[r] \arrow[d, "\psi_i\otimes \id"] & D_i\otimes B/J\arrow[r] \arrow[d, "\psi_i\otimes \id"] & 0\\
0\arrow[r] & \B(\HH) \otimes J\arrow[r] & \B(\HH) \otimes B \arrow[r] & \B(\HH) \otimes B/J\arrow[r] & 0,
\end{tikzcd}
\]
where $D_i=\M_{n(i)}(\C)$.
Since the middle row is exact, $(\phi_i\otimes\id)(x)\in D_i\otimes J$. Applying $\psi_i\otimes\id$ gives $x_i:=\big((\psi_i\circ\phi_i)\otimes\id\big)(x)\in \B(\cH)\otimes J$.

Let $\{e_n\}\in J$ be an approximate unit. 
Fix a unit vector $\xi\in \cH$ and $\varepsilon>0$.
By Lemma~\ref{lem:compressnorm}, we may find contractions $z_i\in (M_0)_+$
	such that $\| (z_i \otimes 1 )(x_i-x)(z_i \otimes 1)\|\to 0$ and $\langle\pi((1-z_i) \otimes 1)\xi, \xi\rangle \to 0$.
We may then choose $e_n$ and $z_j$ such that 
\[
\|(z_j \otimes 1)(x_j-x)(z_j \otimes 1)\|<\varepsilon/8, \ \ \ \ \ \langle \pi((1 - z_j) \otimes 1)\xi, \xi\rangle<\varepsilon/8
\] 
and
\[
\| (1 \otimes 1 -1\otimes e_n)((z_j \otimes 1) x_j (z_j \otimes 1))\|<\varepsilon/8.
\]

Note that $(1\otimes e_n)(z_j \otimes 1) x (z_j \otimes 1)\in M_0\otimes J$ and 
\[
\|(1\otimes e_n)(z_j \otimes 1) x (z_j \otimes 1)- (z_j \otimes 1) x (z_j \otimes 1)\|<\varepsilon/2.
\]
It follows that 
$$|\langle\pi(x) \xi, \xi\rangle| =|\langle \pi(( (1-z_j) \otimes 1)x+(z_j \otimes 1) x( (1- z_j) \otimes 1)+ (z_j \otimes 1) x (z_j \otimes 1))\xi, \xi\rangle|\leq \varepsilon.$$
Since $\xi$ and $\varepsilon$ are both arbitrary, by polarization we conclude $x\in \ker \pi$.
\end{proof}

We next argue that in the unital case, one may take $\phi_i$ and $\psi_i$ in the above proposition to be u.c.p.\ maps.

\begin{prop}\label{prop:wk exact unital}
Let $M \subset \B(\mathcal H)$ be a von Neumann algebra and $M_0 \subset M$ an ultraweakly dense ${\rm C}^*$-subalgebra with $1 \in M_0$. 
Then, $M_0$ is weakly exact in $M$ if and only if the inclusion $M_0\subset \B(\cH)$ is $(M_0\subset M)$-nuclear.
\end{prop}
\begin{proof}
The ``if'' direction follows from Proposition~\ref{prop:approx and exact}.

To see the converse, take c.c.\ maps $\psi_i$ and $\phi_i$ as in Proposition~\ref{prop:approx and exact}.
First notice that by Lemma~\ref{lem:compressionapprox} we may assume $\phi_i=\Ad(P_{\cH_i})$, 
	where $P_{\cH_i}$ is the orthogonal projection onto a finite-dimensional subspace $\cH_i\subset \cH$,
	by replacing $\{\psi_i\}$ with another net of complete contractions.
One then checks that $\psi_i\circ \phi_i\to \id_{M_0}$ in the point-$(M_0\subset M)$-topology after the replacement.

For any finite set of states $F\subset M_0^\sharp$,
	since $\psi_i\circ \phi_i(1)=\psi_i(1_{\M_{n(i)}(\C)})\to 1$ in the $(M_0\subset M)$-topology,
	by Lemma~\ref{lem:perturb} we obtain a net of u.c.p.\ maps $\psi_{i,F}: \M_{n(i)}(\C)\to \B(\cH)$ such that for any $a\in M_0$,
	$\lim_i s_\omega^\rho(\psi_{i,F}\circ \phi_i(a)-a)\leq \lim_i s_\omega^\rho \big((\psi_i-\psi_{i,F})(\phi_i(a))\big)
	+\lim_i s_\omega^\rho(\psi_i\circ \phi_i(a)-a)=0$, for any $\omega,\rho\in F$.
By ordering the collection of finite subsets of states in $M_0^\sharp$ by inclusion,  we may then
produce nets of u.c.p.\ maps that show the inclusion $M_0 \subset \B(\mathcal H)$ is $(M_0 \subset M)$-nuclear. 
\end{proof}

\begin{proof}[Proof of Theorem~\ref{thm:wk exact}]
Proposition~\ref{prop:approx and exact} gives (\ref{item:wk-exact})$\iff$(\ref{item:cb-nuclear inclusion}). Clearly (\ref{item:nuclear inclusion})$\implies$(\ref{item:cb-nuclear inclusion}).	
	
To see (\ref{item:wk-exact})$\implies$(\ref{item:nuclear inclusion}),
	note that by Lemma~\ref{lem:unitization exact} we have $\tilde M_0=M_0+\C 1$ is weakly exact in $M$ 
	and hence Proposition~\ref{prop:wk exact unital} shows the inclusion $\tilde M_0 \subset \B(\mathcal H)$ is $(\tilde M_0 \subset M)$-nuclear.

By Lemma~\ref{lem:approxunifconv} there then exist nets of u.c.p.\ maps $\phi_i: \tilde M_0 \to \mathbb M_{n(i)}(\mathbb C)$ and $\psi_i : \mathbb M_{n(i)}(\mathbb C) \to \B(\mathcal H)$ so that for each $x \in M_0$ there exist $T_i \in \B(\mathcal H)$ and $a_i \in \tilde M_0$ with $\psi_i \circ \phi_i(x) - x = T_i + a_i$ where $\| T_i \| \to 0$, and $\{ a_i \}_i$ is uniformly bounded with $a_i \to 0$ ultrastrongly. 

We fix an approximate unit $\{ e_n \}_n \subset M_0$ and notice that for a fixed $i$ and $x$ we have $\Ad( e_n) \circ \psi_i \circ \phi_i(x) - x$ is asymptotically close in the uniform norm to $e_nT_i e_n + e_n a_i e_n$. Since we have an ultrastrong limit $\lim_{n \to \infty} e_n a_i e_n = a_i$ and since $\Ad( e_n) \circ \psi_i: \mathbb M_{n(i)}(\mathbb C) \to M_0$ are complete contractions, this then shows that $M_0 \subset \B(\mathcal H)$ is $(M_0 \subset M)$-nuclear. 
\end{proof}

In particular, we obtain the following characterization of weakly exact von Neumann algebras.

\begin{cor}\label{cor:characterization weakly exact}
Let $M \subset \B(\mathcal H)$ be a von Neumann algebra.
Then $M$ is weakly exact if and only if the inclusion $M \subset \B(\mathcal H)$ is $M$-nuclear. 
\end{cor}

\begin{cor}\label{cor:wknuclear}
Let $M$ be a von Neumann algebra and $M_0 \subset M$ a weakly dense ${\rm C}^*$-subalgebra such that $M_0$ is weakly exact in $M$. Suppose $E$ is a dual $M_0$-system. The inclusion $M_0 \subset E$ is $(M_0 \subset M)$-nuclear if and only if the inclusion $M_0 \subset E$ is weakly nuclear. 
\end{cor}
\begin{proof}
Suppose $M_0 \subset E$ is weakly nuclear, then let $\phi_i: M_0 \to \mathbb M_{n(i)}(\mathbb C)$ and $\psi_i: \mathbb M_{n(i)}(\mathbb C) \to E$ denote the c.c.p.\ maps realizing weak nuclearity for the inclusion. By Arveson's extension theorem, we extend $\phi_i$ to c.c.p.\ maps $\tilde \phi_i: \B(L^2M) \to \mathbb M_{n(i)}(\mathbb C)$ and let $\Phi: \B(L^2M) \to E$ denote a point-ultraweak limit of the maps $\psi_i \circ \tilde \phi_i$. Then $\Phi$ is $M_0$-bimodular, and hence is continuous with respect to the $(M_0 \subset M)$-topologies. Since $M_0$ is weakly exact in $M$, we have that the inclusion $M_0 \subset \B(L^2M)$ is $(M_0 \subset M)$-nuclear, and composing the corresponding maps with $\Phi$ shows then that the inclusion $M_0 \subset E$ is $(M_0 \subset M)$-nuclear. 
\end{proof}

The following result answers Problem 10.4.3 from \cite{BO08}.

\begin{thm}\label{thm:brownozawa}
Let $M$ be a von Neumann algebra, then $M$ is weakly exact if and only if for every normal faithful representation $M \subset \B(\mathcal H)$, and for any intermediate von Neumann algebra $M \subset N \subset \B(\mathcal H)$ such that there exists an $M$-bimodular u.c.p.\ map $\phi: \B(\mathcal H) \to N$, then the inclusion $M \subset N$ is weakly nuclear. 
\end{thm}
\begin{proof}
We first show that if $M$ is not weakly exact, then we may construct a faithful normal representation $M\subset\B(\mathcal H)$, an intermediate von Neumann algebra $M\subset N\subset\B(\mathcal H)$, and an $M$-bimodular u.c.p.\ map $\phi:\B(\mathcal H)\to N$ such that $M\subset N$ is not weakly nuclear.

Let $M\subset\B(\mathcal K)$ be a faithful normal representation. By Corollary~\ref{cor:characterization weakly exact}, the inclusion $M\subset\B(\mathcal K)$ is not $M$-nuclear. Apply Lemma~\ref{lem:weakly nuclear bidual} with $M_0=E=M$, $F=\B(\mathcal K)$, and $\theta$ the original inclusion, and it follows that $M\subset\B(\mathcal K)^{\sharp *}$ is not weakly nuclear.

Let $\B(\mathcal K)^{\sharp *}\subset\B(\mathcal H)$ be a faithful normal representation. With $\B(\mathcal K)$ identified with its canonical operator-system image in $\B(\mathcal K)^{\sharp *}$, we have
\[
M\subset\B(\mathcal K)\subset\B(\mathcal K)^{\sharp *}\subset\B(\mathcal H),
\]
where the induced representation of $M$ on $\mathcal H$ is faithful and normal.

By injectivity of $\B(\mathcal K)$, the identity map on $\B(\mathcal K)$ extends to a u.c.p.\ map $\phi:\B(\mathcal H)\to\B(\mathcal K)\subset\B(\mathcal K)^{\sharp *}=:N$, which is the identity on $M$ and hence $M$-bimodular.

The converse is Exercise 14.1.4 in \cite{BO08}. Or, just notice that when $M$ is weakly exact we have that $M \subset \B(\mathcal H)$ is $M$-nuclear by Corollary~\ref{cor:characterization weakly exact}, and since $M$-bimodular u.c.p.\ maps are continuous in the $M$-topology, if we are given an $M$-bimodular u.c.p.\ map $\phi: \B(\mathcal H) \to N$, it follows that the inclusion $M \subset N$ is $M$-nuclear and hence is also weakly nuclear.
\end{proof}

\subsection{Weak exactness and free products}\label{sec:freeproductwkexact}

\begin{lem}\label{lem:statepreserve}
Let $M \subset \B(\mathcal H)$ be a von Neumann algebra and let $M_0 \subset M$ be an ultraweakly dense unital ${\rm C}^*$-subalgebra.  Suppose $E \subset \B(\mathcal H)$ is a normal $M_0$-system such that $\K(\mathcal H) \subset E$. If $M_0 \subset E$ is $(M_0 \subset M)$-nuclear, then for each vector state $\varphi(\cdot) = \langle \cdot \xi, \xi \rangle$ there exist nets of u.c.p.\ maps $\phi_i: M_0 \to \mathbb M_{n(i)}(\mathbb C)$ and $\psi_i: \mathbb M_{n(i)}(\mathbb C) \to E$, and pure states $\mu_i$ on $\mathbb M_{n(i)}(\mathbb C)$ such that $\mu_i \circ \phi_i = \varphi$ and $\varphi \circ \psi_i = \mu_i$, and for each $x \in M_0$ we have $\psi_i \circ \phi_i(x) - x = T_i + a_i$, where $T_i\in E$, $\| T_i \| \to 0$, and $a_i \in M_0$ is uniformly bounded with $\varphi(a_i^*a_i) \to 0$.
\end{lem}
\begin{proof}
Let $\mathcal H_1 = \mathcal H \ovt \ell^2 \mathbb N$ and let $E_1 = E \otimes \mathbb C + \K(\mathcal H_1) \subset \B(\mathcal H_1)$. Then, the inclusion map $M_0 = M_0 \otimes \mathbb C \subset E_1$ is also $(M_0 \subset M)$-nuclear and we also have $M_0 \cap \K(\mathcal H_1) = \{ 0 \}$.

Note that $E_1$ is a $\K(\mathcal H_1)$-system, and we have $(E_1^{M_0 \sharp M_0})^* = (\K(\mathcal H_1)^{M_0 \sharp M_0} )^* \oplus q_{\K(\mathcal H_1)}^\perp (E_1^{M_0 \sharp M_0})^*$ where $q_{\K(\mathcal H_1)} \in E_1^{**}$ is the support projection for $\K(\mathcal H_1)$. Since $M_0 \subset E_1$ is $(M_0 \subset M)$-nuclear, we have that the inclusion map from $M_0$ into $(E_1^{M_0 \sharp M_0})^*$ is weakly nuclear. Since $(\K(\mathcal H_1)^{M_0 \sharp M_0} )^* = \K(\mathcal H_1)^{**} \cong \B(\mathcal H_1)$ and since $M_0 \cap \K(\mathcal H_1) = \{ 0 \}$, it then follows that the inclusion map from $M_0 + \K(\mathcal H_1)$ into $(E_1^{M_0 \sharp M_0})^*$ is also weakly nuclear, and hence the inclusion $M_0 + \K(\mathcal H_1) \subset E_1$ is $(M_0 \subset M)$-nuclear. 

By Lemmas~\ref{lem:compressionapprox} and \ref{lem:approxunifconv}, there exist finite-dimensional subspaces $\mathcal K_i \subset \mathcal H_1$, and u.c.p.\ maps $\psi_i: \B(\mathcal K_i) \to E_1$ so that setting $\phi_i(x) = P_{\mathcal K_i} x P_{\mathcal K_i}$ we have that $\psi_i \circ \phi_i$ converges to the identity in the point-$(M_0 \subset M)$-topology on $M_0 + \K(\mathcal H_1)$, and such that for each $x \in M_0$ we have $\psi_i \circ \phi_i(x) - x = T_i + a_i$, where $\| T_i \| \to 0$, and $a_i \in M_0$ is uniformly bounded with $\varphi(a_i^*a_i) \to 0$. If we fix $\xi \in \mathcal H$, then we may assume that $\xi_1 := \xi \otimes \delta_1 \in \mathcal K_i$ and we define a pure state $\mu_i$ on $\B(\mathcal K_i)$ to be the vector state associated to $\xi_1$. Note that we have $\mu_i \circ \phi_i = \varphi$. 

We let $P_{\xi_1}$ denote the rank-one projection onto $\mathbb C\xi_1$, and note that we have $\psi_i(P_{\xi_1})-P_{\xi_1}\to0$ in the $(M_0\subset M)$-topology. Hence, as in the proof of Lemma~\ref{lem:compressnorm}, there exists a net of positive contractions $z_i\in M_0$ such that $\|z_i(\psi_i(P_{\xi_1})-P_{\xi_1})z_i\|\to0$ and $z_i\to1$ ultrastrongly. Since $z_i\to1$ ultrastrongly, we have that $\|z_iP_{\xi_1}z_i-P_{\xi_1}\|\to0$. Choose a state $\eta_i$ on $\B(\mathcal K_i)$ such that $\eta_i(P_{\xi_1})=0$, and define $\psi_i'(T)=z_i\psi_i(T)z_i+(1-z_i^2)\eta_i(T)$. Then we have $\|\psi_i'(P_{\xi_1})-P_{\xi_1}\|\to0$, and $\psi_i'\circ\phi_i$ still satisfies the pointwise convergence described above. Indeed, for $x\in M_0$, we have $\psi_i'\phi_i(x)-x=T_i'+a_i'$, where
\[
T_i'=z_iT_iz_i,\qquad a_i'=z_ia_iz_i+z_ixz_i-x+(1-z_i^2)\eta_i(\phi_i(x)).
\]
Writing $B_x=\sup_i\|a_i\|$, we have $\|T_i'\|\to0$ and $\|a_i'\|\leq B_x+3\|x\|$. Moreover,
\[
\|a_i'\xi_1\|\leq\|a_i\xi_1\|+(B_x+3\|x\|)\|(1-z_i)\xi_1\|+\|(1-z_i)x\xi_1\|\longrightarrow0,
\]
and hence $\varphi((a_i')^*a_i')\to0$.

We claim that there exist u.c.p.\ maps $\psi_i'':\mathbb M_{n(i)}(\mathbb C)\to E_1$ such that $\psi_i''(P_{\xi_1})=P_{\xi_1}$ and $\|\psi_i''-\psi_i'\|_{\mathrm{cb}}\to0$. In fact, take a Stinespring representation $\psi_i'(T)=V_i^*\pi_i(T)V_i$, where $V_i$ is an isometry, and set
\[
W_i=\pi_i(P_{\xi_1})V_iP_{\xi_1}+(1-\pi_i(P_{\xi_1}))V_i(1-P_{\xi_1}).
\]
As in the proof of \cite[Proposition 2.2.7]{BO08}, define
\[
\psi_i''(T)=W_i^*\pi_i(T)W_i+\mu_i(T)P_{\xi_1}(1-\psi_i'(P_{\xi_1}))P_{\xi_1}+\eta_i(T)(1-P_{\xi_1})\psi_i'(P_{\xi_1})(1-P_{\xi_1}).
\]
Since
\[
1-W_i^*W_i=P_{\xi_1}(1-\psi_i'(P_{\xi_1}))P_{\xi_1}+(1-P_{\xi_1})\psi_i'(P_{\xi_1})(1-P_{\xi_1}),
\]
the map $\psi_i''$ is u.c.p.\ and satisfies $\psi_i''(P_{\xi_1})=P_{\xi_1}$. Expanding the first term shows that its range lies in $E_1$, since $E_1$ is an operator $\K(\mathcal H_1)$-system. As in the proof of Lemma~\ref{lem:perturb}, we obtain
\[
\|\psi_i''-\psi_i'\|_{\mathrm{cb}}\leq 2\|V_i-W_i\|+\|1-W_i^*W_i\|\leq 2\|\psi_i'(P_{\xi_1})-P_{\xi_1}\|^{1/2}+\|\psi_i'(P_{\xi_1})-P_{\xi_1}\|\longrightarrow0.
\]

Since $P_{\xi_1}$ is then in the multiplicative domain of $\psi_i''$, we have that $\varphi \circ \psi_i''(T) =  P_{\xi_1} \psi_i''(T) P_{\xi_1} = \psi_i''(P_{\xi_1} T P_{\xi_1}) = \mu_i(T)$ for all $T \in \mathbb M_{n(i)}(\mathbb C)$. Recall that $\mu_i$ is a pure state on $\B(\mathcal K_i)$. We may then compose $\psi_i''$ with the compression map given by ${\rm id} \otimes {\rm Ad } (P_{\delta_1} ): E_1 \to E$ (which is $M$-bimodular) to finish the proof. 
\end{proof}

In order to set notation, we briefly recall the free product construction. We refer the reader to \cite{VoDyNi92} for a more detailed discussion. If $(\mathcal H_i, \xi_i)$ are Hilbert spaces with distinguished unit vectors, for $i = 1, 2$, then we set $\mathcal H_i^o = \mathcal H_i \ominus \mathbb C \xi_i$. The Hilbert space free product is $(\mathcal H, \xi)$ given by
\[
\mathcal H = \mathbb C \xi \oplus \bigoplus_{n \geq 1} \left( \bigoplus_{\iota_1 \not= \iota_2 \not= \cdots \not= \iota_n} \mathcal H_{\iota_1}^o \ovt \cdots \ovt \mathcal H_{\iota_n}^o \right). 
\]
We also consider
\[
\mathcal H(i) =  \mathbb C \xi \oplus \bigoplus_{n \geq 1} \left( \bigoplus_{\iota_1 \not= \iota_2 \not= \cdots \not= \iota_n  \atop \iota_1 \not= i} \mathcal H_{\iota_1}^o \ovt \cdots \ovt \mathcal H_{\iota_n}^o \right), 
\]
and note that we have a canonical isomorphism $\mathcal H \cong \mathcal H_i \ovt \mathcal H(i)$. Via this isomorphism we obtain a representation $\lambda: \B(\mathcal H_i) \to \B(\mathcal H)$. 
We may similarly consider the Hilbert space
\[
\mathcal H(r, i) =  \mathbb C \xi \oplus \bigoplus_{n \geq 1} \left( \bigoplus_{\iota_1 \not= \iota_2 \not= \cdots \not= \iota_n  \atop \iota_n \not= i} \mathcal H_{\iota_1}^o \ovt \cdots \ovt \mathcal H_{\iota_n}^o \right), 
\]
and note that we again have a canonical isomorphism $\mathcal H \cong  \mathcal H(r, i) \ovt \mathcal H_i$, so that we obtain another representation $\rho: \B(\mathcal H_i) \to \B(\mathcal H)$. 

We let $\omega_i$ denote the vector state on $\B(\mathcal H_i)$ given by $\xi_i$. If $A_i \subset \B(\mathcal H_i)$ are ${\rm C}^*$-algebras, then the reduced free product ${\rm C}^*$-algebra $(A_1, \omega_1) *_r (A_2, \omega_2)$ is the ${\rm C}^*$-subalgebra of $\B(\mathcal H)$ generated by $\lambda_1(A_1)$ and $\lambda_2(A_2)$. If $M_i \subset \B(\mathcal H_i)$ are von Neumann algebras, then the free product von Neumann algebra $(M_1, \omega_1) * (M_2, \omega_2)$ is the von Neumann algebra generate by $\lambda(M_1)$ and $\lambda(M_2)$. The representations $\lambda_i$ give rise to a representation of the unital algebraic free product $\B(\mathcal H_1) *_{\mathbb C} \B(\mathcal H_1)$ and if $E_i \subset \B(\mathcal H_i)$ are operator systems, then we denote by $(E_1, \omega_1) *_r (E_2, \omega_2)$ the operator subsystem of $(\B(\mathcal H_1), \omega_1) *_r (\B(\mathcal H_2), \omega_2)$ generated as the closed span of applying the representation $\lambda$ to all reduced words that appear in $E_1$ and $E_2$. 

If $A_i, B_i \subset \B(\mathcal H_i)$ are ${\rm C}^*$-subalgebras, and we have u.c.p.\ maps $\phi_i: A_i \to B_i$ that preserve the state $\omega_i$, then there is a unique u.c.p.\ map $(\phi_1 * \phi_2): A_1 *_r A_2 \to B_1 *_r B_2$ such that $(\phi_1 * \phi_2)(a_1 b_1 a_2 \cdots a_n b_n) = \phi_1(a_1) \phi_2(b_1) \cdots \phi_1(a_n) \phi_2(b_n)$, whenever $a_i \in A_1$ with $\omega_1(a_i) = 0$ for $i > 1$, and $b_i \in A_2$ with $\omega_2(b_i) = 0$ for $i < n$. 

The following is adapted from \cite[Lemma 2.4]{Oza06}.

\begin{thm}\label{thm:freeproductnuclear}
For each $i = 1, 2$, let $M_i \subset \B(\mathcal H_i)$ be a von Neumann algebra with cyclic vector $\xi_i \in \mathcal H_i$, and suppose $A_i \subset M_i$ is an ultraweakly dense unital ${\rm C}^*$-algebra. Suppose $E_i \subset \B(\mathcal H_i)$ is an operator $A_i$-system such that $\K(\mathcal H_i) \subset E_i$. 
If for each $i = 1, 2$, the inclusion $A_i \subset E_i$ is $(A_i \subset M_i)$-nuclear, then the inclusion $(A_1, \omega_1) *_r (A_2, \omega_2) \subset (E_1, \omega_1) *_r (E_2, \omega_2)$ is $( (A_1, \omega_1) *_r (A_2, \omega_2)  \subset (M_1, \omega_1) * (M_2, \omega_2))$-nuclear.
\end{thm}
\begin{proof}
By Lemma~\ref{lem:statepreserve} for each $i=1,2$, there exist two nets of state preserving u.c.p.\ maps $\phi_k^i: (M_i, \omega_i)\to (\M_{n(k, i)}(\C), \mu_k^i)$,
	and $\psi_k^i: (\M_{n(k, i)}(\C), \mu_k^i)\to E_i$
	such that if we set $\theta_k^i = \psi_k^i\circ \phi_k^i$, then for any $a \in A_i$, $\theta_k^i(a)-a$ is of the form 
\begin{equation}\label{eq:convergenceform}
\theta_k^i(a)-a = T^i_k + a_k
\end{equation} 
with $a_k\in A_i$ satisfying $\omega_i(a_k^* a_k) \to 0$, and such that $T^i_k\in E_i$ is converging to $0$ in norm,
	where $\mu_k^i$ is a pure state on $\M_{n(k, i)}(\C)$.

We consider 
$$\phi_{(j, k)}: \phi_j^1\ast \phi_k^2: (A_1, \omega_1)\ast_r (A_2, \omega_2)\to (\M_{n(j, 1)}(\C), \mu_j^1)\ast_r (\M_{n(k, 2)}(\C), \mu_k^2)$$
and 
$$\psi_{(j, k)}:\psi_j^1\ast\psi_k^2: (\M_{n(k, 1)}(\C), \mu_j^1)\ast_r (\M_{n(k, 2)}(\C), \mu_k^2)\to (E_1, \omega_1)\ast_r (E_2,  \omega_2).$$
Notice that for $x=a_1 b_1 a_2 b_2\cdots a_n b_n\in A_1\ast_r A_2$, with $a_i\in A_1$ with $\omega_1(a_i) = 0$ for $i > 1$, and $b_i\in A_2$ with $\omega_2(b_i) = 0$ for $i < m$,
	we have $\psi_{(j, k)}\circ \phi_{(j, k)}(x)\to x$ in the point-weak-$(M_0 \subset M)$-topology, where $M_0=A_1\ast_r A_2$ and $M=M_1\ast M_2$.
Indeed, for any $\varphi\in (E_1 *_r E_2)^{_{M_0}\sharp _{M_0}}$, 
it follows from (\ref{eq:convergenceform}) that for each $1 \leq m \leq n$ we have
\[
|\varphi(\theta_j^1(a_1)\theta_k^2(b_1)\cdots \theta_j^1(a_m)(\theta^2_k(b_m)-b_m )a_{m+1}\cdots a_n b_n)| \to 0,
\]
and similarly we also have 
\[
|\varphi(\theta_j^1(a_1)\theta_k^2(b_1)\cdots \theta^2_{k}(b_{m-1}) ( \theta_j^1(a_m) - a_m) b_m\cdots a_n b_n)| \to 0.
\]

Taking closed spans then shows that $\theta_{(j, k)}$ converges pointwise in the weak-$(M_0 \subset M)$-topology. Since $\mu_j^1$ and $\mu_k^2$ are pure, $(\M_{n(j, 1)}(\C), \mu_j^1)\ast_r (\M_{n(k, 2)}(\C),\mu_k^2)$ is nuclear by \cite[Theorem 1.1]{Oza01}, and the result follows. 
\end{proof}

We note that the assumption that each $E_i$ contains the compact operators is necessary even in the case when each $M_i$ is finite-dimensional, since otherwise by taking $E_i = M_i$ the previous result would falsely claim that reduced free products of finite-dimensional ${\rm C}^*$-algebras are nuclear. 

We recall that a state on a ${\rm C}^*$-algebra is said to be nondegenerate if the corresponding GNS-representation is faithful. Some special cases of the following corollary were previously obtained by Isono \cite{Iso13}.

\begin{cor}\label{cor:free product weak exact}
For $i = 1, 2$, let $M_i$ be von Neumann algebras with nondegenerate normal states $\omega_i$. 
If each $M_i$ is weakly exact, 
then $M=(M_1, \omega_1) * (M_2, \omega_2)$ is weakly exact. 
\end{cor}
\begin{proof}
Using Corollary~\ref{cor:characterization weakly exact}, we may set $E_i = \B(L^2(M_i, \omega_i))$ and then apply Theorem~\ref{thm:freeproductnuclear} and Corollary~\ref{cor:densenuclear}. 
\end{proof}

Note that an increasing union of weakly exact von Neumann algebras is again weakly exact. When the union is separable, this is \cite[Proposition 4.1.2]{Iso13}. The general case also follows from \cite[Proposition 4.1.2]{Iso13}, but one needs to then apply Corollary~\ref{cor:densenuclear} and use Corollary~\ref{cor:characterization weakly exact}. It therefore follows that the free product of an arbitrary family of weakly exact von Neumann algebras with nondegenerate states is again weakly exact.

\section{Biexact von Neumann algebras}

We recall from Section~\ref{subsec:smallboundary} that if $M$ is a von Neumann algebra and $\X \subset \B(L^2M)$ is a boundary piece, then the small-at-infinity boundary relative to $\X$ is the normal operator $M$-system
\[
\bS_\X(M)=\{T\in \B(L^2M)\mid [T,x]\in \K_\X^{\infty,1}(M), \forall x\in M'\}.
\]

\begin{defn}
Let $M$ be a von Neumann algebra with an $M$-boundary piece $\X\subset\B(L^2M)$. 
We say $M$ is {\it biexact relative to} $\X$ if the inclusion $M \subset \bS_{\X}(M)$ is $M$-nuclear.
\end{defn}

Given a discrete group $\Gamma$, a boundary piece $I$ is a $\Gamma\times \Gamma$ invariant closed ideal such that $c_0\Gamma\subset I\subset \ell^\infty\Gamma$ \cite{BIP21}.  
The small at infinity compactification of $\Gamma$ relative to $I$ is the spectrum of the ${\rm C}^*$-algebra $\bS_I(\Gamma)=\{f\in\ell^\infty\Gamma\mid f-R_tf\in I, {\rm\ for\ any\ } t\in\Gamma\}$.
Recall that $\Gamma$ is biexact relative to $I$ if $\Gamma \actson \bS_I(\Gamma)/I$ is topologically amenable 
	\cite{Oza04, BO08, BIP21}\footnote{Actually, we are abusing terminology somewhat here. In \cite[Definition 15.1.2]{BO08} biexactness is defined via an a priori more restrictive formulation, which is only shown to be equivalent to this formulation for countable groups.}.
We remark that amenability of $\Gamma\actson\bS_I(\Gamma)/I$ is equivalent to amenability of $\Gamma\actson\bS_I(\Gamma)$.\footnote{Following \cite[p.~411, footnote~3]{BO08}, we regard the empty $\Gamma$-space as amenable if and only if $\Gamma$ is exact.} Indeed, let $\{e_i\}_i$ be a positive contractive approximate identity for $I$, with weak$^*$ limit $p$ in $(\ell^\infty\Gamma)^{**}$. The c.c.p.\ maps $f\mapsto e_i f$ converge point-weak$^*$ to the unital $*$-homomorphism $\phi:\ell^\infty\Gamma\to I^{**}$ given by $\phi(f)=pf$. Since $I$ is $\Gamma$-invariant, $p$ is fixed by $\Gamma$, and hence $\phi$ is equivariant.

If $\Gamma\actson\bS_I(\Gamma)/I$ is amenable, then $\Gamma$ is exact. Hence $\Gamma\actson I^{**}$ is amenable by \cite[Proposition 2.7(i)]{BEW20}. By \cite[Proposition 2.7]{BEW20}, the action $\Gamma\actson(\bS_I(\Gamma)/I)^{**}$ is also amenable. The equivariant decomposition
\[
\bS_I(\Gamma)^{**}\cong I^{**}\oplus(\bS_I(\Gamma)/I)^{**}
\]
therefore shows that $\Gamma\actson\bS_I(\Gamma)^{**}$ is amenable. Another application of \cite[Proposition 2.7]{BEW20} gives amenability of $\Gamma\actson\bS_I(\Gamma)$. The converse follows by restriction to the closed invariant subset given by the Gelfand spectrum of $\bS_I(\Gamma)/I$.

\begin{thm}\label{thm:iff}
	Let $\Gamma$ be a discrete group with a boundary piece $I$, and $\X$ an $L\Gamma$-boundary piece. 
	Let $I_\X\subset \ell^\infty\Gamma$ denote the closed ideal generated by $E(\X)$, where $E: \B(\ell^2\Gamma)\to \ell^\infty\Gamma$ is canonical conditional expectation,
	and let $\X_I$ denote the $L\Gamma$-boundary piece $\overline{I\B(\ell^2\Gamma)I}^{\|\cdot\|}$ generated by $I$.
	The following statements are true:
\begin{enumerate}
\item \label{item:iff 1} If $L\Gamma$ is biexact relative to $\X$, then $\Gamma$ is biexact relative to $I_\X$.
\item \label{item:iff 2} $\Gamma$ is biexact relative to $I$ if and only if $L\Gamma$ is biexact relative to $\X_I$. 
\end{enumerate}
\end{thm}
\begin{proof}
(\ref{item:iff 1})
Suppose $L\Gamma$ is biexact relative to $\X$.
By the argument in \cite[Theorem 6.4]{DKEP22}, we have $E: \bS_\X(L\Gamma)\to \bS_{I_\X}(\Gamma)$ where $E$ is the canonical conditional expectation from $\B(\ell^2 \Gamma)$ to the diagonal subalgebra $\ell^\infty \Gamma$.
Now for each $n\in\mathbb N$, consider $h_n:\Gamma\to \bS_{I_\X}(\Gamma)$ given by $h_n(\gamma)=E(\theta_n(\lambda_\gamma)\lambda_\gamma^*)$, where $\theta_n = \psi_n \circ \phi_n:L\Gamma\to\bS_\X(L\Gamma)$ are u.c.p.\ maps that factor through matrix algebras coming from the $M$-nuclear embedding.
Note that we may assume $\theta_n(\lambda_s)=0$ for all but finitely many $s\in\Gamma$, 
	since by Lemma~\ref{lem:compressionapprox} we may take the first u.c.p.\ map $\phi_n$ from $M$-nuclearity to be a compression down to $\B(\ell^2 E_n)$ for some finite subsets $E_n \subset \Gamma$. Indeed, if $s\notin E_nE_n^{-1}$, then $sE_n\cap E_n=\varnothing$, so $P_{E_n}\lambda_sP_{E_n}=0$ and hence $\theta_n(\lambda_s)=0$. 	
	Viewing $h_n$ as an element in $C_c(\Gamma, \bS_{I_\X}(\Gamma))\subset \bS_{I_\X}(\Gamma)\rtimes_r\Gamma$, we now check that $h_n$ is positive.
For a finite subset $\{s_i\}_{i=1}^m\subset \Gamma$, observe that 
\begin{align}
[ \lambda_{s_i}^{-1} h_n(s_i s_j^{-1})]_{i,j}&=[E(\lambda_{s_i} \theta_n(\lambda_{s_i s_j^{-1}}) \lambda_{s_j^{-1}})]_{i,j} \nonumber \\
&=E( \diag(\lambda_{s_1}, \dots, \lambda_{s_m})^* \theta_n^{(m)}([\lambda_{s_i}\lambda_{s_j}^*]_{i,j})\diag(\lambda_{s_1}, \dots, \lambda_{s_m})) \geq 0. \nonumber
\end{align}
The same argument as in the proof of Theorem~\ref{thm:amenaction} shows that $E$ is continuous from the $L\Gamma$-topology into the norm topology, and so for any $s\in \Gamma$, we have $\|h_n(s)-1\|_{\bS_{I_\X}(\Gamma)}=\|E((\lambda_s-\theta_n(\lambda_s))\lambda_s^*)\| \to 0$. We therefore conclude that $\Gamma\actson \bS_{I_\X}(\Gamma)$ is amenable \cite[Theorem 4.4.3]{BO08}.

(\ref{item:iff 2})
As $\bS_I(\Gamma)\subset\bS_{\X_I}(L\Gamma)$, we have an embedding $\iota:\bS_I(\Gamma)\rtimes_r \Gamma \to \bS_{\X_I}(L\Gamma)$ given by $\iota(f)=M_f$ and $\iota(u_\gamma)=\lambda_\gamma$ for any $f\in \bS_I(\Gamma)$ and $\gamma\in\Gamma$ by \cite[Proposition 5.1.3]{BO08}.
Since $\bS_I(\Gamma)\rtimes_r \Gamma$ is nuclear, in particular, the inclusion 
	$$C^*_\lambda\Gamma\subset \bS_I(\Gamma)\rtimes_r \Gamma\subset \bS_{\X_I}(L\Gamma)$$
is $(C^*_\lambda\Gamma\subset L\Gamma)$-nuclear. 
Corollary~\ref{cor:densenuclear} then gives that $L\Gamma\subset \bS_{\X_I}(L\Gamma)$ is $L\Gamma$-nuclear.

The ``if'' part follows from the first statement upon noticing that $E(\X_I)=I$.
\end{proof}

Taking $I=c_0\Gamma$, for which $\X_I=\K(\ell^2\Gamma)$, yields the following corollary.

\begin{cor}\label{cor:w*-equivalence}
Let $\Gamma$ and $\Lambda$ be two discrete groups. If $\Gamma$ is biexact and $L\Gamma\cong L\Lambda$, then $\Lambda$ is also biexact.
\end{cor}

\subsection{Properties of biexact von Neumann algebras}\label{sec:properties of biexact}
Next we collect some basic properties of biexact von Neumann algebras, generalizing some results from the group von Neumann algebra setting.

Given $M$ a von Neumann algebra with a normal faithful representation $\pi: M\to \B(\cH)$, 
	we may consider the following normal $M$-system, which is a natural variant of $\bS(M)$,
	\[
	\bS(M; \cH)=\{T\in \B(\cH)\mid [T,x]\in \K(M;\cH),\forall x\in M'\},
	\]
where $\K(M;\cH)$ is the $M$-$M$ and $M'$-$M'$ closure of $\K(\cH)$.
We say $M$ is biexact with respect to the normal faithful representation $\pi: M\to \B(\cH)$ if the inclusion $M\subset \bS(M;\cH)$ is $M$-nuclear.
Since every normal faithful representation of $M$ is a reduction of an amplification of the standard form of $M$,
	the next lemma shows that biexactness is independent of representations of $M$.

\begin{lem}\label{lem:representation independent}
Let $M\subset \B(\cH)$ be a von Neumann algebra.
\begin{enumerate}
\item If $M\subset \bS(M; \cH)$ is $M$-nuclear, then $M \subset \bS(M; \cH\otimes \ell^2S)$ is $M$-nuclear for any set $S$.\label{item:representation}
\item If $M\subset \bS(M; \cH)$ is $M$-nuclear, $p \in \mathcal P(M)$, and $q \in \mathcal P(M')$ with $pq \not= 0$, then $p M p q \subset  \bS(pMpq; pq\cH)$ is $pMpq$-nuclear.\label{item:cut down}
\end{enumerate}
\end{lem}
\begin{proof}
(\ref{item:representation}) 	It suffices to show $\bS(M; \cH)\otimes \id_{\ell^2S}\subset \bS(M;\cH\otimes\ell^2S)$.
For any $T\in \bS(M; \cH)$, $x\in M' \subset \B(\cH)$ and $i,j\in S$, we denote by $e_{i,j}=\langle \cdot, \delta_i\rangle \delta_j$
	and compute $[T\otimes \id_{\ell^2 S}, x\otimes e_{i,j}]=[T,x]\otimes e_{i,j}\in \K(M; \cH\otimes \ell^2S)$.
Since the span of $\{x\otimes e_{i,j}\mid x\in M', i,j\in S\}\subset M'\cap \B(\cH\otimes \ell^2S)$ is ultraweakly dense,
	we conclude that $\bS(M; \cH)\otimes \id_{\ell^2S}\subset \bS(M;\cH\otimes\ell^2S)$.

(\ref{item:cut down}) If $S \in \B(\cH)$ and $a, b \in M$, then for a non-zero projection $p \in \mathcal P(M)$ we have $pa^* S b p = (pa^* a p)^{1/2} v^* S w (p b^*b p)^{1/2}$, where $pa = v | p a |$ and $pb = w | p b |$ are the polar decompositions. In particular, if $\varphi$ is a normal state on $pMp$, and we denote by $\tilde \phi$ the extension to $M$ given by $\tilde \varphi (x) = \varphi(pxp)$, then we have $s_\varphi( p a^* S b p ) \leq \varphi(p a^* a p)^{1/2} \| S \| \varphi( p b^*b p)^{1/2} = \tilde \varphi(a^*a)^{1/2} \| S \| \tilde \varphi(b^*b)^{1/2}$. Taking infimums shows that $s_\varphi(p T p) \leq s_{\tilde \varphi}(T)$ for each $T \in \B(\cH)$, and so conjugation by $p$ is continuous from the $M$-topology on $\B(\cH)$ to the $pMp$-topology on $\B(p \cH)$. Note that conjugation by $p$ is also continuous from the $M'$-topology on $\B(\cH)$ to the $pMp'$-topology on $\B(p \cH)$, and we have similar continuity properties for conjugation by $q \in \mathcal P(M')$. 

Hence, we have that $pq \K(M; \cH) pq \subset \K(pMpq; pq \cH)$, and so for $T\in \bS(M;\cH)$ and $x\in M' \subset \B(\cH)$ we have
	$[pqTpq, qxqp]=pq [T,qxq]pq\in \K(pMpq; pq \cH)$.
Therefore if we consider the maps associated to the $M$-nuclear embedding $M\subset \bS(M;\cH)$, then conjugation by $pq$ gives rise to an $pMpq$-nuclear embedding of $pMpq$ in to $\bS(pMpq; pq \cH)$. 
\end{proof}

\begin{prop}[cf.\  {\cite[Proposition 2.10]{HoIs17}}]\label{prop:amplificiationstable}
Let $M$ be a biexact von Neumann algebra.
Then $M\ovt \B(\cK)$ and $pMp$ are also biexact, for any Hilbert space $\cK$ and nonzero projection $p\in \cP(M)$.
\end{prop}
\begin{proof}
Suppose $M$ is represented on $\cH$.
To see that $N:=M\ovt \B(\cK)$ is bi-exact, first we note that $\K(M;\cH)\otimes \B(\cK)\subset \K(N; \cH\otimes \cK)$. 
Indeed, it is easy to see that we have $\K(M;\cH) \otimes \K(\cK) \subset \K(N; \cH \otimes \cK)$ and $\K(\cK)$ is dense in $\B(\cK)$ in the ultrastrong topology, and hence also in the $\B(\cK)$-topology. 
It then follows that $\bS(M;\cH)\otimes \B(\cK)\subset \bS(N;\cH\otimes \cK)$,
	since $[T\otimes S, a\otimes \id_{\cK}]=[T,a]\otimes S \in \K(N; \cH\otimes \cK)$ for any $a\in M' \cap \B(\cH)$, 
	$T\in \bS(M;\cH)$, and $S\in \B(\cK)$.

Since the identity map on $\B(\cK)$ is weakly nuclear, by tensoring with the maps coming from the $M$-nuclearity of the embedding $M \subset \bS(M; \cH)$ we may obtain u.c.p.\ maps $\phi_i: M\otimes \B(\cK)\to \M_{n(i)}(\C)$ 
	and $\psi_i: \M_{n(i)}(\C)\to \bS(M; \cH)\otimes \B(\cK)$ such that $\psi_i\circ \phi_i\to \id_{M\otimes \B(\cK)}$
	in the $( M \otimes \B(\cK) \subset M \ovt \B(\cK))$-topology. Corollary~\ref{cor:densenuclear} then gives that 	
	the embedding $N \subset \bS(N; \cH \ovt \cK)$ is $N$-nuclear. 
	
That $pMp$ is biexact when $M$ is already follows from Lemma~\ref{lem:representation independent}. 
\end{proof}

Let $M$ be a von Neumann algebra, and $\X\subset \B(L^2M)$ an $M$-boundary piece. A net of unitaries $\{u_i\}\subset \cU(M)$ is said to converge to $0$ over $\X$ if we have $\Ad(u_i)(T)\to 0$ ultraweakly for each $T\in \K_\X^{\infty, 1}(M)$. If $N \subset M$ is a von Neumann subalgebra, then we say that $N$ is weakly mixing over $\X$ if there exists a net $\{ u_i \} \subset \cU(N)$ such that $\{ u_i \}$ converges to $0$ over $\X$. If $M$ is finite, then it is not difficult to see that $\{ u_i \}$ converges to $0$ over $\X$ if and only if $\Ad(u_i)(T)\to 0$ ultraweakly for each $T\in \X$ (see, e.g., \cite[Lemma 4.2]{DKEP22}). For instance, when $M$ is finite, a net $\{ u_i \}$ converges to $0$ over $\K(L^2M)$ if and only if the net $\{ u_i \}$ converges to $0$ ultraweakly. This does not hold for general von Neumann algebra, e.g., if $M = \B(\cH)$, then $\K_{\K(\cH)}(M) = \B(\cH)$, and so there is no net $\{ u_i \}$ that converges to $0$ over $\K(\cH)$ in this case. The following lemma gives a general condition for concluding that unitaries converge to $0$ over $\X$. 

\begin{lem}\label{lem:wkconvergenceoverX}
Let $M$ be a $\sigma$-finite von Neumann algebra with a normal faithful state $\omega$, and let $\X\subset \B(L^2(M, \omega))$ be an $M$-boundary piece. Suppose $\{ u_i \} \subset \cU(M)$ is such that the weak$^*$-limit $\lim_{i \to \infty} \omega \circ {\rm Ad}(u_i)$ exists and is in $M_* \subset M^*$. If $u_i T u_i^* \to 0$ ultraweakly for each $T \in \X$, then $\{ u_i \}$ converges to $0$ over $\X$.
\end{lem}
\begin{proof}
If $a, b \in M$, $S \in \B(L^2(M, \omega))$, and $x, y \in M'$, then by Cauchy-Schwarz we have
\[
| \langle u_i^* a^* S b u_i x \omega^{1/2}, y \omega^{1/2} \rangle |
\leq \omega(u_i^* a^* a u_i )^{1/2} \| y^* \| \| S \| \| x \| \omega(u_i^* b^* b u_i)^{1/2}. 
\]
Taking infimums over all decompositions of the type $T = a^* S b$, we then have
\[
\limsup_{i \to \infty} | \langle u_i^* T u_i x \omega^{1/2}, y \omega^{1/2} \rangle |
\leq \| x \| \| y \| s_{\tilde \omega}(T),
\]
where $\tilde \omega$ is the weak$^*$-limit of $\omega \circ {\rm Ad}(u_i)$.

Thus, if $u_i^* T u_i \to 0$ ultraweakly for each $T \in \X$, then we also have $u_i^* T u_i \to 0$ ultraweakly for each $T \in \overline{\X}^{_{M-M}}$. Since conjugation by $u_i$ is $M'$-bimodular, we then also have that $u_i^* T u_i \to 0$ ultraweakly for each $T \in \overline{ \overline{\X}^{_{M-M}} }^{_{M'-M'}} = \K_\X^{\infty, 1}(M)$. 
\end{proof}

If $M$ is a $\sigma$-finite von Neumann algebra and $\mathcal U \in \beta N \setminus N$ is a nonprincipal ultrafilter, then we define
\[
\mathcal I_{\mathcal U} = \{ (x_n)_n \in \ell^\infty(M) \mid x_n \to 0 \ *-{\rm strongly \ as \ } n \to \mathcal U \};
\]
\[
\mathcal M^{\mathcal U}(M) = \{ (x_n)_n \in \ell^\infty(M) \mid (x_n)_n \mathcal I_{\mathcal U} \subset \mathcal I_{\mathcal U} {\rm \ and \ } \mathcal I_{\mathcal U} (x_n)_n \subset \mathcal I_{\mathcal U} \}.
\]
The multiplier algebra $\mathcal M^{\mathcal U}(M)$ is a ${\rm C}^*$-algebra containing $\mathcal I_{\mathcal U}$ as a norm closed two-sided ideal. The Ocneanu ultraproduct $M^{\mathcal U}$ is the quotient $\mathcal M^{\mathcal U}/ \mathcal I_{\mathcal U}$, which is a von Neumann algebra \cite{Oc85}. If $(x_n)_n \in \mathcal M^{\mathcal U}$, then we denote by $(x_n)^{\mathcal U}$ the image of $(x_n)_n$ in $M^{\mathcal U}$. We may view $M$ as a von Neumann subalgebra of $M^{\mathcal U}$ by considering the constant sequences. We have a canonical normal conditional expectation $E_M: M^{\mathcal U} \to M$ by considering the weak$^*$-limit $E_M( (x_n)^\omega ) = \lim_{n \to \mathcal U} x_n$.

If $Q \subset M$ is a von Neumann subalgebra with expectation, then given $(x_n)_n \in \ell^\infty(Q)$ we have that $(x_n)_n \in \mathcal M^{\mathcal U}(Q)$ if and only if $(x_n)_n \in \mathcal M^{\mathcal U}(M)$, and this gives rise to a canonical embedding $Q^{\mathcal U} \subset M^{\mathcal U}$. Moreover, the map $M^{\mathcal U} \ni (x_n)^\omega \mapsto (E(x_n) )^\omega \in Q^{\mathcal U}$ gives a normal faithful conditional expectation onto $Q^{\mathcal U}$ that satisfies
\[
E^{M^{\mathcal U}}_M E^{M^{\mathcal U}}_{Q^{\mathcal U}} = E^{M^{\mathcal U}}_{Q^{\mathcal U}} E^{M^{\mathcal U}}_M  =  E^M_Q E^{M^{\mathcal U}}_M = E^{Q^{\mathcal U}}_Q E^{M^{\mathcal U}}_{Q^{\mathcal U}}. 
\]
We refer to \cite{AnHa14} for more information on the Ocneanu ultraproduct. 

\begin{thm}\label{thm:biexact implies solid}
Let $M$ be a $\sigma$-finite biexact von Neumann algebra, and let $\mathcal U \in \beta \N\setminus \N$ be a nonprincipal ultrafilter. If $A \subset M$ is a von Neumann subalgebra such that there exists $u \in \mathcal U( A'\cap M^{\mathcal U})$ with $E_M(u) = 0$, then the inclusion $A \subset M$ is weakly nuclear. 
\end{thm}
\begin{proof}
Fix a normal faithful state $\varphi$ on $M$. Suppose $A\subset M$ is
	a von Neumann subalgebra and $u \in \mathcal U( A'\cap M^{\mathcal U})$ with $E_M(u) = 0$. We may then choose a sequence $\{u_n\}\subset \cU(M)$ such that $u = (u_n)_n \in \mathcal U(M^{\mathcal U})$ \cite[Lemma 2.1]{HoIs16}, i.e., we have $\lim_{n \to \mathcal U} u_n = 0$ ultraweakly, $\lim_{n \to \mathcal U} [u_n,a]= 0$ $*$-strongly for any $a\in A$, and $\tilde \varphi := \lim_{n \to \mathcal U} \varphi \circ {\rm Ad}(u_n) = \varphi \circ E_M \circ {\rm Ad}(u)_{|M}$ is contained in $M_*$. 
	
	We let $\theta: \B(L^2(M, \varphi)) \to \B(L^2(M, \varphi))$ be the point-ultraweak limit $\theta = \lim_{n \to \mathcal U} {\rm Ad}(u_n)$. By Lemma~\ref{lem:wkconvergenceoverX}, applied with $\X=\K(L^2(M,\varphi))$, we have $\theta_{| \K^{\infty, 1}(M)} = 0$. From the definition of $\bS(M)$, we therefore have $\theta(\bS(M))\subset M$.
	
Note that since $\lim_{n \to \mathcal U} [u_n, a] \to 0$ $*$-strongly for any $a \in A$, we also have $\theta_{|A} = {\rm id}$ and $\theta_{| M'} = {\rm id}$. Also, note that, by Kadison's Inequality, if $x \in M$ and $a \in M'$, we have $\| \theta(x) a \varphi^{1/2} \|^2 \leq \| a \|^2 \varphi \circ \theta(x^*x) = \| a \|^2 \tilde \varphi( x^*x )$, and from this it follows that $\theta_{|M}$ defines a normal map from $M$ to $M$. 
	
	Now consider
\[
A \subset M\subset \bS(M;\cH) \xrightarrow{\theta} M.
\]
Since $M\subset \bS(M)$ is $M$-nuclear, and since by Lemma~\ref{lem:weakcontinuity} $\theta$ is continuous from the weak-$M$-topology to the ultraweak topology on $M$, it follows that the inclusion $A \subset M$ is weakly nuclear. 
\end{proof}

We remark that it may be the case that a von Neumann subalgebra $A \subset M$ of a biexact von Neumann algebra has diffuse relative commutant even though $A$ is nonamenable. For example, we may consider $A = L^\infty(X, \mu) \ovt L \mathbb F_2 \subset \B(L^2(X, \mu)) \ovt L \mathbb F_2$ where $(X, \mu)$ is some diffuse probability space. This type of situation cannot occur, though, if the subalgebra is with expectation, i.e., if there exists a normal faithful conditional expectation $E: M \to A$.  

If $M$ is $\sigma$-finite and we have a nonprincipal ultrafilter $\mathcal U \in \beta \N\setminus \N$, then $M$ is $\mathcal U$-solid 
	if $A$ is amenable for any von Neumann subalgebra $A\subset M$ with expectation such that $A'\cap M^{\mathcal U}$ is diffuse.
	
\begin{lem}\label{lem:zeroexpectation}
Let $M$ be a diffuse $\sigma$-finite von Neumann algebra, then there exists a unitary $u \in \mathcal U( M^{\mathcal U} )$ such that $E_M(u) = 0$. 
\end{lem}
\begin{proof}
By \cite[Lemma 2.1]{HoUe16a} there exists a normal faithful state $\psi$ on $M$ such that $M^\psi$ is diffuse. If we let $A \subset M^\psi$ denote a masa, then as $M^\psi$ is finite there exists a normal $\psi$-preserving conditional expectation from $M$ to $A$, and hence $A^{\mathcal U} \subset M^{\mathcal U}$. Since $A$ is abelian, it is easy to find $u \in \mathcal U(A^{\mathcal U})$ such that $E_A(u) = 0$, and we then have $E_M(u) = E_A(u) = 0$. 
\end{proof}

\begin{thm}\label{thm:solid}
Let $M$ be a $\sigma$-finite biexact von Neumann algebra, then $M$ is $\mathcal U$-solid.
\end{thm}	
\begin{proof}
Suppose $A \subset M$ is with expectation such that $A' \cap M^{\mathcal U}$ is diffuse. We let $p \in \mathcal Z(A' \cap M)$ denote the central projection such that $(A' \cap M)p$ is diffuse and $p^\perp = \sum_{i \in I} p_i$, where the $p_i\in\mathcal Z(A'\cap M)$ are nonzero pairwise orthogonal projections and $N_i=(A'\cap M)p_i$ is a type I factor.

Since $N_i$ is a type I factor, we have $p_i ( A'  \cap M^{\mathcal U} ) p_i \cong P_i \ovt N_i$ where $P_i$ is a diffuse von Neumann algebra and then the restriction of $E_M$ to $P_i$ maps into $\mathcal Z(N_i) = \mathbb C$, i.e., $E_M$ defines a normal state on $P_i$. Since $P_i$ is diffuse, we may choose some $u_i \in \mathcal U(P_i) \subset \mathcal U(   p_i ( A' \cap M^{\mathcal U} ) p_i )$ so that $E_M(u_i) = 0$. 

Since $A \subset M$ is with expectation, we also have that $p ( A'  \cap  M ) p$ is with expectation in $pMp$, and hence we have a canonical embedding $p ( A' \cap M )^{\mathcal U} p \subset p M^{\mathcal U} p$. By Lemma~\ref{lem:zeroexpectation} we may then choose $u_0 \in p ( A' \cap M )^{\mathcal U} p$ so that $E_{pMp}(u_0) = E_{p( A' \cap M )p}(u_0) = 0$. Setting $u = u_0 + \sum_{i \in I} u_i$ then gives a unitary in $A' \cap M^{\mathcal U}$ so that $E_M(u) = 0$. 

By Theorem~\ref{thm:biexact implies solid} we then have that $A \subset M$ is weakly nuclear, and since $A$ is with expectation it follows that $A$ is amenable.
\end{proof}

Suppose $M$ is a von Neumann algebra and $N\subset M$ is a von Neumann subalgebra with a normal faithful conditional expectation $E: M\to N$.
Denote by $e_N: L^2M\to L^2N$ the corresponding projection between the standard forms.
If $\X\subset \B(L^2M)$ is an $M$-boundary piece, 
	then we have that the norm closure of $e_N \K_\X(M) e_N$ is an $N$-boundary piece, which we denote by $\X^N$.

\begin{prop}\label{prop:subalgebra}
Let $M$ be a von Neumann algebra and $N\subset M$ a von Neumann subalgebra with a normal faithful conditional expectation $E: M \to N$.
Suppose $\X\subset \B(L^2M)$ is an $M$-boundary piece and $M$ is biexact relative to $\X$.
Then $N$ is biexact relative to $\X^N$, where $\X^N=\overline{e_N \K_\X(M) e_N}$ is the $N$-boundary piece associated with $\X$.
In particular, $N$ is biexact if $M$ is biexact.
\end{prop}
\begin{proof}
By Corollary~\ref{cor:condexpcont} we have $\Ad(e_N) (\K_\X^{\infty,1}(M))\subset \K_{\overline{e_N \K_\X(M) e_N}}^{\infty,1}(N)$, and since $[JNJ, e_N]=0$ it then follows that $\Ad(e_N)$ maps $\bS_\X(M)$ to $\bS_{\X^N}(N)$. Since $M \subset \bS_\X(M)$ is $M$-nuclear, by composing the corresponding maps with $\Ad(e_N)$ and using again Corollary~\ref{cor:condexpcont}, we see that $N \subset \bS_{\X^N}(N)$ is $N$-nuclear. 
\end{proof}

We recall from \cite{DKEP22} that a von Neumann algebra $M$ is properly proximal if for every non-zero central projection $p \in \mathcal Z(M)$ there does not exist a conditional expectation $E: \bS(pM)\to pM$. Note that if $M$ is biexact, then by composing the maps giving $M$-nuclearity with such an expectation we see that $pM$ is amenable. Hence, there then exists a unique projection $p \in \mathcal Z(M)$ so that $p M$ is amenable and $p^\perp M$ is properly proximal. The previous proposition, therefore, gives a generalization of Theorem 1.1 from \cite{DKEP22}. 

\begin{cor}
Let $M$ be a biexact von Neumann algebra, then every von Neumann subalgebra is either properly proximal, or else has a non-zero amenable summand. 
\end{cor}

Suppose $\mathcal F = \{ (B_j, E_j) \}_{j \in J}$ is a family of von Neumann subalgebras $B_j\subset M$, together with normal faithful conditional expectations. For each $j \in J$, we let $e_{B_j}: L^2M \to L^2B_j$ denote the orthogonal projection corresponding to $E_j$. We let $\X_{ \mathcal F}$ denote the hereditary ${\rm C}^*$-subalgebra of $\B(L^2M)$ generated by $\{ x JyJ e_{B_j}\mid x, y \in M, j \in J \}$, and we say that $M$ is biexact relative to $\mathcal F$ if it is biexact relative to the boundary piece $\X_{\mathcal F}$. If the conditional expectations are understood from the context (e.g., if $M$ is a II$_1$ factor), then we may write simply $\X_{ \{ B_j \}_{j \in J} }$ instead. For a single von Neumann subalgebra $B \subset M$ with normal faithful conditional expectation $E: M \to B$, we will write even more simply $\X_{(B, E)}$ or $\X_B$ for this boundary piece, and we will say that $M$ is biexact relative to $B$ if it is biexact relative to $\X_B$.

\begin{lem}
Using the above notations, assume $M$ is $\sigma$-finite.
For a finite von Neumann subalgebra $N\subset M$ with expectation, either there exists a net $\{ u_i \}_i \subset \mathcal U(N)$ that converges to $0$ over $\X_{\mathcal F}$, or else we have $N\preceq_M B_j$ for some $j \in J$. 
\end{lem}

\begin{proof}
Let $\varphi$ denote a normal faithful state on $M$ such that $\varphi_{|N}$ is a trace. 
Suppose $N\not\preceq_M B_j$ for any $j \in J$, then we may choose a net $\{u_i\}\subset \cU(N)$
	such that $E_{j}(a u_i b)\to 0$ strongly for any $a,b\in M$ and $j \in J$
	by \cite[Definition 4.1]{HoIs17} and \cite[Remark 1.2]{Ioa12}.
Notice that for any $T\in \B(L^2M)$, $x_n, a_n\in M$ and $j_1, j_2\in J$  
\[\begin{aligned}
&|\langle u_i^* a_1 e_{B_{j_1}} T e_{B_{j_2}} a_2 u_i x_1 1_{\varphi}, x_21_{\varphi}\rangle|\\
	\leq & \|T\| \|E_{{j_1}}(a_2u_ix_1)\|_{\varphi} \|E_{{j_2}}(a_1^* u_ix_2)\|_{\varphi}\to 0,
\end{aligned}
\]
and hence $\Ad(u_i)(a_1 Jb_1J e_{B_j}T e_{B_j} Jb_2J a_2)\to 0$ ultraweakly for any $a_1, a_2, b_1, b_2\in M$. It then follows that ${\rm Ad}(u_i)(S) \to 0$ ultraweakly for any $S \in \X_{\mathcal F}$, and since $u_i \in M$, which is in the centralizer of $\varphi$, it follows from Lemma~\ref{lem:wkconvergenceoverX} that $\{ u_i \}$ converges to $0$ over $\X_{\mathcal F}$. 
\end{proof}

Continuing in the above setting,
notice that the proof of Theorem~\ref{thm:biexact implies solid} shows that if there exists a net $\{u_i\}\subset \cU(N)$ such that $\{u_i \}$ converges to $0$ over $\X$, then there exists a u.c.p.\ map $\theta: \B(L^2M)\to \B(L^2M)$ such that $\theta$ vanishes on $\K_\X^{\infty, 1}(M)$ and $\theta_{| N' \cap M}$ is the identity. 
Therefore, the same proof as in Theorem~\ref{thm:biexact implies solid} yields the following.

\begin{prop}\label{prop:relative biexact intertwine}
Let $M$ be a $\sigma$-finite von Neumann algebra, and $\mathcal F = \{ (B_j, E_j) \}_{j \in J}$ a family of von Neumann subalgebras of $M$ with expectation.
Suppose $M$ is biexact relative to $\X_{\mathcal F}$ and $N\subset M$ is a finite von Neumann subalgebra with expectation.
Then either $N'\cap M$ is amenable, or else $N\preceq_M B_j$ for some $j \in J$.
\end{prop}

\subsection{Permanence properties}

\begin{prop}[cf.\ {\cite[Lemma 15.3.3]{BO08}}]\label{prop:biexacttensor}
For $i = 1, 2$, let $M_i$ be a von Neumann algebra and $\X_i \subset \B(L^2 M_i)$ be an $M_i$-boundary piece. If $M_i$ is biexact relative to $\X_i$, then $M=M_1\ovt M_2$ is biexact relative to the hereditary ${\rm C}^*$-algebra generated by $\X_1 \otimes \B(L^2M_1)$ and $\B(L^2M_1) \otimes \X_2$. 

In particular, if each $M_i$ is biexact, then $M$ is biexact relative to $\X_{\{M_1, M_2\}}$.
\end{prop}
\begin{proof}
Let $\X \subset \B(L^2M)$ denote the hereditary ${\rm C}^*$-algebra generated by $\X_1 \otimes \B(L^2M_1)$ and $\B(L^2M_1) \otimes \X_2$. 

First we show that $\bS_{\X_1}(M_1)\otimes \bS_{\X_2}(M_2)\subset \bS_\X (M)$.
Indeed, note that for any $T\in \bS_{\X_1}(M_1)$, $S\in \bS_{\X_2}$, $a\in M_1$, $b \in M_2$ we have 
\[
[T\otimes S, J(a\otimes 1)J]=[T, JaJ]\otimes S\in \K^{\infty,1}_{\X_1}(M_1)\otimes \B(L^2 M_2)\subset \K_\X^{\infty,1}(M)
\]
and 
\[
[T\otimes S, J(1\otimes b)J]=T\otimes [S, JbJ] \in \B(L^2 M_1) \otimes \K^{\infty,1}_{\X_2}(M_2) \subset \K_\X^{\infty,1}(M).
\]
It then follows from \cite[Lemma 6.1]{DKEP22} that $T\otimes S\in \bS_\X(M)$.

For each $i=1,2$, denote by $\phi_j^i: M_i\to \M_{n(i,j)}(\C)$ and $\psi_j^i: \M_{n(i,j)}(\C)\to \bS_{\X_i}(M_i)$ 
	those u.c.p.\ maps given by relative biexactness of $M_i$.
Then, by considering the u.c.p.\ maps $\phi_j^1\otimes \phi_j^2: M_1\otimes M_2\to \M_{n(1, j)}(\C)\otimes \M_{n(2, j)}(\C)$ and $\psi_j^1\otimes \psi_j^2:\M_{n(1, j)}(\C)\otimes \M_{n(2, j)}(\C)\to \bS_{\X_1}(M_1)\otimes \bS_{\X_2}(M_2)\subset \bS_\X (M)$, we see that $M_1\otimes M_2\subset \bS_\X (M)$ is $(M_1\otimes M_2 \subset M_1 \ovt M_2)$-nuclear.
It follows from Corollary~\ref{cor:densenuclear} that $M\subset \bS_\X (M)$ is $M$-nuclear.
\end{proof}

In \cite{HoIs17}, Houdayer and Isono generalize Ozawa and Popa's unique prime decomposition results from \cite{OzPo04} to the type III setting. The techniques developed in \cite{HoIs17} also apply to general biexact factors. 

\begin{cor}\label{cor:primedecomposition}
Let $m,n\geq 1$ be integers. 
For each $1\leq i\leq n$, let $M_i$ be a nonamenable biexact factor.
For each $1\leq j\leq n$, let $N_j$ be any non-type $I$ factor and suppose 
	$M=\overline{\otimes}_{i=1}^{\ m} M_i=\overline{\otimes}_{j=1}^{\ n} N_j$.
There then exists a surjection $\sigma: \{1,\dots, m\}\to \{1, \dots, n\}$, type $I$ factors $F_1,\dots, F_n$, 
	and a unitary $u\in ( \overline{\otimes}_{j=1}^{\ n}  F_j ) \ovt M $
	such that $u(F_j \ovt N_j)u^*=(\overline{\otimes}_{i\in \sigma^{-1}(j)}  M_i) \ovt F_j $.
\end{cor}
\begin{proof}
We simply replace the class $\cC_{(AO)}$ from \cite{HoIs17} by the class of  biexact von Neumann algebras. The class of biexact von Neumann algebras is stable under amplification and tensoring by type I factors by Proposition~\ref{prop:amplificiationstable}, biexact von Neumann algebras are $\mathcal U$-solid by Theorem~\ref{thm:solid}, and by Propositions~\ref{prop:biexacttensor} and \ref{prop:relative biexact intertwine} we have that biexact von Neumann algebras satisfy the conclusion of \cite[Theorem 5.1]{HoIs17}. The rest of the argument form \cite{HoIs17} may then be applied while using \cite[Application 4]{AnHaHoMa20}.
\end{proof}

We now consider biexactness for free products, and we will continue to use the notation for free products as in Section~\ref{sec:freeproductwkexact} above.

If $(M_1, \omega_1)$ and $(M_2, \omega_1)$ are von Neumann algebras with normal faithful states, then we let $e_i \in \B( L^2( (M_1, \omega_1) * (M_2, \omega_2) ) )$ denote the orthogonal projection onto $L^2(M_i, \omega_i)$. We also let $p_i \in \B(L^2(M_i, \omega_i))$ denote the rank-one projection onto $\mathbb C \omega_i^{1/2}$. 

If $\X_i \subset \B(L^2(M_i, \omega_i))$ is an $M_i$-boundary piece for each $i = 1, 2$, then we denote by $\X_1 \vee \X_2 \subset \B( L^2( (M_1, \omega_1) * (M_2, \omega_2) ) )$ the boundary piece generated by $e_1 \X_1 e_1$ and $e_2 \X_2 e_2$. Note that when $\X_i = \B(L^2(M_i, \omega_i))$ for $i = 1,2$, then $\X_1 \vee \X_2$ is the boundary piece $\X_{ \{ M_1, M_2 \} }$. Also, if $\X_i = \K(L^2(M, \omega_i))$ for $i = 1,2$, then $\X_1 \vee \X_2 = \K( L^2( (M_1, \omega_1) * (M_2, \omega_2) ) )$. The following proposition generalizes Proposition 15.3.13 from \cite{BO08} in the case of group von Neumann algebras. 

\begin{prop}\label{prop:free product biexact}
Let $(M_1, \omega_1)$ and $(M_2, \omega_1)$ be von Neumann algebras with normal faithful states. Suppose for $i = 1,2$, $\X_i \subset \B(L^2(M_i, \omega_i))$ is an $M_i$ boundary piece such that each $M_i$ is biexact relative to $\X_i$.  
Then, $M=M_1\ast M_2$ is biexact relative to $\X_1 \vee \X_2$. In particular, if each $M_i$ is weakly exact, then $(M_1, \omega_1) * (M_2, \omega_2)$ is biexact relative to $\X_{ \{M_1, M_2 \} }$, and if each $M_i$ is biexact, then $(M_1, \omega_1) * (M_2, \omega_2)$ is also biexact.
\end{prop}
\begin{proof}
We set $M_0 = (M_1, \omega_1) *_r (M_2, \omega_2)$ and $M = (M_1, \omega_1) * (M_2, \omega_2)$.  By Theorem~\ref{thm:freeproductnuclear}, we have that the inclusion $M_0 \subset (\bS_{\X_1}(M_1), \omega_1) *_r (\bS_{\X_2}(M_2), \omega_2)$ is $( M_0 \subset M )$-nuclear, and hence by Corollary~\ref{cor:densenuclear} it suffices to check that we have an inclusion 
\[
(\bS_{\X_1}(M_1), \omega_1) *_r (\bS_{\X_2}(M_2), \omega_2) \subset \bS_{\X_1 \vee \X_2}((M_1, \omega_1) * (M_2, \omega_2)).
\] 

The same computation as in \cite[Proposition 3.2]{Oza06} shows that for $T\in \bS_{\X_1}(M_1)$,
	$a\in M_1' \subset \B(L^2(M_1, \omega_1))$ and $b\in M_2' \subset \B(L^2(M_2, \omega_2))$,
we have $[\lambda(T), \rho(b)]=0$ and 
\begin{equation}\label{eq:commutator1}
[\lambda(T), \rho(a)] = e_1 \lambda( [ T, a]  ) e_1 \in e_1 \K_{\X_1}^{\infty, 1}(M_1) e_1 \subset  \K_{\X_1 \vee \X_2}^{\infty, 1}(M).
\end{equation} 
Similarly, if $S \in \bS_{\X_2}(M_2)$, we have 
\begin{equation}\label{eq:commutator2}
[\lambda(S), \rho(b)] = e_2 \lambda( [ S, b ]) e_2 \in e_2 \K_{\X_2}^{\infty, 1}(M_2) e_2 \subset \K_{\X_1 \vee \X_2}^{\infty, 1}(M)
\end{equation} 
and $[\lambda(S), \rho(a)] = 0$. 

Note that for each $T \in \B(L^2(M_1, \omega_1) )$ with $\omega_1(T) = 0$, we have $\lambda(T) e_2 = \lambda(T p_1) e_2$, and since $\omega_1^{1/2}$ is cyclic for $M_1$ we may find a sequence $x_n \in M_1$ with $\omega_1(x_n) = 0$ so that $\| \lambda(Tp_1) - \lambda(x_n p_1) \| \to 0$. More generally, we see that, for any operator of the form $\lambda(T_1) \lambda(S_1) \cdots \lambda(S_{n-1}) \lambda(T_n) e_2$ with $T_i \in \B(L^2(M_1, \omega_1) )$ satisfying $\omega_1(T_i) = 0$, and $S_i \in \B(L^2(M_2, \omega_2))$ satisfying $\omega_2(S_i) = 0$, we have
\[
\lambda(T_1) \lambda(S_1) \cdots \lambda(S_{n-1}) \lambda(T_n) e_2
= \lambda(T_1p_1) \lambda(S_1p_2) \cdots \lambda(S_{n-1}p_{2}) \lambda(T_np_1) e_2,
\]
and hence may be approximated in uniform norm by operators of the form $\lambda(x_1) \lambda(y_1) \cdots \lambda(y_{n-1}) \lambda(x_n) e_2$ where $x_i \in M_1$ with $\omega_1(x_i) = 0$, and $y_i \in M_2$ with $\omega_2(y_i) = 0$. This similarly follows for operators of any of the following forms 
\[
\lambda(S_1) \cdots \lambda(S_{n-1}) \lambda(T_n) e_2,
\]
\[
\lambda(S_1) \lambda(T_1) \cdots \lambda(T_{n-1}) \lambda(S_n) e_1,
\]
\[
\lambda(T_1) \cdots \lambda(T_{n-1}) \lambda(S_n) e_1.
\]

As a consequence, if $T_1, \ldots, T_n \in \bS_{\X_1}(M_1)$ with $\omega_1(T_i) = 0$, and if $S_1, \ldots, S_n \in \bS_{\X_2}(M_2)$ with $\omega_2(S_i) = 0$, then from (\ref{eq:commutator1}) and (\ref{eq:commutator2}) it follows that for $a \in M_1' \subset \B(L^2(M_1, \omega_2))$ or $a \in M_2' \subset \B(L^2(M_2, \omega_2))$ we have
\[
\lambda( T_1 S_1\cdots S_{k-1} )[\lambda(T_k), \rho(a)] \lambda(S_k \cdots T_nS_n ) \in \K_{\X_1 \vee \X_2}^{\infty, 1}(M)
\]
and
\[
\lambda( T_1 S_1\cdots S_{k-1}T_k ) [\lambda(S_k), \rho(a)] \lambda(T_{k + 1} \cdots T_nS_n ) \in \K_{\X_1 \vee \X_2}^{\infty, 1}(M).
\]
By summing these terms we then have
\[
[ \lambda(T_1 S_1 \cdots T_n S_n), \rho(a)] \in  \K_{\X_1 \vee \X_2}^{\infty, 1}(M).
\]
This similarly holds for words starting with operators in $\bS_{\X_2}(M_2)$ or ending with operators in $\bS_{\X_1}(M_1)$. Since $M'$ is generated by $\rho(M_1') \cup \rho(M_2')$ \cite[Theorem 1.6.5]{VoDyNi92}, it then follows from \cite[Lemma 6.1]{DKEP22} that 
\[
(\bS_{\X_1}(M_1), \omega_1) *_r (\bS_{\X_2}(M_2), \omega_2) \subset \bS_{\X_1 \vee \X_2}((M_1, \omega_1) * (M_2, \omega_2)).
\] 
\end{proof}

The previous theorem can be used to derive Kurosh type theorems for general free products of weakly exact von Neumann algebras similar to the results in \cite{Oza06}. However, more  general Kurosh type theorems already exist in \cite{HoUe16}, and so we will not pursue this direction further. 

\subsection{Biexactness relative to the amenable boundary piece}

In \cite[Section 6.1]{DKEP22}, 
	distinguished canonical amenable boundary pieces $I_{\rm amen} \subset \ell^\infty \Gamma$ and $\X_{\rm amen} \subset \B(L^2M)$ were associated to groups and von Neumann algebras, respectively, which we will briefly describe.

Given a discrete group $\Gamma$, $I_{\rm amen}\subset \ell^\infty\Gamma$ consists of all functions $f$ that satisfy $\lim_{i\to \infty} f(t_i)=0$ whenever $\{t_i\}\subset \Gamma$ is net such that $\lambda_{t_i}\to 0$ in the weak topology in $C^*_\lambda\Gamma$.

If $(M,\tau)$ is a tracial von Neumann algebra and $\cK$ is a Hilbert $M$-bimodule, 
	we denote by $\cK^0\subset \cK$ the subspace of left and right bounded vectors, i.e., those vectors $\xi \in \cK$ such that the operators $L_\xi, R_\xi: M \to \cK$ given by $L_\xi x = \xi x$ and $R_\xi x = x \xi$ extend to bounded operators from $L^2M$ to $\cK$. For $\xi\in\cK^0$, set $T_\xi=R_\xi^*L_\xi\in\B(L^2M)$.
	We let 
	\[
	B_0 = {\rm sp} \{ T_\xi \mid \xi \in \cH^0 {\rm \ for \ some \ Hilbert \ bimodule \ } \cH \prec L^2M \ovt L^2M \}.
	\]
	Note that $B_0$ forms a $*$-subalgebra of $\B(L^2M)$,
	since for $\xi \in \cH^0$ and $\eta \in \cK^0$ we have
	$T_\xi^* T_\eta=T_{\bar{\xi}\otimes \eta}$ where $\bar \xi\otimes \eta\in \bar\cH\otimes_M \cK$, and we have $\bar\cH\otimes_M \cK \prec L^2M\otimes L^2M$ if $\cH \prec L^2M\otimes L^2M$, or if $\cK\prec L^2M\otimes L^2M$.
The amenable $M$-boundary piece $\X_{\rm amen}$ is defined to be $B \B(L^2M) B$, where $B \subset \B(L^2M)$ is the ${\rm C}^*$-algebra generated by $B_0$, i.e., $\X_{\rm amen}$ is the hereditary ${\rm C}^*$-subalgebra of $\B(L^2M)$ generated by $B_0$.
It is clear that $\X_{\rm amen}$ is an $M$-boundary piece as $M$ and $JMJ$ are contained in the multiplier algebra of $B$.

It was shown in \cite[Theorem 6.14]{DKEP22} that if $E: \B(\ell^2 \Gamma) \to \ell^\infty \Gamma$ is the canonical conditional expectation, then $E(\X_{\rm amen}) \subset \ell^\infty \Gamma$ generates $I_{\rm amen}$, and $I_{\rm amen} \subset \X_{\rm amen} \subset \B(L^2( L\Gamma))$. As a consequence, we also obtain the following corollary from Theorem~\ref{thm:iff}.

\begin{cor}
Let $\Gamma$ be a discrete group, then $\Gamma$ is biexact relative to its amenable boundary piece if and only if $L\Gamma$ is biexact relative to its amenable boundary piece. 
\end{cor}

If $M$ is a tracial von Neumann algebra and $N \subset M$ is a von Neumann subalgebra, then by \cite[Lemma 6.13]{DKEP22} we have $e_N \in \X_{\rm amen}$ if and only if $N$ is amenable. In particular, it follows that if $M$ is biexact relative to $N$ and $N$ is amenable, then $M$ is biexact relative to its amenable boundary piece.

Note that if $\Gamma$ is a countable group with a normal amenable subgroup $\Sigma \lhd \Gamma$ such that $\Gamma/\Sigma$ is biexact, 
	then $\Gamma$ is biexact relative to its amenable boundary piece.
Since $\Gamma/\Sigma$ is biexact and $\Sigma$ is amenable, the induced action $\Gamma\actson\bS(\Gamma/\Sigma)$ is amenable by \cite[Proposition 5.1.11]{BO08}. Thus, it suffices to construct a unital left $\Gamma$-equivariant embedding from $\bS(\Gamma/\Sigma)$ to $\bS_{I_{\rm amen}}(\Gamma)$.
Consider the embedding $\iota:\ell^\infty(\Gamma/\Sigma)\to \ell^\infty\Gamma$ given by $\iota (f)(t)=f(t\Sigma)$.
Note that $\iota(c_0(\Gamma/\Sigma))\subset I_{\rm amen}$ since $1_\Sigma\in I_{\rm amen}$ by \cite[Lemma 6.12]{DKEP22}
	and $R_t\iota(f)(s)=\iota(R_{t\Sigma}f)(s)$ for $t\in \Gamma$.
It then follows that $\iota (\bS(\Gamma/\Sigma))\subset \bS_{I_{\rm amen}}(\Gamma)$.

Similarly, if $M$ is biexact and $R$ is amenable, then $M \ovt R$ is biexact relative to its amenable boundary piece. 
This follows easily from the proof of Proposition~\ref{prop:biexacttensor} and the above lemma. We also note that, in Proposition~\ref{prop:actiononamenable} below, we show that if $R$ is amenable, $\Gamma$ is biexact, and $\Gamma \actson R$ is a trace-preserving action, then $R \rtimes \Gamma$ is biexact relative to $R$, and hence is biexact relative to its amenable boundary piece.

We remark that the amenable boundary piece is still small enough to obtain structural indecomposability results.
The following is a direct consequence of \cite[Lemma 6.13]{DKEP22} and the proof of Proposition~\ref{prop:relative biexact intertwine}.

\begin{prop}[cf.\ Remark 4.8 in \cite{Oza06}]\label{prop:solid rel amenable boundary}
Let $M$ be a tracial von Neumann algebra that is biexact relative to its amenable boundary piece.
If a von Neumann subalgebra $N\subset M$ has no amenable direct summand, then $N'\cap M$ is amenable. In particular, any subfactor $N\subset M$ is either McDuff or prime.
\end{prop}

\begin{lem}\label{lem:amenable boundary subalg}
Let $M$ be a tracial von Neumann algebra and $N\subset M$ a von Neumann subalgebra.
Denote by $\X_{\rm amen}$ the amenable $M$-boundary piece and $\X_{\rm amen}^N$ the amenable $N$-boundary piece.
Then $e_N \X_{\rm amen} e_N\subset \X_{\rm amen}^N$ and $\overline{e_N \K_{\X_{\rm amen}}(M) e_N} \subset\K_{\X_{\rm amen}^N}(N)$.
\end{lem}
\begin{proof}
Let $\cH$ be a universal Hilbert $M$-bimodule that is weakly contained in the coarse bimodule and denote by $\cH^0$ the set of left and right bounded vectors in $\cH$.
For $\xi\in \cH^0$, note that $e_N T_\xi^*T_\xi e_N= T_{\bar \xi\otimes\xi}$, where $\bar\xi\otimes \xi$ is viewed  in ${}_N \bar\cH\otimes _M\cH_N\prec L^2N\otimes L^2N$.
It follows that $e_N (\X_{\rm amen})_+e_N \subset \X_{\rm amen}^N$.
By Corollary~\ref{cor:condexpcont} $\Ad(e_N):\B(L^2M)\to \B(L^2N)$ is continuous from the weak $M$-topology (resp.\ weak $M'$-topology) to the weak $N$-topology (resp.\ weak $N'$-topology), and hence we then have $\overline{e_N \K_{\X_{\rm amen}}(M) e_N} \subset\K_{\X_{\rm amen}^N}(N)$.
\end{proof}

As a consequence of the Lemma~\ref{lem:amenable boundary subalg}, Proposition~\ref{prop:solid rel amenable boundary}
	and Proposition~\ref{prop:subalgebra}, we obtain the following proposition.

\begin{prop}
Let $M$ be a tracial von Neumann algebra that is biexact relative to its amenable boundary piece.
Then, any von Neumann subalgebra $N\subset M$ is also biexact relative to its amenable $N$-boundary piece.
\end{prop}

\section{Additional formulations and examples of biexact von Neumann algebras}

We now give some equivalent characterizations of biexactness for von Neumann algebras that will be useful for providing additional examples.

\begin{lem}\label{lem:bimodularembedding}
Let $M$ be a von Neumann algebra and $\X \subset \B(L^2M)$ be an $M$-boundary piece. There then exists a u.c.p.\ $M$-bimodular map $\phi: \B(L^2M) \to \K_\X(M)^{\sharp *}$. 
\end{lem}
\begin{proof}
Let $\{ e_i \}_i \subset \K_\X(M)$ be an approximate unit that is quasi-central with respect to $M$, and let $\phi_i: \B(L^2M) \to \K_\X(M) \subset \K_\X(M)^{**}$ be the c.c.p.\ map given by $\phi_i(T) = e_i^{1/2} T e_i^{1/2}$. Then, any point-ultraweak cluster point of $\{ \phi_i \}_i$ gives a u.c.p.\ $M$-bimodular map $\phi: \B(L^2M) \to \K_\X(M)^{**}$, and we may then take the compression to $\K_\X(M)^{\sharp *}$. 
\end{proof}

\begin{lem}\label{lem:cptinclusion}
Let $M$ be a von Neumann algebra with an $M$-boundary piece $\X \subset \B(L^2M)$. Suppose $A \subset \B(L^2M)$ is a ${\rm C}^*$-subalgebra containing the identity operator, and $\phi: A \to \B(L^2M)$ is a u.c.p.\ map such that $\phi(a) - a \in \K_\X^{\infty, 1}(M)$ for each $a \in A$, then $\phi(a) - a \in \K_\X(M)$ for each $a \in A$. 
\end{lem}
\begin{proof}
Take the Stinespring representation $\phi(a) = V^* \pi(a) V$ where $\mathcal H$ is a Hilbert space, $V: L^2M \to \mathcal H$ is an isometry, and $\pi: A \to \B(\mathcal H)$ is a $*$-representation. If $a \in A$, then we have $| \pi(a) V - V a |^2 = \phi(a^* a) - \phi(a^*) a - a^* \phi(a) + a^* a \in \K_\X^{\infty, 1}(M)$, and hence $\pi(a) V - V a \in \K_\X^L(M, \mathcal H)$. Thus, $\phi(a) - a = V^* (\pi(a) V - V a) \in \K_\X^L(M)$. Replacing $a$ by $a^*$ shows that we also have $\phi(a^*) - a^* \in \K_\X^L(M)$, and hence $\phi(a) - a  \in \K_\X^L(M) \cap \K_\X^L(M)^* = \K_\X(M)$. 
\end{proof}

If $M$ is a von Neumann algebra with a normal faithful state $\mu$, then we continue to denote by $\mu$ the vector state on $\B(L^2(M, \mu))$ given by $\B(L^2(M, \mu)) \ni T \mapsto \langle T \mu^{1/2}, \mu^{1/2} \rangle$, where $\mu^{1/2}$ is the cyclic vector given by the GNS-construction of $M$ with respect to $\mu$. We recall that $r_\mu$ denotes the seminorm on $\B(L^2(M, \mu))$ given by 
\[
r_\mu(T) = \inf \{ ( \mu(a^*a) + \mu(b^*b) )^{1/2} \| Z \| (\mu(c^*c) + \mu(d^*d) )^{1/2} \}
\] 
where the infimum is taken over all decompositions $T = \left( \begin{smallmatrix}
 a  \\
b
\end{smallmatrix}
\right)^*  Z \left( \begin{smallmatrix}
 c  \\ 
d
\end{smallmatrix} \right)$ where $a, c \in M'$, $b, d \in M$ and $Z \in \mathbb M_2(\B(L^2(M, \mu)))$.

\begin{thm}\label{thm:biexact}
Let $M$ be a separable von Neumann algebra with a normal faithful state $\mu$. Let $\X \subset \B(L^2 M)$ be an $M$-boundary piece. The following conditions are equivalent: 
\begin{enumerate}
\item\label{item:biexact1} $M$ is biexact relative to $\X$.
\item\label{item:biexactapprox} $M$ is weakly exact, and for every $\varepsilon > 0$ and finite-dimensional operator systems $E \subset F \subset \B(L^2M)$ with $E \subset M$ there exists a u.c.p.\ map $\phi: \B(L^2M) \to \B(L^2M)$ such that $d_{r_\mu}(x - \phi(x), \K_{\X}^{\infty, 1}(M)) < \varepsilon$, and $d_{r_\mu}([JxJ, \phi(T)], \K_{\X}^{\infty, 1}(M)) < \varepsilon$ for all $x \in E$ and $T \in F$. 
\item\label{item:biexactapprox2} There exists a separable ultraweakly dense ${\rm C}^*$-subalgebra $M_0 \subset M$ and a separable ${\rm C}^*$-subalgebra $\B \subset \B(L^2M)$ containing $M_0$ and the identity operator such that the inclusion $M_0 \subset \B$ is $(M_0 \subset M)$-nuclear and such that for every $\varepsilon > 0$ and finite-dimensional operator systems $E \subset F \subset \B$ with $E \subset M_0$ there exists a u.c.p.\ map $\phi: \B(L^2M) \to \B(L^2M)$ such that $d_{r_\mu}(x - \phi(x), \K_{\X}^{\infty, 1}(M)) < \varepsilon$, and $d_{r_\mu}([JxJ, \phi(T)], \K_{\X}^{\infty, 1}(M)) < \varepsilon$ for all $x \in E$ and $T \in F$. 
\item\label{item:biexact3} $M$ is weakly exact, and for every separable ${\rm C}^*$-subalgebra $M_0 \subset M$ and separable ${\rm C}^*$-subalgebra $\B \subset \B(L^2M)$ containing $M_0$ and the identity operator, there exists a u.c.p.\ map $\phi: \B \to \bS_\X(M)$ satisfying $\phi(x) - x \in \K_{\X}(M)$ for all $x \in M_0$. 
\item\label{item:biexact2} There exists a separable ultraweakly dense ${\rm C}^*$-subalgebra $M_0 \subset M$ and a separable ${\rm C}^*$-subalgebra $\B \subset \B(L^2M)$ containing $M_0$ and the identity operator such that the inclusion $M_0 \subset \B$ is $(M_0 \subset M)$-nuclear, and such that there exists a normal u.c.p.\ map $\phi: \B \to \bS_\X(M)$ satisfying $\phi(x) - x \in \K_{\X}(M)$ for all $x \in M_0$. 
\end{enumerate}
\end{thm}
\begin{proof}
First note that (\ref{item:biexact1}) $\implies$ (\ref{item:biexactapprox}) is trivial. We now show that (\ref{item:biexactapprox}) $\implies$ (\ref{item:biexact3}) and (\ref{item:biexactapprox2}) $\implies$ (\ref{item:biexact2}) using an idea inspired by Exercise 15.1.1 in \cite{BO08}. 

Let $M_0 \subset M$ be a separable ${\rm C}^*$-subalgebra, and let $\B \subset \B(L^2M)$ be a separable ${\rm C}^*$-algebra containing $M_0$ and the identity operator. Enumerate dense subsets $\{ x_k \}_k \subset M_0$ and $\{ T_k \}_k \subset \B$. By hypothesis, for each $n \geq 1$ there exists a u.c.p.\ map $\phi_n: \B(L^2M) \to \B(L^2M)$ so that the following conditions hold:
\begin{enumerate}
\item For each $i \leq n$, there exists $A_{i, 0, n} \in \K_\X(M)$ such that $r_{\mu}( x_i - \phi_n(x_i) - A_{i, 0, n} ) < 2^{-n}$.
\item For each $1 \leq i, j \leq n$, there exists $A_{i, j, n} \in \K_\X(M)$ such that $r_{\mu}( [ J x_i J, T_j ]  - A_{i, j, n} ) < 2^{-n}$. 
\end{enumerate}
Using the definition of the norm $r_\mu$ there then exist, for each $n \geq 1$, $1 \leq i \leq n$, and  $0 \leq j \leq n$, elements $a_{i, j, n}, c_{i, j, n} \in M'$, $b_{i, j, n}, d_{i, j, n} \in M$ and $Z_{i, j, n} \in \B(L^2M)$ with $\| Z_{i, j, n} \| = 1$ and 
\[
( \mu( a_{i, j, n}^* a_{i, j, n} ) + \mu( b_{i, j, n}^* b_{i, j, n} ))^{1/2} = ( \mu (c_{i, j, n}^* c_{i, j, n} ) + \mu(d_{i, j, n}^* d_{i, j, n} ) )^{1/2} < 2^{-n}
\] 
so that
\[
x_i - \phi_n(x_i) - A_{i, 0, n} = \left( \begin{smallmatrix}
 a_{i, 0, n}  \\
b_{i, 0, n}
\end{smallmatrix}
\right)^* 
Z_{i, 0, n} \left( \begin{smallmatrix}
 c_{i, 0, n}  \\
d_{i, 0, n}
\end{smallmatrix}
\right),  
\ \ \ \ \ [ J x_i J, T_j ]  - A_{i, j, n} = 
\left( \begin{smallmatrix}
 a_{i, j, n}  \\
b_{i, j, n}
\end{smallmatrix}
\right)^* 
Z_{i, j, n} \left( \begin{smallmatrix}
 c_{i, j, n}  \\
d_{i, j, n}
\end{smallmatrix}
\right).
\]

We let $\tilde \B \subset \B(L^2M)$ denote the separable ${\rm C}^*$-algebra generated by $\B$, together with the elements $A_{i, j, n}, Z_{i, j, n}, a_{i, j, n}, c_{i, j, n}$ for all $n \geq 1$, $1 \leq i \leq n$, and  $0 \leq j \leq n$. We let $I = \tilde \B \cap \K_{\X}(M)$ and note that $I$ is a separable ideal in $\tilde \B$. We may, therefore, choose an approximate unit $\{ e_n \}_n \subset I_+$, so that whenever $n \geq 1$, $1 \leq i \leq n$, and  $0 \leq j \leq n$ we have
\[
\| (e_{n + 1} - e_n)^{1/2} A_{i, j, n} (e_{n + 1} - e_n)^{1/2} \| < 2^{-n}, \ \ \ \ \ 
\| [e_n, x_i] \| < 2^{-n},
\]
\[
\| [(e_{n + 1} - e_n)^{1/2}, a_{i, j, n} ] \| < 2^{-n} \delta_{i, j, n}, \ \ \ \ \ 
 \| [(e_{n + 1} - e_n)^{1/2}, b_{i, j, n} ] \| < 2^{-n} \delta_{i, j, n},
\]
\[
\| [(e_{n + 1} - e_n)^{1/2}, c_{i, j, n} ] \| < 2^{-n} \delta_{i, j, n}, \ \ \ {\rm \ and \ } \ \ \
 \| [(e_{n + 1} - e_n)^{1/2}, d_{i, j, n} ] \| < 2^{-n} \delta_{i, j, n},
\]
where $\delta_{i, j, n} = (1 + \| a_{i, j, n} \|)^{-1} (1 + \| b_{i, j, n} \| )^{-1}(1 + \| c_{i, j, n} \| )^{-1}(1 + \| d_{i, j, n} \| )^{-1}$.

We define the u.c.p.\ map $\phi: \B \to \B(L^2M)$ by 
\[
\phi(T) = e_1 T e_1 +  \sum_{n = 1}^\infty (e_{n + 1} - e_{n})^{1/2} \phi_n(T) (e_{n + 1} - e_{n})^{1/2}.
\] 
If we denote by $\equiv$ equality modulo $\K_{\X}(M)$, it is then easy to see that for each $i \geq 1$ and $k \geq 1$ we have 
\begin{align}
\phi(x_i) - x_i 
& \equiv \sum_{n = 1}^\infty (e_{n + 1} - e_n)^{1/2} ( \phi_n(x_i) - x_i ) (e_{n + 1} - e_n)^{1/2} \nonumber \\
& \equiv \sum_{n = k}^\infty (e_{n + 1} - e_n)^{1/2} ( \phi_n(x_i) - x_i ) (e_{n + 1} - e_n)^{1/2} \nonumber \\
& \equiv \sum_{n = k}^\infty (e_{n + 1} - e_n)^{1/2} ( \phi_n(x_i) - x_i - A_{i, 0, n}) (e_{n + 1} - e_n)^{1/2} \nonumber \\
& \equiv \sum_{n = k}^\infty \left( \begin{smallmatrix}
 a_{i, 0, n}  \\
b_{i, 0, n}
\end{smallmatrix}
\right)^* 
\left( \begin{smallmatrix}
 (e_{n + 1} - e_n)^{1/2}  & 0 \\
0 & (e_{n + 1} - e_n)^{1/2}
\end{smallmatrix}
\right)
 Z_{i, 0, n} \left( \begin{smallmatrix}
 (e_{n + 1} - e_n)^{1/2}  & 0 \\
0 & (e_{n + 1} - e_n)^{1/2}
\end{smallmatrix}
\right) 
\left( \begin{smallmatrix}
 c_{i, 0, n}  \\
d_{i, 0, n}
\end{smallmatrix}
\right)
\nonumber
\end{align}
where the last equivalence uses $[b_{i,0,n},(e_{n+1}-e_n)^{1/2}]\in\K_{\X}(M)$, and similarly for $a_{i,0,n},c_{i,0,n},d_{i,0,n}$. Indeed, $(e_{n+1}-e_n)^{1/2}\in\K_{\X}(M)$, and $M,M'$ lie in its multiplier algebra. Heredity and the preceding commutator estimates show that the resulting differences have a norm-convergent sum in $\K_{\X}(M)$.

Since $\left\| \left( \begin{smallmatrix}
 (e_{n + 1} - e_n)^{1/2}  & 0 \\
0 & (e_{n + 1} - e_n)^{1/2}
\end{smallmatrix}
\right)
 Z_{i, 0, n} \left( \begin{smallmatrix}
 (e_{n + 1} - e_n)^{1/2}  & 0 \\
0 & (e_{n + 1} - e_n)^{1/2}
\end{smallmatrix}
\right)   \right\| \leq 1$, it then follows that 
\begin{align}
&d_{r_\mu}\left( \phi(x_i) - x_i, \K_\X^{\infty, 1}(M) \right) \nonumber \\
&\leq \limsup_{k \to \infty} \left( \sum_{n = k}^\infty  \mu(a_{i, 0, n}^* a_{i, 0, n})  + \mu(b_{i, 0, n}^* b_{i, 0, n}) \right)^{1/2}  \left( \sum_{n = k}^\infty  \mu(c_{i, 0, n}^* c_{i, 0, n} ) + \mu(d_{i, 0, n}^* d_{i, 0, n}) \right)^{1/2}
= 0. \nonumber
\end{align}
Since $k \geq 1$ was arbitrary, we then have $\phi(x_i) - x_i \in \K_\X^{\infty, 1}(M)$ for each $i \geq 1$. Since $\{ x_i \}_{i \geq 1}$ is dense in $M_0$, it follows that $\phi(x) - x \in \K_\X^{\infty, 1}(M)$ for each $x \in M_0$, and hence $\phi(x) - x \in \K_\X(M)$ for each $x \in M_0$ by Lemma~\ref{lem:cptinclusion}. A similar computation as above shows that $[J x J, \phi(T)] \in \K_\X^{\infty, 1}(M)$ for each $T \in \B$ and $x \in M_0$, and hence $\phi(\B) \subset \bS_\X(M)$.

To see the implications (\ref{item:biexactapprox}) $\implies$ (\ref{item:biexactapprox2}) and (\ref{item:biexact3}) $\implies$ (\ref{item:biexact2}), just note that, from Proposition~\ref{prop:sepnuc} and Corollary 3.1.6 in \cite{Iso13}, there exists a separable ultraweakly dense ${\rm C}^*$-subalgebra $M_0 \subset M$ so that the inclusion $M_0 \subset \B(L^2M)$ is $(M_0 \subset M)$-exact. If we take sequences of u.c.p.\ maps $\phi_n: M_0 \to \mathbb M_{k(n)}(\mathbb C)$ and $\psi_n: \mathbb M_{k(n)}(\mathbb C) \to \B(L^2M)$ that realize $(M_0 \subset M)$-exactness, then letting $\B$ be the (separable) ${\rm C}^*$-algebra generated by $M_0$ and $\cup_{n \geq 1} \psi_n(  \mathbb M_{k(n)}(\mathbb C) )$ we have that the inclusion $M_0 \subset \B$ is $(M_0 \subset M)$-nuclear. 

Suppose now that (\ref{item:biexact2}) holds and let $M_0 \subset M$, $\B \subset \B(L^2M)$, and $\phi: \B \to \bS_\X(M)$ satisfy the condition in (\ref{item:biexact2}). 

We let $p \in \K_\X(M)^{\sharp *} \subset \B(L^2M)^{\sharp *}$ denote the support projection of $\K_\X(M)^{\sharp *}$ so that, as an operator subsystem of the von Neumann algebra $\B(L^2M)^{\sharp *} = \begin{pmatrix}
{\K_\X(M)^{\sharp *}} & {p \B(L^2M)^{\sharp *} p^\perp} \\
{p^\perp \B(L^2M)^{\sharp *} p} & {p^\perp \B(L^2M)^{\sharp *} p^\perp} 
\end{pmatrix}$, we have $\bS_\X(M)^{\sharp *} =  \begin{pmatrix}
{\K_\X(M)^{\sharp *}} & {p \bS_X(M)^{\sharp *} p^\perp} \\
{p^\perp \bS_X(M)^{\sharp *} p} & {p^\perp \bS_X(M)^{\sharp *} p^\perp} 
\end{pmatrix}.$

We let $\phi_0: \B(L^2M) \to \K_\X(M)^{\sharp *}$ be a u.c.p.\ $M$-bimodular map coming from Lemma~\ref{lem:bimodularembedding}. Since $\phi(x) - x \in \K_\X(M)$ for all $x \in M_0$, it follows that $\tilde \phi: \B \to \bS_\X(M)^{\sharp *}$ defined by $\tilde \phi(T) = \phi_0(T) + p^\perp \phi(T) p^\perp$ is a u.c.p.\ $M_0$-bimodular map. Since the inclusion $M_0 \subset \B$ is $(M_0 \subset M)$-nuclear, it then follows that the inclusion $M_0 \subset \bS_\X(M)^{\sharp *}$ is also $(M_0 \subset M)$-nuclear, and hence the inclusion $M \subset \bS_\X(M)^{\sharp *}$ is $M$-nuclear by Corollary~\ref{cor:densenuclear}. Lemma~\ref{lem:weakly nuclear bidual} then shows that $M$ is biexact, thereby proving the implication (\ref{item:biexact2}) $\implies$ (\ref{item:biexact1}).
\end{proof}

\begin{lem}\label{lem:boundaryintersection}
Let $M$ be a separable von Neumann and let $\X_1, \X_2, \Y \subset \B(L^2M)$ be boundary pieces such that each $\X_k$ is generated as a hereditary ${\rm C}^*$-subalgebra of $\B(L^2M)$ by a ${\rm C}^*$-subalgebra $A_k \subset \X_k$. Suppose also that $M, M' \subset M(A_1) \cap M(A_2)$ and $A_1 A_2 A_1 \subset \Y$. If $M$ is biexact relative to both $\X_1$ and $\X_2$, then $M$ is biexact relative to $\Y$. 
\end{lem}

\begin{proof}
We fix a normal faithful state $\mu$ on $M$. First note that we may take quasi-central approximate units $\{ e_i^2 \}_i$ for $A_1$ and $\{ f_j^2 \}_j$ for $A_2$ that are also approximate units for $\X_1$ and $\X_2$, respectively.  For notational simplicity we will write $e_i^\perp = (1 - e_i^2)^{1/2}$ and $f_j^\perp = (1 - f_j^2)^{1/2}$.  Since $M, M' \subset M(A_1) \cap M(A_2)$ we have 
\[
\limsup_{i \to \infty} r_\mu(e_i S e_i) \leq r_\mu(S), \ \ \ \ \ \limsup_{i \to \infty} r_\mu( e_i^\perp S e_i^\perp ) \leq r_\mu(S),
\]
for $S \in \B(L^2M)$, and similarly we have $\limsup_{j \to \infty} r_\mu( f_j T f_j) \leq r_\mu(T)$ and $\limsup_{j \to \infty} r_\mu( f_j^\perp T f_j^\perp )$ for $T \in \B(L^2M)$.  Also, note that if $T \in \K_{\X_1}^{\infty, 1}(M)$, then for every $\varepsilon > 0$ there exists $S \in \X_1$ so that $r_{\mu}(T-S) < \varepsilon$. Since $\{ e_i^2 \}_i$ is an approximate unit for $\X_1$ we then have $\| e_i^\perp S e_i^\perp \| \to 0$, and hence $\limsup_{i \to \infty} r_\mu( e_i^\perp T e_i^\perp ) \leq \varepsilon$. Since $\varepsilon > 0$ was arbitrary, we have $\lim_{i \to \infty} r_\mu( e_i^\perp T e_i^\perp ) = 0$, and we similarly have $\lim_{j \to \infty} r_\mu( f_j^\perp T f_j^\perp ) = 0$ for all $T \in \K_{\X_2}^{\infty, 1}(M)$. 

We fix $\varepsilon > 0$ and $E \subset F \subset \B(L^2(M, \mu))$ finite subsets such that $E \subset M$. Since $M$ is biexact relative to both $\X_1$ and $\X_2$, there exist u.c.p.\ maps $\phi_k: \B(L^2(M, \mu)) \to \bS_{\X_k}(M)$ such that $r_\mu( \phi_k(x) - x ) < \varepsilon$ for each $x \in E$ and $k \in \{ 1, 2 \}$. From the discussion above, there then exists some $j$ so that we have 
\[
r_\mu( f_j^\perp (x - \phi_2(x) ) f_j^\perp ) < \varepsilon,   \ \ \ \ \
r_\mu( f_j^\perp [ \phi_2(T), JxJ] f_j^\perp ) < \varepsilon,
\]
\[ 
\| [f_j, x] \|, \| [f_j^\perp, x] \| < \varepsilon,
\] 
for all $x \in E$ and $T \in F$, and we may then choose some $i$ so that 
\[
r_\mu( e_i^\perp (x - \phi_1(x) ) e_i^\perp ) < \varepsilon,  \ \ \ \ \ \ \ 
r_\mu( e_i^\perp [\phi_1(T), JxJ] e_i^\perp ) < \varepsilon, 
\]
\[ 
r_\mu( e_i f_j^\perp [\phi_2(T), JxJ] f_j^\perp e_i ) < \varepsilon,  \ \ \ \ \ \ \ 
\| [e_i, x ] \|, \| [e_i^\perp, x] \| < \varepsilon,
\] 
for all $x \in E$ and $T \in F$. 

We define $\psi: \B(L^2M) \to \B(L^2M)$ by 
\[
\psi(T) 
= e_i^\perp \phi_1(T) e_i^\perp + e_i f_j^\perp \phi_2(T) f_j^\perp e_i + e_i f_j T f_j e_i. \nonumber \\
\]

For $x \in E$ we have
\begin{align}
r_\mu( x - \psi(x) )
&\leq r_\mu( e_i^\perp \phi_1(x) e_i^\perp - (e_i^\perp)^2 x ) 
+ r_\mu( e_i f_i^\perp \phi_2(x) f_i^\perp e_i -  e_i (f_i^\perp )^2 e_i  x )
+ r_\mu( e_i f_j x f_i e_j - e_i f_j^2 e_i x ) \nonumber \\
&< r_\mu( e_i^\perp ( \phi_1(x) - x ) e_i^\perp) + r_\mu( e_i f_j^\perp (\phi_2(x) - x) f_j^\perp e_i ) + 5 \varepsilon \| x \|  
< 2 \varepsilon + 5 \varepsilon \| x \|. \nonumber
\end{align}

Since $e_i f_j e_i \in A_1 A_2 A_1 \subset \Y$, we have $f_j e_i \in \B(L^2M) \Y \subset K_\Y^L(M)$ and so $e_i f_j \B(L^2M) f_j e_i \subset \K_{\Y}(M)$. Hence, for $T \in F$ and $x \in E$, we have
\begin{align}
d_{r_\mu}( [ \psi(T), JxJ ], \K_\Y^{\infty, 1}(M) ) 
& < r_\mu( [ e_i^\perp \phi_1(T) e_i^\perp, JxJ ] ) + r_\mu( [ e_i f_i^\perp \phi_2(T) f_i^\perp e_i, JxJ ] ) \nonumber \\
&\leq r_\mu( e_i^\perp [\phi_1(T), JxJ] e_i^\perp ) + r_\mu( e_i f_i^\perp [\phi_2(T), JxJ] f_i^\perp e_i ) + 6\| T \| \varepsilon \nonumber \\
&\leq 2 \varepsilon + 6 \| T \| \varepsilon. \nonumber
\end{align}

Part (\ref{item:biexactapprox}) of Theorem~\ref{thm:biexact} then gives that $M$ is biexact relative to $\Y$. 
\end{proof}

\begin{lem}\label{lem:expectationcontinuity}
Let $M$ be a von Neumann algebra and, for each $1 \leq i \leq n$, let $B_i \subset M$ be a von Neumann subalgebra with a normal conditional expectation $E_{B_i}: M \to B_i$. Let $e_{B_i}$ denote the corresponding Jones projection. Consider 
\[
x_1 Jy_1 J e_{B_{i_1}} x_2 J x_2 J e_{B_{i_2}} \cdots x_n Jy_nJ e_{B_{i_n}} \in \B(L^2M)
\] 
where $x_1, \ldots, x_n, y_1, \ldots, y_n \in M$. Then for each $1 \leq k \leq n$, this is continuous as a function of $x_k$ or of $y_k$ from the unit ball of $M$ with the ultrastrong-$^*$ topology to $\B(L^2M)$ with either of the locally convex topologies generated by $\{ r_\omega^\ell \}_{\omega \in (M_*)_+}$ or $\{ r_\omega^r \}_{\omega \in (M_*)_+}$, where $r_\omega^\ell$ and $r_\omega^r$ are defined as in Section~\ref{subsec:smallboundary}.
\end{lem}
\begin{proof}
We will show continuity in the variable $x_k$ with respect to the topology generated by $\{ r_\omega^\ell \}_{\omega \in (M^*)_+}$ as the cases for the other variable $y_k$, and the other topology follow similarly. We will prove the result by induction on $k$ (with $n \geq k$ arbitrary). First note that if $k = 1$, then the result if obvious. In general, if we set 
$x = x_1 Jy_1 J e_{B_{i_1}} x_2 J y_2 J e_{B_{i_2}} \cdots  e_{B_{i_k}} x_k$ and $y = J y_k J e_{B_{i_{k + 1}}} \cdots x_n Jy_nJ e_{B_{i_n}}$, then by Proposition~\ref{prop:leftequality} and Corollary~\ref{cor:productcont} for any positive normal linear functional $\mu$ on $M$ we have
\[
r_\mu^\ell( x_1 Jy_1 J e_{B_1} x_2 J y_2 J e_{B_2} \cdots x_n Jy_nJ e_{B_n} ) 
\leq r_\mu( x x^* )^{1/2} \| y \| 
\leq r_\mu^\ell(x x^*)^{1/2} \| y \|.
\]
But 
\[
xx^* = x_1 J y_1 J e_{B_{i_1}} \cdots e_{B_{i_{k-2}} } x_{k - 1}  E_{B_{i_{k-1}}}(x_k x_k^*)  J y_{k-1} J e_{B_{i_{k-1}}} J y_{k-1}^* J x_{k-1}^* e_{B_{i_{k -2}}} \cdots e_{B_{i_1}} J y_1^* J x_1^*, 
\]
and so, by the induction hypothesis, if $x_k$ is in the unit ball, then as $x_k$ approaches $0$ in the ultrastrong-$^*$ operator topology $s_\mu^\ell(x x^*)$ also approaches zero, proving the induction step.  
\end{proof}

The following theorem gives an analog of \cite[Proposition 15.2.7]{BO08} in the setting of von Neumann algebras. 

\begin{thm}\label{thm:biexactbypieces}
Let $M$ be a separable von Neumann algebra, let $\X, \Y \subset \B(L^2 M)$ be boundary pieces and let $\mathcal F = \{ (B_i, E_i) \}_{i \in I}$ be a family of von Neumann subalgebras $B_i \subset M$, with normal faithful conditional expectations $E_i: M \to B_i$. Suppose that there is an ultraweakly dense ${\rm C}^*$-subalgebra $M_0 \subset M$ such that $\B(L^2M) \X A \subset \K_\Y^L(M)$, where $A \subset \B(L^2M)$ is the ${\rm C}^*$-algebra generated by $\{ x J y J e_{B_i} \mid x, y \in M_0, i \in I \}$.  

If $M$ is biexact relative to both $\X$ and $\X_{\mathcal F}$, then $M$ is biexact relative to $\Y$. 
\end{thm}
\begin{proof}
If $S \in \B(L^2M)$, $T \in \X$, $x_1, x_2, \ldots, x_n, y_1, \ldots, y_n \in M_0$, and $B_{i_1}, \ldots, B_{i_n} \in \mathcal F$, then by hypothesis we have 
\[
ST x_1 Jy_1 J e_{B_{i_1}} x_2 J x_2 J e_{B_{i_2}} \cdots x_n Jy_nJ e_{B_{i_n}} \in \K_\Y^L(M).
\]
By Lemma~\ref{lem:expectationcontinuity}, it then follows that if $A_1$ denotes the ${\rm C}^*$-algebra generated by $M JMJ \X$, and $A_2$ denotes the ${\rm C}^*$-algebra generated by $\{ M JMJ e_{B_i} \mid i \in I \}$, then we still have $\B(L^2M) A_1 A_2 \subset \K_\Y^L(M)$, and hence $A_2 A_1 A_2 \subset \K_\Y(M)$ so that the result then follows from Lemma~\ref{lem:boundaryintersection}.
\end{proof}

\begin{cor}
Suppose $M$ is a separable von Neumann algebra and $P, Q \subset M$ are von Neumann subalgebras with expectation such that for an ultraweakly dense ${\rm C}^*$-subalgebra $M_0 \subset M$ we have $e_P x J y J e_Q \in \K(L^2M)$ for all $x, y \in M_0$. If $M$ is biexact relative to both $P$ and $Q$, then $M$ is biexact.
\end{cor}

\subsection{Biexactness and malleable deformations}

We now refine some of the techniques introduced in Section 9 of \cite{DKEP22} to show how biexactness can be deduced from the existence of Popa's malleable deformations.

\begin{thm}\label{thm:deformationbiexact}
Let $M$ be a separable weakly exact von Neumann algebra with a normal faithful state $\varphi$, and let $\X \subset \B(L^2M)$ be an $M$-boundary piece. Suppose $M \subset \tilde M$ is an extension with a normal faithful conditional expectation $E_M: \tilde M \to M$, and such that as a Hilbert $M$-bimodule $L^2(\tilde M, \varphi) \ominus L^2(M, \varphi)$ is weakly contained in coarse bimodule, where by abuse of notation we also let $\varphi$ denote the state $\varphi \circ E_M$ on $\tilde M$.  Suppose that we have a sequence of state-preserving automorphisms $\alpha_n \in {\rm Aut}(\tilde M, \psi)$ so that $\alpha_n(x) \to x$ ultraweakly for each $x \in M$. Let $e_M$ denote the orthogonal projection onto $L^2(M, \varphi)$, and let $V_n: L^2(\tilde M, \psi) \to L^2(\tilde M, \psi)$ denote the unitary given by $V_n(x \psi^{1/2}) = \alpha_n^{-1}(x) \psi^{1/2}$. 

Suppose also that for each $n \geq 1$, we have $e_M V_n e_M \in \K_\X^L(M, L^2(\tilde M, \psi))$, then $M$ is biexact relative to $\X$. 
\end{thm}
\begin{proof}
We define the u.c.p.\ map $\phi_n: \B(L^2 \tilde M) \to \B(L^2 M)$ by $\phi_n(T) = e_MV_n^* T V_n e_M$. Since ${\alpha_n}_{|M}$ converges to the identity pointwise in the ultraweak topology, it follows that for $x \in M$ we have that
\[
| [V_n, x] e_M |^2 = E_M(x^*x - \alpha_n(x^*) x - x^* \alpha_n(x) + \alpha_n(x^* x))
\] 
converges to $0$ in the ultraweak topology, so that $| [V_n, x] e_M |$ converges in the strong operator topology in $M$. Considering the polar decomposition $[V_n, x] e_M = W_n | [V_n, x] e_M |$, we then see that $[V_n, x] e_M$ converges in the $\mathbb C$-$M$-topology.  Hence, for any $T \in \B(L^2 \tilde M)$, we have $\phi_n(T)x - \phi_n(T x) = e_M V_n^* T [V_n, x] e_M$ converges to $0$ in the $\mathbb C$-$M$-topology, and, in particular, we have $r_\varphi( \phi_n(T) x - \phi_n(Tx) ) \to 0$. By similar arguments we then conclude that $r_\varphi( x\phi_n(T)y - \phi_n(xTy) ) \to 0$ for any $T \in \B(L^2 \tilde M)$ and $x, y \in M \cup JMJ$. 

Since $L^2(\tilde M, \varphi) \ominus L^2(M, \varphi)$ is weakly contained in the coarse $M$-bimodule, there exists an $M$-bimodular u.c.p.\ map $\psi: \B(L^2M) \to {M^{\rm op}}' \cap \B(L^2(\tilde M, \varphi) \ominus L^2(M, \varphi) )$. We may then define a u.c.p.\ map $\tilde \psi: \B(L^2M) \to \B(L^2 \tilde M)$ by $\tilde \psi(T) = e_MTe_M + \psi(T)$, where we identify $\B(L^2(\tilde M, \varphi) \ominus L^2(M, \varphi) )$ with $e_M^\perp \B(L^2(\tilde M, \varphi) ) e_M^\perp$.

Since $e_M V_n e_M \in \K_\X^L(M, L^2(\tilde M, \varphi))$, it then follows that for $x \in M$ we have 
\[
d_{r_\varphi}(\phi_n \circ \tilde \psi(x) - x, \K_\X^{\infty, 1}(M) ) \leq r_\varphi( \phi_n(x) - x ) \to 0.
\] 
Similarly, if $T \in \B(L^2M)$ and $x \in JMJ$, we have 
\[
d_{r_\varphi}( [\phi_n \circ \tilde \psi(T), x], \K_\X^{\infty, 1}(M) ) \leq r_\varphi([ \phi_n \circ \tilde \psi(T), x] - \phi_n( [x, \tilde \psi(T)] ) ) \to 0.
\] 
By Theorem~\ref{thm:biexact}, we then have that $M$ is biexact relative to $\X$.
\end{proof}

\begin{rem}
The assumptions that $L^2(\tilde M, \varphi) \ominus L^2(M, \varphi)$ be weakly contained in coarse bimodule and that $e_M V_n e_M \in \K_\X^L(M, L^2(\tilde M, \varphi))$ are only used to produce the existence of a c.c.p.\ map $\psi: \B(L^2M) \to \B(L^2(\tilde M, \varphi)) \cap JMJ'$ such that $(1 - \psi(x) )^{1/2} V_n e_M \in \K_\X^L(M, L^2 (\tilde M, \varphi))$ for each $n$. 
\end{rem}

\subsection{Biexactness and closable derivations}

In Section 9.2 of \cite{DKEP22}, a connection was established between proper proximality and the existence of certain closable real derivations into Hilbert bimodules. In this section, we refine the techniques from \cite{DKEP22} to show that under natural conditions on the derivation and bimodule one can deduce biexactness. This, then, formally shows how under the assumption of weak exactness one can obtain the results from \cite{Pet09} via the techniques from \cite{Oza04}. This should also be compared with Theorem 5.13 in \cite{CaIsWa21} where property (AO)$^+$ (and consequently biexactness) is obtained under additional assumptions on the growth of the eigenvalues of the associated Laplacian $\Delta = \delta^* \overline \delta$.

Let $M$ be a tracial von Neumann algebra and $\mathcal H$ an $M$-$M$ correspondence that has a real structure, i.e., there exists an antilinear isometric involution $\mathcal J: \mathcal H \to \mathcal H$ such that $\mathcal J( x \xi y ) = y^* \mathcal J(\xi) x^*$ for all $x, y \in M$, $\xi \in \mathcal H$. A closable real derivation is an unbounded closable linear map $\delta: L^2 M \to \mathcal H$, such that the domain $D(\delta)$ is an ultraweakly dense unital $*$-subalgebra of $M \subset L^2M$, and such that $\delta$ preserves the real structure ($\delta(x^*) = \mathcal J( \delta(x))$ for $x \in D(\delta)$) and satisfies Leibniz's formula
\[
\delta(xy) = x \delta(y) + \delta(x)y \ \ \ \ \ x, y \in D(\delta).
\]

A result of Davies and Lindsay in \cite{DaLi92} shows that $D(\overline{\delta}) \cap M$ is then again a $*$-subalgebra and $\overline \delta_{| D(\overline \delta) \cap M}$ again gives a closable real derivation. We recycle the following notation from \cite{Pet09, OzPo10} 
\[
\Delta = \delta^* \overline{ \delta},  \ \ \ \ \ \rho_\alpha = \frac{\alpha}{\alpha + \Delta}, \ \ \ \ \ \zeta_\alpha = \rho_\alpha^{1/2}, \ \ \ \ \ \tilde \delta_\alpha = \frac{1}{\sqrt{\alpha}} \overline \delta \circ \zeta_\alpha, \ \ \ \ \ \tilde \Delta_\alpha = \frac{1}{\sqrt{\alpha}} \Delta^{1/2} \circ \zeta_\alpha = ( 1 - \rho_\alpha )^{1/2}.
\] 
We note that from \cite{Sa90, Sa99} we have that $\zeta_\alpha$ is a $\tau$-symmetric u.c.p.\ map for $\alpha > 0$.

Building on ideas from \cite{Pet09} and \cite{OzPo10}, the following approximate bimodularity was established in Lemma 9.3 of \cite{DKEP22} for $x \in M \cap D(\overline \delta)$, $a \in M$, and $\alpha > 1$.
\begin{equation}\label{eq:A}
\| x \tilde \delta_\alpha(a) - \tilde \delta_\alpha(xa) \| \leq \alpha^{-1/4} (2 \| \delta(x) \|^{1/2} + 6\| x \|^{1/2} ) \| \delta(x) \|^{1/2} \| a \|. 
\end{equation}
\begin{equation}\label{eq:B}
\| \tilde \delta_\alpha(a)x - \tilde \delta_\alpha(ax) \| \leq \alpha^{-1/4} (2 \| \delta(x) \|^{1/2} + 6\| x \|^{1/2} ) \| \delta(x) \|^{1/2} \| a \|.  
\end{equation}

The following lemma is evident from \cite{Oza04}. 

\begin{lem}\label{lem:extend}
Let $M$ be a von Neumann algebra and $M_0 \subset M$ an ultraweakly dense ${\rm C}^*$-subalgebra containing the unit of $M$. Suppose that, for each finite-dimensional operator system $F \subset M$, there exists a net of u.c.p.\ maps $\theta_i: F \to M_0 \subset M$ converging pointwise ultraweakly to the identity operator. If $E$ is a dual normal $M$-system, then there exists a $M_0$-bimodular u.c.p.\ map $\Psi: \B(L^2M) \to E$ if and only if there exists an $M$-bimodular u.c.p.\ map $\Phi: \B(L^2M) \to E$. 
\end{lem}

\begin{thm}
Let $M$ be a weakly exact tracial von Neumann algebra and $\mathcal H$ an $M$-$M$ correspondence with a real structure so that $\mathcal H$ is weakly contained in the coarse correspondence $L^2M \ovt L^2M$, and suppose $\delta: L^2M \to \mathcal H$ is a closable real derivation. If $\X \subset \B(L^2M)$ is a boundary piece such that $\rho_\alpha \in \K_\X(M)$ for each $\alpha > 0$, then $M$ is biexact relative to $\X$.
\end{thm}
\begin{proof}
We consider the polar decomposition $\overline \delta = V \Delta^{1/2}$, and note that $V \tilde \Delta_\alpha = \tilde \delta_\alpha$ for $\alpha > 0$. 

From (\ref{eq:A}) and (\ref{eq:B}) we see that if $x \in D(\overline \delta) \cap M$, then we have 
\[
\lim_{\alpha \to \infty} \| x V \tilde \Delta_\alpha -V \tilde \Delta_\alpha x \|_{\B(M, \mathcal H)} = 0, \ \ \ \ \ \lim_{\alpha \to \infty} \| \mathcal J x \mathcal J V \tilde \Delta_\alpha -V \tilde \Delta_\alpha JxJ \|_{\B(M, \mathcal H)} = 0.
\]
Since $\rho_\alpha \in \K_\X(M)$, we then have $1 - \tilde \Delta_\alpha = 1 - ( 1 - \rho_\alpha )^{1/2} \in \K_\X(M)$. Hence, $V(1 - \tilde \Delta_\alpha) \in \K_\X(M, \mathcal H)$, and it follows that $xV - Vx, \mathcal J x \mathcal J V - V J x J \in \K_\X(M, \mathcal H)$ since this space is closed in $\B(M, \mathcal H)$.

The u.c.p.\ map $\Psi: \mathcal J M \mathcal J' \cap \B(\mathcal H) \to \B(L^2M)$ given by $\Psi(T) = V^*TV$ then satisfies $[\Psi(T), JxJ] \in \K_\X^{\infty, 1}(M)$ for $x \in D(\overline \delta) \cap M$, and hence $\Psi:\mathcal J M \mathcal J' \cap \B(\mathcal H) \to \bS_\X(M)$. We also have $\Psi(x) - x \in \K_\X(M)$ for $x \in D(\overline \delta) \cap M$. Hence, just as in the proof of Theorem~\ref{thm:condition AO+}, we consider the map $\Psi': \mathcal J M \mathcal J' \cap \B(\mathcal H) \to \bS_\X(M)^{\sharp *}$ given by $\Psi'(T) = p^\perp \Psi(T) + pT$ where $p$ is the support projection of $\X^{\sharp *} \subset \bS_\X(M)^{\sharp *}$, and note that this map is bimodular with respect to the ${\rm C}^*$-algebra $M_0 \subset M$ generated by $D(\overline \delta) \cap M$. 

Since $\mathcal H$ is weakly contained in the coarse correspondence, there exists a $M$-bimodular u.c.p.\ map $\theta: \B(L^2M) \to \mathcal J M \mathcal J' \cap \B(\mathcal H)$ so that $\Psi' \circ \theta$ gives an $M_0$-bimodular u.c.p.\ map from $\B(L^2M)$ into $\bS_\X(M)^{\sharp *}$. Since for each $\alpha > 0$ we have $\zeta_\alpha: M \to M_0$ and $\zeta_\alpha(x) \to x$ ultraweakly for each $x \in M$, it then follows from Lemma~\ref{lem:extend} that there exists an $M$-bimodular u.c.p.\ map from $\B(L^2M)$ to $\bS_\X(M)^{\sharp *}$. Since $M$ is weakly exact, Corollary~\ref{cor:characterization weakly exact} gives that $M\subset\B(L^2M)$ is $M$-nuclear. Composing with the $M$-bimodular u.c.p.\ map above, we obtain that $M\subset\bS_\X(M)^{\sharp *}$ is $M$-nuclear, hence weakly nuclear. By Lemma~\ref{lem:weakly nuclear bidual}, $M\subset\bS_\X(M)$ is therefore $M$-nuclear, so $M$ is biexact relative to $\X$. 
\end{proof}

\begin{rem}
Even without the weak exactness assumption, the previous theorem actually gives the following weak (AO) type property: There exist ultraweakly dense ${\rm C}^*$-subalgebras $B \subset M$ and $C \subset M'$ such that for each finite-dimensional operator system $E \subset M$ there exists a net of u.c.p.\ maps $\theta_i: E \to B \subset M$ converging pointwise ultraweakly to the identity operator, and such that there exists a u.c.p.\ map $\mu: B \otimes C \to \B(L^2M)$ satisfying $\mu(b \otimes c) - bc \in \K_\X(M)$ for each $b \in B$ and $c \in C$. Using Lemma~\ref{lem:extend} as in \cite{Oza04}, this can then give the solidity results from \cite{Pet09}.
\end{rem}

The proof of the previous theorem gives a more direct and more general result than in Proposition 9.4 from \cite{DKEP22}. This also gives the following analog in the von Neumann algebra setting of Example 4.8 from \cite{BIP21}. We leave the details of the proof to the reader. 

\begin{prop}If $M$ is a tracial von Neumann algebra, $\mathcal H$ is an $M$-$M$ correspondence with a real structure so that $(M^{\rm op})' \cap \B(\mathcal H)$ does not have an $M$-central state $\psi$ such that $\psi_M$ is normal, and $\delta: L^2M \to \mathcal H$ is a closable real derivation with $\rho_\alpha$ contained in the boundary piece $\X$ for each $\alpha > 0$, then $M$ is properly proximal relative to $\X$. 
\end{prop}

\subsection{Biexactness and Akemann-Ostrand type properties}

Recall from \cite{Oza04} that a von Neumann algebra $M \subset \B(\mathcal H)$ has property (AO) if there exists ultraweakly dense unital ${\rm C}^*$-subalgebras $B \subset M$ and $C \subset M'$ such that 
\begin{enumerate}
\item $B$ is locally reflexive;
\item the multiplication map $\nu: B \odot C\ni b\otimes c\mapsto bc+\K(\cH)\in \B(\cH)/\K(\cH)$ extends continuously to $B \otimes C$.
\end{enumerate}

The principal example of a property (AO) von Neumann algebra is the group von Neumann algebra $L\Gamma$ associated to a biexact group $\Gamma$, where one may take $B = C^*_\lambda \Gamma$ and $C = C^*_\rho \Gamma$. In fact, biexact groups can be characterized as exact groups such that the multiplication map $\nu: C^*_\lambda \Gamma \odot C^*_\rho \Gamma \to \B(\ell^2 \Gamma)/\K(\ell^2 \Gamma)$ extends continuously to $C^*_\lambda \Gamma \otimes C^*_\rho \Gamma$ and has a u.c.p.\ lift into $\B(\ell^2 \Gamma)$ \cite[Lemma 15.1.4]{BO08}.

Several variations on this definition have since appeared in the literature in connection with (strong) solidity type properties \cite{Iso15, HoIs17, Oza10, Cas22}. In this section, we will investigate the connection between these properties and biexactness. As there is extensive literature devoted toward von Neumann algebras with (AO)-type properties \cite{Oza04, Shl04, Oza06, VaVe06, GaJu07, Oza09, Av11, CS13, Iso15, HoRa15, HoIs17, Is19, Dep20, Ca21, CaIsWa21}, this will provide us with a rich source of examples of biexact von Neumann algebras. We begin by recalling some of these definitions. 

\begin{defn}[Definition 3.1.1 in \cite{Iso15}]\label{defn:AOplus}
A von Neumann algebra $M$ has property (AO)$^+$ if there exist unital weak$^*$ dense ${\rm C}^*$-algebras $B\subset M$, $C\subset JMJ$ such that
\begin{enumerate}
	\item $B$ is locally reflexive.
	\item The multiplication map $\nu: B \odot C\ni \sum_{i = 1}^n b_i \otimes c_i \mapsto \sum_{i = 1}^n b_i c_i + \K(M) \in \B(\cH)/\K(\cH)$ extends continuously to $B \otimes C$ and has a u.c.p.\ lifting.
\end{enumerate}
\end{defn}

Clearly, property (AO)$^+$ implies property (AO), but property (AO)$^+$ has better stability properties, and in certain situations the u.c.p.\ lifting can be exploited \cite{PV14}. Von Neumann algebras with property (AO)$^+$ include all von Neumann algebras associated to discrete quantum groups that are biexact in the sense of \cite[Definition 3.1]{Is15}.

\begin{defn}[Definition 2.6 in \cite{HoIs17}]\label{defn:strongAO}
A von Neumann algebra $M$ satisfies strong condition (AO) if there exists a unital weak$^*$ dense ${\rm C}^*$-algebra $A\subset M$, and if there exists a nuclear ${\rm C}^*$-algebra $\mathcal C \subset \B(L^2M)$ containing $A$ such that the commutators $\{ [c, JaJ] \mid a \in A, c \in \mathcal C \}$ are contained in $\K(L^2M)$. 
\end{defn}

The strong (AO) property implies property (AO) in general, and if the nuclear ${\rm C}^*$-algebra $\mathcal C$ above can be taken to be separable and satisfying $[\mathcal C, J \mathcal C J] \subset \K(L^2M)$, then this implies (AO)$^+$ \cite[Remarks 2.7]{HoIs17}.  Many von Neumann algebras are known to satisfy the strong (AO) property including separable amenable von Neumann algebras, any group von Neumann algebra $L\Gamma$ associated to a biexact group, the von Neumann algebra associated to any discrete quantum group in class $\mathcal C$ from \cite{Is17}, and any free Araki-Woods factor \cite[Theorem C.2]{HoIs17}.

Note that the ${\rm C}^*$-algebra $A$ in the definition of the strong (AO) condition is exact, being a ${\rm C}^*$-subalgebra of a nuclear ${\rm C}^*$-algebra, and consequently the von Neumann algebra $M$ is weakly exact since it contains an ultraweakly dense exact ${\rm C}^*$-subalgebra. 

\begin{defn}[Definition 2.1 in \cite{Cas22}]
A von Neumann algebra $M$ has (W$^*$AO) if the map 
\[
M \odot JMJ \ni \sum_{i = 1}^n a_i \otimes x_i \mapsto \sum_{i = 1}^n a_i x_i + \K(M) \in ( C^*(M, JMJ) + \K(M) )/\K(M)
\]
is continuous with respect to the minimal tensor norm.
\end{defn}

Ozawa showed in \cite{Oza10} that if $\Gamma$ is a biexact group, then $L\Gamma$ has (W$^*$AO). Whereas Caspers showed in \cite{Cas22} that the $q$-Gaussian von Neumann algebra $M_q( \mathcal H_{\mathbb R} )$ associated to an infinite-dimensional real Hilbert space $\mathcal H_{\mathbb R}$ does not have (W$^*$AO) when $-1 < q < 1$ with $q \not= 0$.

\begin{thm}\label{thm:condition AO+}
	Let $M$ be a weakly exact von Neumann algebra, then both conditions (AO)$^+$ and strong (AO) imply that $M$ is biexact.
\end{thm}
\begin{proof}
	We first suppose that $M$ is a weakly exact von Neumann algebra satisfying condition (AO)$^{+}$. Using the notation from Definition~\ref{defn:AOplus}, we denote by $\theta: B\otimes C\to \B(L^2M)$ the u.c.p.\ lifting of the multiplication map $\nu$ and extend $\theta$ to a u.c.p.\ map $\bar\theta:\B(L^2M)\otimes C\to \B(L^2M)$.
	Set $\bar\nu:=\pi\circ\bar\theta$, where $\pi: \B(L^2M)\to \B(L^2M)/\K(L^2M)$ is the canonical quotient map.
	
	Define $\Psi: \B(L^2M)\to \B(L^2M)$ by $\Psi(T)=\bar\theta(T\otimes 1)$ and note that for any $c\in C \subset JMJ$, 
	\[
	\pi([\bar\theta(T\otimes 1),c])=[\bar\nu(T\otimes 1), \bar\nu(1\otimes c)]=0,
	\]
	as $C$ is in the multiplicative domain of $\bar \nu$. Thus, the range of $\Psi$ is contained in $\bS(M)$. Consider $\Psi' :\B(L^2M)\to \bS(M)^{\sharp *}$ given by $\Psi'(T)=p^\perp \pi\circ\Psi(T) +p T$, where $p$ is the support projection of $\B(L^2M) \cong \K(L^2M)^{\sharp *}\subset \bS(M)^{\sharp *}$, and it is easy to check that $\Psi'\mid _B=\id_B$.
	
We may then apply Lemma~\ref{lem:extend} to obtain an $M$-bimodular u.c.p.\ map from $\B(L^2M)$ into $\bS(M)^{\sharp *}$, and, since $M$ is weakly exact, it then follows from Corollary~\ref{cor:characterization weakly exact} and Lemma~\ref{lem:weakly nuclear bidual} that $M$ is biexact.

We now suppose that $M$ satisfies strong condition (AO). Using the notation from Definition~\ref{defn:strongAO}, we note that as $[c, JaJ] \in \K(L^2M)$ for each $a \in A$, it follows that $\mathcal C \subset \bS(M)$. Since $\mathcal C$ is nuclear, we then have that the inclusion $A \subset \bS(M)$ is nuclear, and so $M$ is biexact by Corollary~\ref{cor:densenuclear}.
\end{proof}

\begin{cor}
For an exact group $\Gamma$, the notions of (AO)$^+$, strong (AO), and biexactness for $L\Gamma$ are equivalent, and coincide with biexactness for $\Gamma$.
\end{cor}

Concerning the connection between biexactness and (W$^*$AO), we have the following.

\begin{thm}\label{thm:biexact implies AO}
Let $M$ be a von Neumann algebra and $\X \subset \B(L^2M)$ an $M$-boundary piece.
If $M$ is biexact relative to $\X$, then the multiplication map 
\[
v: M\odot M^{\rm op}\ni \sum_{i = 1}^n a_i \otimes x_i \mapsto \sum_{i = 1}^n a_i x_i + \K_\X(M)\in {\rm C}^*(M, M^{\rm op}, \K_\X(M))/\K_\X(M)
\] 
is min-continuous.
In particular, biexact von Neumann algebras satisfy condition (W$^*$AO).
\end{thm}
\begin{proof}
Denote by $\iota: \B(L^2M)\to \B(L^2M)^{\sharp *}_J$ the canonical embedding.
Let $q_\X \in \B(L^2M)^{\sharp *}_J$ denote the support projection of $\K_\X(M)^{\sharp *}_J$ and set 
\[
\pi:= \Ad(q_\X^\perp ): \B(L^2M)^{\sharp *}_J \to q_\X^\perp \B(L^2M)^{\sharp *}_J q_\X^\perp.
\]
Notice that for any $T\in \bS_\X(M)$, we have $\pi \circ \iota (T)$ commutes with $\pi\circ \iota(JMJ)$.

Let $\phi_n$, $\psi_n$ be u.c.p.\ maps coming from the $M$-nuclear inclusion $M\subset \bS_\X(M)$. 
We may consider the following diagram
\[\small\begin{tikzcd}
M\otimes JMJ \arrow[dr, swap, "\phi_n\otimes \rm id"]& & q_\X^\perp \B(L^2M)^{\sharp *}_J q_\X^\perp,\\
&\mathbb M_{k(n)}(\C)\otimes JMJ \arrow[ur, swap, "{(\pi\circ \iota \circ \psi_n)\times (\pi\circ \iota)}"]&
\end{tikzcd}\]	
where 
\[ 
(\pi\circ \iota \circ \psi_n)\times (\pi\circ \iota): \mathbb M_{k(n)}(\C)\otimes JMJ \to  q_\X^\perp \B(L^2M)^{\sharp *}_J q_\X^\perp
\]
is justified since $[(\pi\circ \iota \circ \psi_n)(\M_{k(n)}(\C)), \pi\circ \iota(JMJ)]=0$ and $\M_{k(n)}(\C)$ is nuclear.
Denote by $\tilde v$ the u.c.p.\ map that is a point-weak$^*$ limit point of these maps, 
	and it is clear that $\tilde v(a\otimes x)= q_\X^\perp \iota(a)\iota(x)$
	for any $a\in M$ and $x\in JMJ$.

Moreover, notice that $\tilde v(M\otimes JMJ)\subset q_\X^\perp {\rm C}^*(\iota(M), \iota(JMJ))$ and 
	we claim that 
	$$q_\X^\perp {\rm C}^*(\iota(M), \iota(JMJ))\cong {\rm C}^*(M, JMJ)/I$$ 
	as ${\rm C}^*$-algebras, where $I=\K_\X(M)\cap  {\rm C}^*(M, JMJ)$.

To see this, first observe that $\iota_{\mid {\rm C}^*(M, JMJ)}$ implements a $*$-isomorphism between 
$$\B(L^2M)\supset {\rm C}^*(M, JMJ)\cong {\rm C}^*(\iota(M), \iota(JMJ))\subset \B(L^2M)^{\sharp *}_J.$$
Indeed, this follows from the fact that $\iota$ is a complete isometry and $\iota_{\mid \alg(M, JMJ)}$ is a $*$-homomorphism.
Next, note that $\theta :=\pi\circ \iota: {\rm C}^*(M, JMJ) \to q_\X^\perp {\rm C}^*(\iota(M), \iota(JMJ))$ is a $*$-homomorphism with 
	$$\ker \theta=\iota^{-1}(\ker (\pi) )\cap {\rm C}^*(M, JMJ).$$
Since $\ker (\pi) = \K_\X(M)^{\sharp *}_J$ and $\iota^{-1} (\K_\X(M)^{\sharp *}_J)=\K^{\infty,1}_\X (M)$, we have $\ker \theta= \K^{\infty,1}_\X (M)\cap  {\rm C}^*(M, JMJ)$.
Finally, if $x \in \ker \theta$, then $x^*x \in \ker \theta  \subset \K_\X^{\infty, 1}(M)_+ = \K_\X(M)_+$, and hence $x \in \K_\X^L(M)$. 
We similarly have $x^* \in \K_\X^L(M)$, and hence $x \in \K_\X(M)$,
	i.e., $\ker \theta = C^*(M, JMJ) \cap \K_\X(M) = I$.
If we denote by 
\[
\theta':q_\X^\perp {\rm C}^*(\iota(M), \iota(JMJ))\to {\rm C}^*(M, JMJ)/ I
\] the $*$-isomorphism,
	then the u.c.p.\ map
\[	
	v:= \theta' \circ \tilde v: M\otimes JMJ\to  {\rm C}^*(M, JMJ)/ I
\]	 
	satisfies $v(a\otimes x)=ax+ \ker \theta$ for $a \in M$ and $x \in JMJ$.
	The result then follows, since the inclusion map of ${\rm C}^*(M, JMJ)$ into ${\rm C}^*(M, M^{\rm op}, \K_\X(M))$ induces a ${\rm C}^*$-isomorphism from ${\rm C}^*(M, JMJ)/ I$ onto its image in ${\rm C}^*(M, M^{\rm op}, \K_\X(M))/\K_\X(M)$.
\end{proof}

By the previous theorem biexact von Neumann algebras will have the following consequence of condition (W$^*$AO).

\begin{thm}\label{thm:weakcontainultrapower}
Let $M$ be a $\sigma$-finite von Neumann algebra and consider the following conditions:
\begin{enumerate}
\item\label{item:ultraA} $M$ satisfies condition (W$^*$AO).
\item\label{item:ultraB} $M$ does not have property $(\Gamma)$, and if $\mathcal H$ is any normal Hilbert $M$-bimodule such that $\mathcal H \prec L^2M$, and such that $\mathcal H$ is disjoint from $L^2M$, then we have $\mathcal H \prec L^2 M \ovt L^2M$. 
\item\label{item:ultraC}  If $\mathcal U$ is a nonprincipal ultrafilter on $\mathbb N$, then we have the following weak containment of Hilbert $M$-bimodule 
\[
L^2(M^\mathcal U) \ominus L^2M \prec L^2 M \ovt L^2M.
\]
\end{enumerate}
Then, we have the implications (\ref{item:ultraA}) $\implies$ (\ref{item:ultraB}) $\implies$ (\ref{item:ultraC}). Moreover, if $M$ is finite, then all three conditions are equivalent.
\end{thm}
\begin{proof}
To see (\ref{item:ultraA}) $\implies$ (\ref{item:ultraB}), note that since $\cH\prec L^2M$, we have a $*$-homomorphism $\pi: {\rm C}^*(M, JMJ)\to \B(\cH)$
	satisfying $\pi(ab)= ab$ for $a\in M$ and $b \in JMJ$.
As $\cH$ is disjoint from $L^2M$, we have $\K(L^2M)\cap  {\rm C}^*(M, M^{\rm op})\subset \ker \pi$.
Note that for any $\xi \in \cH$, $\langle \pi(\cdot) \xi, \xi\rangle$ is a state that is normal when restricted to $M$ and $M^{\rm op}$,
	therefore we further have $\K(M)\cap  {\rm C}^*(M, M^{\rm op})\subset \ker \pi$.
The condition (W$^*$AO) then implies $\cH\prec L^2M\otimes L^2M$. Essentially the same argument is used to prove the implication (\ref{item:ultraA}) $\implies$ (\ref{item:ultraC}), and we note that  (\ref{item:ultraC}) easily implies that $M$ does not have property ($\Gamma$).

The implication (\ref{item:ultraB}) $\implies$ (\ref{item:ultraC}) is trivial.  

If $M$ is finite, then it is shown in \cite{Oza10} that $T \in \K^L(M)$ if and only if the $\| T \widehat {x_n} \|_2 \to 0$ for any uniformly bounded sequence $\{ x_n \}_n \subset M$ with $\widehat{x_n}$ converging weakly to $0$. It then follows that $\ker \pi = C^*(M, JMJ) \cap \K^L(M) = C^*(M, JMJ) \cap \K(M)$. If $\pi$ factors through the minimal tensor product, we then see that $M$ satisfies condition (W$^*$AO).
\end{proof}

\begin{rem}
In an initial version of this article, we presented only the implication (\ref{item:ultraA}) $\implies$ (\ref{item:ultraC}) in the previous theorem. We would like to thank Amine Marrakchi for explaining to us \cite[Corollary 3.5]{Mar23} where the implication  (\ref{item:ultraA}) $\implies$ (\ref{item:ultraB}) is shown for finite von Neumann algebras, and for suggesting to us that the implication (\ref{item:ultraA}) $\implies$ (\ref{item:ultraB}) might hold for general von Neumann algebras. 

We also remark that when $M$ has separable predual the requirement in condition (\ref{item:ultraB}) that $M$ does not have property ($\Gamma$) is superfluous (see Remark 3.3 in \cite{BaMaOz19}).
\end{rem}

Note that Theorems~\ref{thm:biexact implies AO} and \ref{thm:weakcontainultrapower} give a generalization of Theorem~\ref{thm:biexact implies solid}. Indeed, if $M$ is biexact, $A \subset M$ is a von Neumann subalgebra and $u \in A' \cap M^{\mathcal U}$ is a unitary with $E_M(u) = 0$, then fixing a normal faithful state $\varphi_M$ on $M$ and setting $\varphi = \varphi_M \circ E_M$ on $M^{\mathcal U}$, it then follows from Theorems~\ref{thm:biexact implies AO} and \ref{thm:weakcontainultrapower} that for $a_1, \ldots, a_m \in A$ and $b_1, \ldots, b_m \in M$ we have
\[
\| \sum_{i = 1}^m a_i J b_i J \varphi_M^{1/2} \|
= \| u \sum_{i = 1}^m a_i J b_i J \varphi^{1/2} \| 
= \| \sum_{i = 1}^m a_i J b_i^* J u \varphi^{1/2} \| 
\leq \| \sum_{i = 1}^m a_i \otimes b_i^{\rm op} \|_{A \otimes M^{\rm op}}.
\]
The inclusion $A \subset M$ is then weakly nuclear by \cite[Theorem 3.8.5]{BO08}.

\section{Solid factors that are not biexact}\label{sec:solidnotbiexact}

In this section we study biexactness for crossed-products of von Neumann algebras, and we give examples of solid factors that are not biexact. Suppose $\Gamma$ is a discrete group and we have an action $\Gamma \actson^\sigma M$ on a von Neumann algebra. We let $\sigma^0: \Gamma \to \mathcal U(L^2 M)$ denote the Koopman representation, and note that conjugation by the Koopman representation gives us an extension of the action to an action on the standard representation $\Gamma \actson^{{\rm Ad}(\sigma^0)} \B(L^2M)$. If $\X \subset \B(L^2M)$ is an $M$-boundary piece, then we say that $\X$ is $\Gamma$-invariant if it is invariant under this conjugation action. In this case the conjugation action also preserves the small-at-infinity boundary $\bS_\X(M)$. 

By a normal operator $(\Gamma \actson^\sigma M)$-system, we mean an operator system $E \subset \B(\mathcal H)$, together with a normal faithful representation of $M$ in $\B(\mathcal H)$, and a unitary representation $\sigma^0: \Gamma \to \mathcal U(\mathcal H)$ so that $E$ gives a normal operator $M$-system that is also invariant under conjugation by $\sigma^0(\Gamma)$, and such that the representations of $\Gamma$ and $M$ are covariant in the sense that $\sigma_t(x) = \sigma_t^0 x (\sigma_t^0)^*$ for each $x \in M$ and $t \in \Gamma$. So, for example, $\bS_\X(M) \subset \B(L^2M)$ gives a normal $(\Gamma \actson^\sigma M)$-system whenever $\X$ is $\Gamma$-invariant. If $E \subset \B(\mathcal H)$ is ultraweakly closed, then we say that this is a dual normal operator $(\Gamma \actson^\sigma M)$-system.

If $E$ is a normal operator $(\Gamma \actson^\sigma M)$-system and $\Gamma \actson A$ is an action on a unital ${\rm C}^*$-algebra, then, by considering a faithful unital covariant representation of $A$ into $\B(\mathcal K)$, we see that the diagonal action on $E \otimes A$ turns $E \otimes A$ into a normal operator $(\Gamma \actson^\sigma M)$-system. Also, if $E \subset \B(\mathcal H)$ is a normal operator $(\Gamma \actson^\sigma M)$-system, we then let $E \rtimes_r \Gamma$ denote the closed subspace of $\B(\mathcal H) \rtimes_r \Gamma \subset \B(\mathcal H \ovt \ell^2 \Gamma)$ spanned by elements of the form $(a \otimes 1) (\sigma^0_t \otimes \lambda_t)$ for $a \in E$, $t \in \Gamma$. Note that $E \rtimes_r \Gamma$ is, then, an operator $M \rtimes_r \Gamma$-system that is $( M \rtimes_r \Gamma \subset M \rtimes \Gamma )$-normal.

\begin{lem}\label{lem:amenaction}
Let $E$ denote a normal operator $(\Gamma \actson^\sigma M)$-system. If $\Gamma \actson^\alpha K$ is an amenable action on a compact Hausdorff space $K$, and if the inclusion $M \subset E$ is $M$-nuclear, then the inclusion $M \rtimes_r \Gamma \subset (E \otimes C(K)) \rtimes_r \Gamma$ is $(M \rtimes_r \Gamma \subset M \rtimes \Gamma)$-nuclear.
\end{lem}
\begin{proof}
The proof of Lemma 4.3.3 in \cite{BO08} shows that there exists a net of u.c.p.\ maps $\phi_i: ( E \otimes C(K) ) \rtimes_r \Gamma \to \mathbb M_{F_i}( E \otimes C(K) )$ and $\psi_i: \mathbb M_{F_i}( E \otimes C(K) ) \to ( E \otimes C(K) ) \rtimes_r \Gamma$, where $F_i \subset \Gamma$ are finite, such that $\psi_i \circ \phi_i$ converges to the identity in the point-norm topology.  Moreover, if we consider $\mathbb M_{F_i}(  E \otimes C(K) )$ to be an $M$-system via the twisted diagonal embedding $M \ni a \mapsto \oplus_{t \in F_i}  ( \sigma_{t^{-1}} \otimes {\rm id} ) (a) \in \mathbb M_{F_i}(  E \otimes C(K) )$, then the maps $\phi_i$ and $\psi_i$ above are $M$-bimodular, and hence continuous in the $M$-topology. 

Thus, the lemma follows from the fact that, for each finite set $F \subset \Gamma$, the inclusion $M \subset \mathbb M_{F}( E \otimes C(K) )$ is $M$-nuclear.
\end{proof}

The following gives an equivariant version of Proposition~\ref{prop:biexacttensor}. 

\begin{thm}
Suppose $\Gamma \actson M$ and $\X \subset \B(L^2 M)$ is a $\Gamma$-equivariant $M$-boundary piece. Suppose also that $I \subset \ell^\infty \Gamma$ is a boundary piece, that $\Gamma$ is biexact relative to $I$, and that $M$ is biexact relative to $\X$. Then, $M \rtimes \Gamma$ is biexact relative to the hereditary ${\rm C}^*$-algebra generated by $\ell^\infty(\Gamma; \X)$ and $\B(L^2M) \otimes I$. 
\end{thm}
\begin{proof}
We let $\Y \subset \B(L^2 M \ovt \ell^2 \Gamma)$ denote the hereditary ${\rm C}^*$-algebra generated by $\ell^\infty(\Gamma; \X)$ and $\B(L^2M) \otimes I$. By Lemma~\ref{lem:amenaction}, we have that the inclusion $M \rtimes_r \Gamma \subset ( \mathbb S_{\X}(M) \otimes \mathbb S_I(\Gamma) ) \rtimes_r \Gamma$ is $(M \rtimes_r \Gamma \subset M \rtimes \Gamma)$-nuclear, and so by Corollary~\ref{cor:densenuclear} it suffices to show that we have an inclusion
\begin{equation}\label{eq:inclusiongpaction}
( \mathbb S_{\X}(M) \otimes \mathbb S_I(\Gamma) ) \rtimes_r \Gamma
\subset \mathbb S_\Y(M \rtimes \Gamma).
\end{equation}
Since the latter space is an $L\Gamma$-bimodule, it then suffices to show that, for all $T \in \bS_\X(M)$ and $f \in \bS_I(\Gamma)$, we have $T \otimes f \in \bS_\Y(M \rtimes \Gamma)$. Letting $J$ denote the modular conjugation operator on $M \rtimes \Gamma$, and $J_M$ denote the modular conjugation operator on $M$, for $x \in M$ and $t \in \Gamma$ we may use (\ref{eq:modulargp}) to compute
\[
J u_t J (T \otimes f) J u_t^* J - (T \otimes f)
= T \otimes ( \rho_t f \rho_t^* - f ) \in \B(L^2 M) \otimes I,
\]
and
\[
[ T \otimes f, J x J ]
= \oplus_{t \in \Gamma} f(t) [T, J_M \sigma_t(x) J_M] \in \ell^\infty( \Gamma, \X).
\]
By \cite[Lemma 6.1]{DKEP22}, we then see that the inclusion (\ref{eq:inclusiongpaction}) holds. 
\end{proof}

We note that when $M$ or $\Gamma$ is amenable, then in the proof above we can replace $\bS_\X(M)$ with $M$, or $\bS_I(\Gamma)$ with $\mathbb C$, respectively. In the former case, we then need only to show the inclusion above for $T \otimes f$ with $T \in M$, while in the latter case we may assume $f \in \mathbb C$. This, then, gives the following two results.

\begin{prop}[Cf.\ Lemma 15.3.5 in \cite{BO08}]\label{prop:actiononamenable}
Suppose $\Gamma \actson M$, with $M$ amenable, and $I \subset \ell^\infty \Gamma$ is a boundary piece so that $\Gamma$ is biexact relative to $I$, then $M \rtimes \Gamma$ is biexact relative to the boundary piece generated by $\B(L^2 M) \otimes I$. In particular, if $M$ is amenable and $\Gamma$ is biexact, then $M \rtimes \Gamma$ is biexact relative to $M$. 
\end{prop}

\begin{prop}
Suppose $\Gamma \actson M$, with $\Gamma$ amenable, and $\X \subset \B(L^2M)$ is a $\Gamma$-invariant boundary piece so that $M$ is biexact relative to $\X$, then $M \rtimes \Gamma$ is biexact relative to the boundary piece generated by $\ell^\infty(\Gamma, \X)$. 
\end{prop}

We now briefly recall the $q$-Gaussian construction. Let $\mathcal H$ be a real Hilbert space, which we will always assume has dimension greater than 1. Let $\mathcal H_{\mathbb C} = \mathcal H \, \otimes_{\mathbb R} \, \mathbb C$ be its complexification and let $\mathcal F(\mathcal H_{\mathbb C}) = \mathbb C \Omega \oplus \bigoplus_{n \geq 1} \mathcal H_{\mathbb C}^{\otimes n}$ be the algebraic Fock space over $\mathcal H_{\mathbb C}$. Fix $-1 \leq q \leq 1$. Following \cite{BoSp91}, we consider the sequilinear form on $\mathcal F(\mathcal H_{\mathbb C})$ satisfying
\[
\langle \xi_1 \otimes \cdots \otimes \xi_n, \eta_1 \otimes \cdots \eta_m \rangle_q
= \delta_{n, m} \sum_{\sigma \in S_n} q^{i(\sigma)} \prod_{j} \langle \xi_j, \eta_{\sigma(j)} \rangle
\]
where $S_n$ is the symmetric group on $n$ characters, and $\iota(\sigma)$ denotes the number of inversions of $\sigma \in S_n$. This form is nonnegative definite, and we let $\mathcal F_q(\mathcal H_{\mathbb C})$ denote the Hilbert space obtained after separation and completion. We will abuse notation, and for $\xi \in \mathcal F(\mathcal H_{\mathbb C})$ we will continue to write $\xi$ for its image in $\mathcal F_q(\mathcal H_{\mathbb C})$. 

For $\xi \in \mathcal H$, we let $l_q(\xi)$ denote the left creation operator on $\mathcal F_q(\mathcal H_{\mathbb C})$, satisfying 
\[
l_q(\xi) \xi_1 \otimes \cdots \otimes \xi_n = \xi \otimes \xi_1 \otimes \cdots \otimes \xi_n.
\]
This operator is bounded if $-1 \leq q < 1$, and is a closed densely defined operator if $q = 1$. We let $s_q(\xi) = l_q(\xi) + l_q(\xi)^*$. 

If $-1 \leq q < 1$, the $q$-Gaussian ${\rm C}^*$-algebra is defined as the unital ${\rm C}^*$-subalgebra $A_q(\mathcal H)$ of $\B(\mathcal F_q(\mathcal H_{\mathbb C}))$ generated by $s(\xi)$ for $\xi \in \mathcal H$. The $q$-Gaussian von Neumann algebra $M_q(\mathcal H)$ is the von Neumann completion of $A_q(\mathcal H)$ in $\B(\mathcal F_q(\mathcal H_{\mathbb C}))$. When $q = 1$, the $q$-Gaussian von Neumann algebra $M_1(\mathcal H)$ is generated by the spectral projections of the self adjoint unbounded operators $s_1(\xi)$. 

The vacuum vector $\Omega$ is cyclic and separating for $M_q(\mathcal H)$ so that $\mathcal F_q( \mathcal H_{\mathbb C})$ may be identified with $L^2( M_q( \mathcal H) )$. Note that for $\xi \in \mathcal H$, we have $s(\xi) \Omega = \xi$, and if $q \not= 1$, it is then easy to see by induction on the maximal length of simple tensors that for a general vector $\xi \in \mathcal F (\mathcal H) \subset \mathcal F(\mathcal H_{\mathbb C})$ there still exists a bounded operator $s(\xi) \in M_q(\mathcal H)$ so that $s(\xi) \Omega = \xi$. The operator $s(\xi)$ is called the Wick operator associated to $\xi$. In the case when $q = 1$, the Wick operator $s(\xi)$ is a densely defined closed operator and contains the image of $\mathcal F(\mathcal H_{\mathbb C})$ in $\mathcal F_1(\mathcal H_{\mathbb C})$ in its domain. 

We define a linear idempotent $\xi \mapsto \xi^*$ on $\mathcal F(\mathcal H)$ by letting $(\xi_1 \otimes \cdots \otimes \xi_n)^* = \xi_n \otimes \cdots \otimes \xi_1$. Note that the Wick operators then satisfy $s(\xi)^* = s(\xi^*)$ for $\xi \in \mathcal F(\mathcal H)$. We also have 
\begin{equation}\label{eq:bimodular}
s( s(\xi_1) \xi_2) \Omega = s(\xi_1) s(\xi_2) \Omega \ \ \ \ \ s( J s(\eta_2) J \eta_1 ) \Omega = J s(\eta_2) J s(\eta_1) \Omega 
\end{equation}
for $\xi_1, \xi_2, \eta_1, \eta_2 \in \mathcal F(\mathcal H)$. The modular conjugation operator $J$ for $M_q(\mathcal H)$ satisfies $J (\xi ) = \xi^*$ for all $\xi \in \mathcal F(\mathcal H)$, or rather for $\xi$ in the image of $\mathcal F(\mathcal H)$ in $\mathcal F_q(\mathcal H_{\mathbb C})$. 

The $q$-Gaussian von Neumann algebra is abelian when $q = 1$, amenable when $q = -1$, a free group factor when $q = 0$, and a nonamenable II$_1$ factor when $-1 < q < 1$. Moreover, when $\dim(\mathcal H) < \infty$, it is known that $M_q(\mathcal H) \cong M_0(\mathcal H)$ for $|q|$ small enough (depending on $\dim(\mathcal H)$) \cite{GuSh14}.

We note that when $\dim(\mathcal H) < \infty$, then $M_q(\mathcal H)$ is biexact. Indeed, Shlyakhtenko noted in \cite[Section 4]{Shl04} that for $\dim(\mathcal H) < \infty$, if the ${\rm C}^*$-algebra generated by $l(\xi)$ for $\xi \in \mathcal H$ is nuclear, then $M_q(\mathcal H)$ has strong property (AO) (see also \cite{CaIsWa21}). Generalizing a result of Dykema and Nica \cite{DyNi93}, Kuzmin has recently shown in \cite{Ku22} that this ${\rm C}^*$-algebra is always nuclear for $-1 \leq q < 1$, and hence it then follows from Shlyakhtenko's result and Theorem~\ref{thm:condition AO+} that $M_q(\mathcal H)$ is biexact. 

If $\dim(\mathcal H ) = \infty$ and $q \not\in \{ -1, 0, 1 \}$, then from Theorem~\ref{thm:biexact implies AO} and Caspers' result \cite{Cas22} we see that $M_q(\mathcal H)$ is not biexact in this case. Indeed, using \cite[Theorem 2]{Nou04}, it is shown in \cite[Theorem 3.3]{BCKW22} that if $d \geq 1$ is such that $d > q^2 d > 1$, and $\mathcal H_0 \subset \mathcal H$ with $d = \dim( \mathcal H_0 ) < \dim(\mathcal H)$, then the Hilbert $M_q(\mathcal H_0)$-bimodule $L^2 (M_q(\mathcal H) ) \ominus L^2( M_q(\mathcal H_0) )$ is not weakly contained in the coarse bimodule. If $\mathcal U$ is a nonprincipal ultrafilter on $\mathbb N$, then by considering the natural embedding 
\[
L^2( M_q( \mathcal H^{\mathcal U} ) ) \ominus L^2 (M_q( \mathcal H ) ) 
\subset L^2 ( (M_q( \mathcal H) )^{\mathcal U} ) \ominus L^2( M_q( \mathcal H))
\]
we then see that $M_q(\mathcal H)$ fails to satisfy the conclusion of Theorem~\ref{thm:weakcontainultrapower} whenever $\dim(\mathcal H) = \infty$ and $q \not\in \{ -1, 0, 1 \}$.

If we have an isometry $V \in \B(\mathcal H, \mathcal K)$, then we obtain an isometry $V^{\mathcal F} \in \B(\mathcal F(\mathcal H_{\mathbb C}), \mathcal F(\mathcal K_{\mathbb C}))$, and conjugation by this isometry then gives us a normal embedding $M_q(\mathcal H) 	\hookrightarrow M_q(\mathcal K)$. The trace-preserving conditional expectation is then given by conjugation by the coisometry $(V^*)^{\mathcal F}$. 

In particular, if we have an orthogonal transformation $V \in \mathcal O(\mathcal H)$, then conjugation by $V^{\mathcal F}$ gives rise to a trace-preserving automorphism $\sigma_V \in {\rm Aut}(M_q(\mathcal H))$ that satisfies $\sigma_V( s(\xi) ) = s(V\xi)$ for $\xi \in \mathcal H$. If we have an orthogonal representation $\pi: \Gamma \to \mathcal O(\mathcal H)$, then the resulting action $\Gamma \actson^{\sigma_\pi} {\rm Aut}(M_q(\mathcal H))$ is called the $q$-Gaussian action associated to $\pi$, or simply the Gaussian action when $q = 1$.

If $\theta \in (0, \pi/2)$, then we let $V_\theta \in \mathcal O(\mathcal H \oplus \mathcal H)$ denote the isometry $V_\theta = \begin{pmatrix}
\cos \theta & - \sin \theta  \\
\sin \theta  & \cos \theta
\end{pmatrix}$, and we let $\alpha_\theta \in {\rm Aut}(M_q(\mathcal H \oplus \mathcal H))$ denote the corresponding automorphism. We will identify $M_q(\mathcal H)$ with $M_q(\mathcal H \oplus 0) \subset M_q(\mathcal H \oplus \mathcal H)$. Note that since $V_\theta^{\mathcal F}$ preserves each direct summand in the decomposition $\mathcal F(\mathcal H_{\mathbb C}) =  \mathbb C \Omega \oplus \bigoplus_{n \geq 1} \mathcal H_{\mathbb C}^{\otimes n}$, we may then explicitly compute
\begin{equation}\label{eq:deformationnorm}
e_{M_q(\mathcal H)} V_\theta^{\mathcal F} e_{M_q(\mathcal H)} 
= \sum_{n \geq 0} \cos^n \theta P_n
\end{equation}
where $P_n$ denotes the orthogonal projection onto $\mathcal H_{\mathbb C}^{\otimes n} \subset \mathcal F(\mathcal H_{\mathbb C})$. In particular, $e_{M_q(\mathcal H)} V_\theta^{\mathcal F} e_{M_q(\mathcal H)}$ is a positive operator on $\mathcal F(\mathcal H_{\mathbb C})$.

The equivalence between (\ref{item:convergeatom4}) and (\ref{item:convergeatom1})  in the following lemma is claimed in \cite{Av11}. It is easy to check in the case when $\dim(\mathcal H) < \infty$, which is the main case of interest in \cite{Av11}. Our proof for the general case is adapted from \cite[Lemma 2.4]{Io07}.

\begin{lem}\label{lem:convergeatomic}
Let $-1 \leq q \leq 1$, let $p \in \mathcal P(M_q(\mathcal H))$ be a nonzero projection, and suppose $B \subset p M_q(\mathcal H) p$ is a von Neumann subalgebra. The following conditions are equivalent:
\begin{enumerate}
\item \label{item:convergeatom4} $B$ is completely atomic. 
\item \label{item:convergeatom1} We have uniform convergence $\alpha_\theta \to {\rm id}$ in $\| \cdot \|_2$ on $\mathcal U(B)$ as $\theta \to 0$. 
\item \label{item:convergeatom2} For each nonzero projection $r \in \mathcal P(B' \cap p M_q(\mathcal H) p)$ and $\theta \in (0, \pi/2)$ we have 
\[
\inf_{u \in \mathcal U( B r) } \| E_{M_q(\mathcal H)} \circ \alpha_\theta( u) \|_2 > 0.
\] 
\item \label{item:convergeatom3} For each nonzero projection $r \in \mathcal P(B' \cap p M_q(\mathcal H) p)$ there exists $\theta \in (0, \pi/2)$, an orthogonal transformation $\gamma \in \mathcal O(\mathcal H)$, and a nonzero partial isometry $v \in \alpha_\theta(r) M_q(\mathcal H \oplus \mathcal H) \sigma_\gamma(r)$ such that $\alpha_\theta(b) v = v \sigma_\gamma(b)$ for all $b \in B$. 
\end{enumerate}
\end{lem}
\begin{proof}
The implications (\ref{item:convergeatom4}) $\implies$ (\ref{item:convergeatom1}) and  (\ref{item:convergeatom1}) $\implies$ (\ref{item:convergeatom2}) are trivial. If (\ref{item:convergeatom2}) holds, and if $r \in \mathcal P(B' \cap p M_q(\mathcal H) p)$ and $\theta \in (0, \pi/2)$ are such that 
\[
\inf_{u \in \mathcal U(B r) } \| E_{M_q(\mathcal H)} \circ \alpha_\theta( u) \|_2 > 0,
\] 
then by (\ref{eq:deformationnorm}) we have
\[
\inf_{u \in \mathcal U(Br)} \tau( \alpha_\theta(u) u^* ) > 0.
\]

We may then apply a ubiquitous convexity argument to get a nonzero partial isometry $v \in \alpha_\theta(r) M_q(\mathcal H \oplus \mathcal H) r$ so that $\alpha_\theta(b)v = v b$ for all $b \in B$. Indeed, if $d > 0$ is such that  $\tau( \alpha_\theta(u) u^* ) \geq d$ for all $u \in \mathcal U(B r)$, then the unique element $x$ of minimal $\| \cdot \|_2$ in the strongly closed convex hull of $\{  \alpha_\theta(u) u^* \mid u \in \mathcal U(Br) \} \subset \alpha_\theta(r) M_q(\mathcal H \oplus \mathcal H) r$ is nonzero (since $ \tau( x )  \geq d$) and satisfies $\alpha_\theta(u) x u^* = x$ for $u \in \mathcal U(B)$.  If $x$ has polar decomposition $x = v | x |$, then we have $v \in \alpha_\theta(r) M_q(\mathcal H \oplus \mathcal H) r$ is a nonzero partial isometry that satisfies $\alpha_\theta(b) v = v b$ for $b \in B$. This then shows that (\ref{item:convergeatom2}) $\implies$ (\ref{item:convergeatom3}), with $\gamma$ being the identity automorphism. 

To show (\ref{item:convergeatom3}) $\implies$ (\ref{item:convergeatom4}), we argue by way of contradiction and assume that (\ref{item:convergeatom3}) holds and that $B$ is not completely atomic. We may then choose a projection $z \in \mathcal Z(B)$ so that $zB$ is diffuse, and replacing $B$ with $zB$ we may assume that $z = p$. 

We let $\theta \in (0, \pi/2)$, $\gamma \in \mathcal O(\mathcal H)$, and $v \in \alpha_\theta(p) M_q(\mathcal H) \sigma_\gamma(p)$ be given as in (\ref{item:convergeatom3}). We now take $\varepsilon > 0$ to be chosen later and let $\mathcal H_0 \subset \mathcal H$ be a finite-dimensional subspace so that $v' = E_{M_q(\mathcal H_0 \oplus \mathcal H_0)}( v)$ satisfies $\| v' - v \|_2 < \varepsilon$. Set $r = \sigma_{\gamma^{-1}}(v^*v) \in B' \cap p M_q(\mathcal H)p$. Then, for $b \in (r B)_1$ we have
\[
\| \alpha_\theta(b) v' - v' \sigma_\gamma(b)\|_2 \leq 2 \varepsilon.
\]

Note that $\alpha_\theta(b) v' \in M_q(\mathcal K)=Q$, where $\mathcal K \subset \mathcal H \oplus \mathcal H$ is the closed subspace spanned by $V_\theta( (\mathcal H \ominus \mathcal H_0) \oplus 0 )$ and $\mathcal H_0 \oplus \mathcal H_0$. Indeed, $V_\theta(\mathcal H_0\oplus0)\subset\mathcal H_0\oplus\mathcal H_0$, and hence by definition of $\mathcal K$, $V_\theta(\mathcal H\oplus0)\subset\mathcal K$. Thus $\alpha_\theta(M_q(\mathcal H))\subset Q$, while $v'\in M_q(\mathcal H_0\oplus\mathcal H_0)\subset Q$, proving the assertion.

Since $v' \in Q$, we have $E_Q(v' \sigma_\gamma( b) ) = v' E_Q( \sigma_\gamma(b) )$ for $b \in r B$, and hence for all $b \in ( r B)_1$ we have
\[
\| v' E_{Q} ( \sigma_\gamma( b ) ) - v' \sigma_\gamma( b ) \|_2 \leq 4 \varepsilon.
\]
Since $v' \sigma_\gamma( b ) \in M_q(\mathcal H \oplus \mathcal H_0)$, by a simple computation we also have 
\begin{align}
\| v' E_Q ( \sigma_\gamma( b ) ) \|_2^2 
&\leq \cos \theta \| v' \sigma_\gamma( b) - E_{M_q(\mathcal H_0 \oplus \mathcal H_0)}(v' \sigma_\gamma( b ) ) \|_2^2 + \| E_{M_q(\mathcal H_0 \oplus \mathcal H_0)}(v' \sigma_\gamma( b ) ) \|_2^2 \nonumber \\
&= \cos \theta \| v' \sigma_\gamma( b ) \|_2^2 + (1 - \cos \theta) \| E_{M_q(\mathcal H_0 \oplus \mathcal H_0)}(v' \sigma_\gamma( b ) ) \|_2^2. \nonumber
\end{align}
Hence
\begin{align}
(1 - \cos \theta) \| v' E_{M_q(\mathcal H_0 \oplus \mathcal H_0)}( \sigma_\gamma( b ) ) \|_2^2
& \geq \| v' E_Q ( \sigma_\gamma( b ) ) \|_2^2 - \cos \theta \| v' \sigma_\gamma (b ) \|_2^2 \nonumber \\
&\geq (1 - \cos \theta) \| v' \sigma_\gamma (b ) \|_2^2 - (4 \varepsilon)^2. \nonumber
\end{align}

Choosing $\varepsilon > 0$ sufficiently small, we may then find $c > 0$ and a finite-dimensional subspace $\mathcal H_0 \subset \mathcal H$ so that 
\[
\| E_{M_q( \mathcal H_0)}(  \sigma_\gamma( u )  ) \|_2 = \| E_{M_q(\mathcal H_0 \oplus \mathcal H_0)}(\sigma_\gamma( u )) \|_2 \geq \| v E_{M_q(\mathcal H_0 \oplus \mathcal H_0)}( \sigma_\gamma( u ) ) \|_2 \geq c
\]
for all $u \in \mathcal U(r B)$. 

By Popa's Intertwining Theorem \cite[Theorem 2.1]{Po06B}, we then have $\sigma_\gamma(r B)  \preceq_{M_q(\mathcal H)} M_q( \mathcal H_0)$, i.e., there exist projections $e \in r B$, $f \in M_q( \mathcal H_0)$, a nonzero partial isometry $w \in e M_q(\mathcal H) f$ and a unital normal $*$-homomorphism $\phi: \sigma_\gamma( e B  e ) \to f M_q(\mathcal H_0) f$ such that $ b w = w \phi(  b  )$ for all $b \in  \sigma_\gamma(eBe)$. 

We then have that $u = \alpha_\theta( \sigma_{\gamma^{-1}}(w^*) ) v w$ is a nonzero partial isometry with $u^* u \leq f$ and satisfying $\alpha_\theta( x ) u = u \sigma_\gamma(x)$ for $x \in \tilde B := \sigma_{\gamma^{-1}}( \phi( \sigma_\gamma( eBe ) ) ) \subset M_q(\gamma^{-1} \mathcal H_0)$. If $\delta > 0$, then we may take $\mathcal H_1 \subset \mathcal H$ finite-dimensional with $\gamma^{-1}\mathcal H_0 \subset \mathcal H_1$ so that $\| u - E_{M_q(\mathcal H_1)}(u) \|_2 < \delta$, and it then follows as above that
\[
\| E_{M_q(\mathcal H_1)} \circ \alpha_\theta (x) u - u \sigma_\gamma( x )  \|_2  = 
\| E_{M_q(\mathcal H_1)} \circ \alpha_\theta (x) u - u E_{M_q(\mathcal H_1)}( \sigma_\gamma(x) ) \|_2 < 4 \delta
\]
for all $x \in (\tilde B)_1$. Since $\mathcal H_1$ is finite-dimensional we have that $E_{M_q(\mathcal H_1)} \circ \alpha_\theta$ is compact as an operator on $\mathcal F( \mathcal H_{\mathbb C})$, and hence, if $\{ x_k \}_k \in \tilde B$ is any uniformly bounded sequence that converges weakly to $0$, we have $\limsup_{k \to \infty} \| u x_k \|_2 < 4 \delta$. Since $\delta > 0$ was arbitrary, we then have that $\tilde B$ is not diffuse, and hence, neither is $B$, since it has a corner that is isomorphic to $\tilde B$.
\end{proof}

We continue to let $V_\theta = \begin{pmatrix}
\cos \theta & - \sin \theta  \\
\sin \theta  & \cos \theta
\end{pmatrix}$, and let $\alpha_\theta \in {\rm Aut}(M_q(\mathcal H \oplus \mathcal H))$ be the associated automorphism as defined above. Note that if we have an orthogonal representation $\pi: \Gamma \to \mathcal O(\mathcal H)$, then $\alpha_\theta$ is $\Gamma$-equivariant with respect to the $q$-Gaussian action associated to $\pi \oplus \pi$ and so we may extend it to an automorphism (again denoted by $\alpha_\theta$) of $\tilde M = M_q(\mathcal H \oplus \mathcal H) \rtimes^{\sigma_{\pi \oplus \pi}} \Gamma$ so that $\alpha_\theta$ is the identity map on $L\Gamma \subset \tilde M$. 

We may bootstrap the previous lemma to obtain the following version adapted to the setting of crossed-products.

\begin{lem}\label{lem:convegeatomic2}
Let $-1 \leq q \leq 1$, and $\pi: \Gamma \to \mathcal O(\mathcal H)$ be an orthogonal representation. Set $M = M_q(\mathcal H) \rtimes^{\sigma_\pi} \Gamma$, and let $p \in \mathcal P(M)$ be a nonzero projection. Suppose $B \subset p M p$ is a von Neumann subalgebra such that $B r \preceq_M M_q(\mathcal H)$ for each nonzero projection $r \in B' \cap p M p$. The following conditions are equivalent:
\begin{enumerate}
\item \label{item:convergeatomB4} $B$ is completely atomic. 
\item \label{item:convergeatomB1} We have uniform convergence $\alpha_\theta \to {\rm id}$ in $\| \cdot \|_2$ on $\mathcal U(B)$ as $\theta \to 0$. 
\item \label{item:convergeatomB2} For each nonzero projection $r \in \mathcal P(B' \cap p M p)$ and $\theta \in (0, \pi/2)$, we have 
\[
\inf_{u \in \mathcal U( B r) } \| E_{M} \circ \alpha_\theta( u) \|_2 > 0.
\] 
\item \label{item:convergeatomB3} For each nonzero projection $r \in \mathcal P(B' \cap p M p)$, there exists $\theta \in (0, \pi/2)$, and a nonzero partial isometry $v \in \alpha_\theta(r) \tilde M r$ such that $\alpha_\theta(b) v = v b$ for all $b \in B$. 
\end{enumerate}
\end{lem}
\begin{proof}
The proofs of the implications (\ref{item:convergeatomB4}) $\implies$ (\ref{item:convergeatomB1}) and (\ref{item:convergeatomB1}) $\implies$ (\ref{item:convergeatomB2}) are trivial. Moreover, the same proof as in Lemma~\ref{lem:convergeatomic} shows (\ref{item:convergeatomB2}) $\implies$ (\ref{item:convergeatomB3}).

We now suppose that (\ref{item:convergeatomB3}) holds, and, by way of contradiction, we may restrict to the diffuse corner of $B$ and assume that $B$ itself is diffuse. Let $r \in \mathcal P( B' \cap pMp)$, $\theta \in (0, \pi/2)$, and $v \in \alpha_\theta(r) \tilde M r$ be as above. Since $B v^*v \preceq_M M_q(\mathcal H)$, we may find projections $e \in Bv^* v$, $f \in M_q(\mathcal H)$, a nonzero partial isometry $w \in e M f$, and a unital normal $*$-homomorphism $\phi: e B e \to f M_q(\mathcal H) f$ so that $b w = w \phi(b)$ for all $b \in eBe$. 

Then $u = \alpha_\theta(w^*) v w$ is a nonzero partial isometry with $u^*u \leq f$ so that $\alpha_\theta( x ) u = u \phi(x)$ for all $x \in  \phi(eBe) \subset M_q(\mathcal H)$. If we take the Fourier representation $u = \sum_{t \in \Gamma} a_t u_t$ with $a_t \in M_q(\mathcal H \oplus \mathcal H)$, then by uniqueness of the Fourier representation we have $\alpha_\theta(x) a_t = a_t \sigma_{\pi(t)}( x )$ for all $x \in \phi(e B e)$ and $t \in \Gamma$. It then follows from Lemma~\ref{lem:convergeatomic} that $\phi(eBe)$ is not diffuse, and hence, we conclude that $B$ is also not diffuse. 
\end{proof}

The previous lemmas are of interest even in the case of the classical Gaussian actions when $q = 1$, where they can be used to give the following relatively simple proof of Boutonnet's result \cite{Bo12} showing solid ergodicity for Gaussian actions associated to representations $\pi$ having the property that $\pi^{\otimes k} \prec \lambda$ for some $k \geq 1$. Moreover, like in \cite[Section 9]{DKEP22}, this approach avoids any hypothesis of mixingness for the representation.

\begin{thm}\label{thm:solidergodic}
Let $\pi: \Gamma \to \mathcal O(\mathcal H)$ be an orthogonal representation such that $\pi^{\otimes k} \prec \lambda$ for some $k \geq 1$. Set $A = M_1(\mathcal H)$ and $M = A \rtimes^{\sigma_\pi} \Gamma$. If $p \in \mathcal P(M)$ is a nonzero projection and $B \subset pMp$ is a diffuse von Neumann subalgebra such that $B r \preceq_M A$ for any nonzero projection $r \in B' \cap pMp$, then $B' \cap pMp$ is amenable.
\end{thm}

\begin{proof}
Since $B$ is diffuse and $B r \preceq_M A$ for any nonzero projection $r \in B' \cap pMp$, it follows from Lemma~\ref{lem:convegeatomic2} that there exist sequences $\{ \theta_n \}_n \subset (0, \pi/2)$ with $\theta_n \to 0$ and $\{ u_n \}_n \subset \mathcal U(B)$ such that $\| E_M \circ \alpha_{\theta_n}( u_n) \|_2 \to 0$. Setting $\xi_n = \alpha_{\theta_n}(u_n) - E_M \circ \alpha_{\theta_n}( u_n)$, we then have that $\{ \xi_n \}_n \subset L^2\tilde M \ominus L^2M$ defines a sequence of asymptotically left and right tracial vectors that are also asymptotically $B'\cap pMp$-central. Taking the $k$-fold tensor product then gives a sequence $\{ \xi_n \otimes_M \cdots \otimes_M \xi_n \}_n \subset ( L^2 \tilde M \ominus L^2 M)^{\otimes_M k}$, which is also asymptotically left and right tracial and asymptotically $B$-central. 

Since $\pi^{\otimes k} \prec \lambda$, the Hilbert $M$-bimodule $( L^2 \tilde M \ominus L^2 M)^{\otimes_M k}$ is weakly contained in the coarse bimodule \cite[Lemma 3.3]{Bo12}, and it therefore follows that $B' \cap pMp$ is amenable. 
 \end{proof}

\begin{cor}\label{cor:solidgaussian}
Let $\Gamma$ be biexact and let $\pi: \Gamma \to \mathcal O(\mathcal H)$ be an orthogonal representation such that $\pi^{\otimes k} \prec \lambda$ for some $k \geq 1$, then $M_1(\mathcal H) \rtimes^{\sigma_\pi} \Gamma$ is solid.
\end{cor}
\begin{proof}
We let $B \subset M$ be a diffuse von Neumann subalgebra. Set $M = M_1(\mathcal H) \rtimes^{\sigma_\pi} \Gamma$ and $A = M_1(\mathcal H)$. We let $p \in B' \cap M$ denote the maximal projection so that $Bpr \preceq_M A$ for all nonzero projections $r \in (Bp)' \cap p M p$ \cite[Proposition 1.1]{ChPe13}. By Theorem~\ref{thm:solidergodic}, we have that $(B p)' \cap M$ is amenable. By Proposition~\ref{prop:actiononamenable}, $M$ is biexact relative to $A$, and so it follows from Proposition~\ref{prop:relative biexact intertwine} that $(Bp^\perp)' \cap M$ is also amenable. Hence, $B' \cap M$ is amenable. 
\end{proof}

The previous corollary and the following result, applied to Bekka's examples \cite[Proposition 2.9]{Bo12}, give solid von Neumann algebras that are not biexact.

\begin{thm}
Let $\pi: \Gamma \to \mathcal O(\mathcal H)$ be an orthogonal representation such that $\pi \not\prec \lambda$, then $M = M_1(\mathcal H \ovt \ell^2 \mathbb N) \rtimes^{\sigma_{\pi \otimes 1}} \Gamma$ is not biexact relative to $L\Gamma$. 
\end{thm}
\begin{proof}
The proof is similar to the proof of Theorem~\ref{thm:weakcontainultrapower}. 
Note that, for $t \in \Gamma$, we have $u_t J u_t J (\xi \otimes \delta_n \otimes \delta_e ) = \pi(t)\xi \otimes \delta_n \otimes \delta_e$, and since $\pi \not\prec \lambda$ there then exists a unit vector $\eta \in \mathcal H$ and a finitely supported function $a = \sum_{t \in \Gamma} \alpha_t t \in \mathbb C \Gamma$ so that 
\[
\| ( \sum_{t \in \Gamma} \alpha_t u_t J u_t J ) ( \eta \otimes \delta_n \otimes \delta_e  ) \| 
> \| \sum_{t \in \Gamma} \alpha_t \lambda_t \| 
= \| \sum_{t \in \Gamma} \alpha_t u_t \otimes J u_t J \|.
\]

We fix a nonprincipal ultrafilter $\mathcal U$ on $\mathbb N$ and define a state $\varphi$ on $\B(L^2 M)$ by 
\[
\varphi(T) = \lim_{n \to \mathcal U} \langle T (\eta \otimes \delta_n \otimes \delta_e), (\eta \otimes \delta_n \otimes \delta_e) \rangle.
\] 
It is then easy to see that $\varphi_{|M}$ and $\varphi_{|JMJ}$ are both normal, and $\varphi(T) = 0$ for any $T \in \K( \mathcal F(\mathcal H \ovt \ell^2 \mathbb N ) ) \otimes \B(\ell^2 \Gamma)$. Hence, we have $\varphi_{|  \K_{\X_{L\Gamma}}^{\infty, 1}(M)} = 0$, and so $\varphi$ defines a state on
\[
 C^*(M,JMJ,\K_{\X_{L\Gamma}}(M))/\K_{\X_{L\Gamma}}(M).
\]
If $x = \sum_{t \in \Gamma} \alpha_t u_t J u_t J$, then from above we have 
\[
\| x \|^2 \geq \varphi( x^* x ) > \| \sum_{t \in \Gamma} \alpha_t u_t \otimes J u_t J \|^2,
\]
so that the map 
\[
\begin{aligned}
 C^*_\lambda \Gamma \odot J C^*_\lambda J
 &\longrightarrow C^*(M, JMJ, \K_{\X_{L\Gamma}}(M))/\K_{\X_{L\Gamma}}(M),\\
 \sum_{i = 1}^m b_i \otimes c_i
 &\longmapsto \sum_{i = 1}^m b_i c_i + \K_{\X_{L\Gamma}}(M).
\end{aligned}
\] 
is not min-continuous and hence $M$ is not biexact relative to $L\Gamma$ by Theorem~\ref{thm:biexact implies AO}.
\end{proof}

\subsection{Strong Solidity of $q$-Gaussian von Neumann algebras for infinite variables}

If $M$ and $N$ are tracial von Neumann algebras and $\mathcal H$ is a normal Hilbert $M$-$N$ bimodule, then a vector $\xi \in \mathcal H$ is left-bounded if the map $L_\xi: N \to \mathcal H$ defined by $L_\xi(x) = \xi x$ is bounded when $N$ is endowed with the norm $\| \cdot \|_2$. We may then view $L_\xi$ as an operator in $\B(L^2N, \mathcal H)$. We let $\mathcal H^0$ denote the space of left-bounded vectors. Given two left-bounded vectors $\xi, \eta \in \mathcal H$ we may check that $L_\eta^* L_\xi \in JNJ' \cap \B(L^2N) = N$. If $Q$ is another tracial von Neumann algebra and $\mathcal K$ is a normal Hilbert $N$-$Q$ bimodule, then we have a non-negative definite sequilinear form on $\mathcal H^0 \otimes_{\rm alg} \mathcal K$, satisfying
\[
\langle \xi \otimes \eta, \xi' \otimes \eta' \rangle
= \langle ( L_{\xi'}^* L_\xi  ) \eta, \eta' \rangle
\]
for $\xi, \xi \in \mathcal H^0$ and $\eta, \eta' \in \mathcal K$. The Connes fusion on $\mathcal H$ and $\mathcal K$ over $N$ is the separation and completion of $\mathcal H^0 \otimes_{\rm alg} \mathcal K$ with respect to this sesquilinear form, and is denoted by $\mathcal H \oovt{N} \mathcal K$ (see \cite{Po86}). We denote by $\xi \otimes_N \eta$ the image of $\xi \otimes \eta$ in $\mathcal H \oovt{N} \mathcal K$. This is, then, a normal Hilbert $M$-$Q$ bimodule satisfying 
\[
a (\xi \otimes_N \eta) b = (a \xi) \otimes_N (\eta b) 
\]
for $a \in M$, $b \in Q$, $\xi \in \mathcal H^0$ and $\eta \in \mathcal K$. 

Note that if $\mathcal H$ is a real Hilbert space and $-1 \leq q < 1$, then each vector $\xi \in \mathcal F(\mathcal H)$ defines a left-bounded vector in the trivial $M_q(\mathcal H)$ bimodule $L^2( M_q(\mathcal H) ) \cong \mathcal F_q( \mathcal H_{\mathbb C})$, and in this case we have $L_\xi \hat x = J x^* J \xi  = s(\xi) \hat{x}$ so that $L_\xi = s(\xi)$. If $B \subset M_q(\mathcal H)$ is a von Neumann subalgebra and we view $\mathcal F(\mathcal H)$ as a left Hilbert $B$ module, then we have $L_\eta^* L_\xi = E_B( s(\eta)^*s(\xi))$ for all $\xi, \eta \in \mathcal F(\mathcal H)$.

\begin{prop}\label{prop:grading}
Let $-1 \leq q < 1$ and let $\mathcal H_1, \mathcal H_2$, and $\mathcal K$ be real Hilbert spaces, then we have an $M_q(\mathcal H_1 \oplus \mathcal K)$-$M_q(\mathcal K \oplus \mathcal H_2)$ bimodular isometry 
\[
\Xi:  L^2( M_q( \mathcal H_1 \oplus \mathcal K)) \oovt{M_q(\mathcal K)} L^2( M_q(\mathcal K \oplus \mathcal H_2) ) \to L^2 (M_q( \mathcal H_1 \oplus \mathcal K \oplus \mathcal H_2) )
\] 
satisfying 
\[
\Xi( \xi \otimes_{M_q(\mathcal K)} \eta) = s(\xi \oplus 0) s(0 \oplus \eta) \Omega
\]
for all $\xi \in \mathcal F(\mathcal H_1 \oplus \mathcal K)$ and $\eta \in \mathcal F(\mathcal K \oplus \mathcal H_2)$. 
\end{prop}
\begin{proof}
We define the real-linear map 
\[
\Xi_0: \mathcal F(\mathcal H_1 \oplus \mathcal K) \otimes_{\rm alg} \mathcal F(\mathcal K \oplus \mathcal H_2) \to L^2 (M_q( \mathcal H_1 \oplus \mathcal K \oplus \mathcal H_2) )
\] 
by setting $\Xi_0( \xi \otimes \eta ) = s(\xi \oplus 0 ) s(0 \oplus \eta)  \Omega$ for $\xi \in \mathcal F(\mathcal H_1 \oplus \mathcal K)$ and $\eta \in \mathcal F(\mathcal K \oplus \mathcal H_2)$. Using that $s(\eta) \Omega = J s(\eta^*) J \Omega$, we then check that 
\begin{align}
\langle \Xi_0( \xi_1 \otimes \eta_1 ), \Xi_0( \xi_2 \otimes \eta_2) \rangle
&= \langle s(\xi_2^* \oplus 0) s(\xi_1 \oplus 0) s(0 \oplus \eta_1)\Omega, s(0 \oplus \eta_2) \Omega \rangle \nonumber \\
&= \langle E_{M_q(0 \oplus \mathcal K \oplus 0)}( s(\xi_2^* \oplus 0) s(\xi_1 \oplus 0) ) s(0 \oplus \eta_1)\Omega, s(0 \oplus \eta_2) \Omega \rangle \nonumber \\
&= \langle \xi_1 \otimes \eta_1, \xi_2 \otimes \eta_2 \rangle, \nonumber
\end{align}
where the last inner-product is taken in ${L^2( M_q( \mathcal H_1 \oplus \mathcal K)) \oovt{M_q(\mathcal K)} L^2( M_q(\mathcal K \oplus \mathcal H_2) ) }$.
It then follows that $\Xi$ is a well-defined isometry, and it is easy to see that $\Xi$ is $M_q(\mathcal H_1 \oplus \mathcal K)$-$M_q(\mathcal K \oplus \mathcal H_2)$ bimodular from (\ref{eq:bimodular}). 
\end{proof}

We remark that, when $q = 1$, the above proposition still holds and is easy to check, although in this case $\xi \in \mathcal F(\mathcal H_1 \oplus \mathcal K)$ does not define a left-bounded vector, and so one has to properly interpret the vector $\xi \otimes_{M_1(\mathcal K)} \eta$. 

\begin{thm}\label{thm:everysubalgebraproperlyp}
Let $\mathcal H$ be a real Hilbert space and $-1 \leq q \leq 1$. If $p \in \mathcal P(M_q(\mathcal H))$ is a nonzero projection and $P \subset p M_q(\mathcal H) p$ is a von Neumann subalgebra with no amenable direct summand, then $P$ is properly proximal. 
\end{thm}
\begin{proof}
We use the same strategy as in \cite[Propositions 9.1 and 9.2]{DKEP22}. We set $M = M_q(\mathcal H)$ and $\tilde M = M_q(\mathcal H)$. First, note that we have a grading of Hilbert $M$-bimodules $L^2 (\tilde M) = \oplus_{k = 0}^\infty \mathcal L_k$, where $\mathcal L_k$ denotes the span of all simple tensors of the form $\xi_1 \otimes \cdots \otimes \xi_n \in \mathcal F( \mathcal H_{\mathbb C} \oplus \mathcal H_{\mathbb C} )$ such that exactly $k$ of the terms $\xi_i$ are contained in $0 \oplus \mathcal H_{\mathbb C}$ and all other terms are contained in $\mathcal H_{\mathbb C} \oplus 0$. Note also that $\mathcal L_k \cong \overline{\mathcal L_k}$ and by Proposition~\ref{prop:grading} it follows that for $j, k \geq 0$ we have embedding of Hilbert $M$-bimodules $\mathcal L_j \oovt{M} \mathcal L_k \hookrightarrow \mathcal L_{j + k}$.

Suppose $p \in \mathcal P(M)$ and $P \subset p M p$ is a von Neumann subalgebra. We may take $r \leq p$ to be the maximal subprojection in $\mathcal Z(P)$ such that $r P$ is properly proximal.  Our goal is, then, to show that $(p - r) P$ is amenable. By replacing $P$ with $(p-r)P + (p - r)^\perp \mathbb C$, we may assume that $1 \in P$. 

Let $A_n \in \K(\mathcal H)$ be an approximate unit, and take $\theta_n \in (0, \pi/2)$ with $\theta_n \to 0$. We let $V_n \in \mathcal O(\mathcal H \oplus \mathcal H)$ be given by the matrix
\[
V_n = \begin{pmatrix}
\cos \theta & - \sin \theta  \\
\sin \theta  & \cos \theta
\end{pmatrix} \begin{pmatrix}
A_n & - \sqrt{1 - A_n^2}  \\
\sqrt{1 - A_n^2}  & A_n
\end{pmatrix}
\]
and we then have that $e_M V_n^{\mathcal F} e_M$ defines a compact operator on $L^2(M) \cong \mathcal F_q(\mathcal H_{\mathbb C})$ and the corresponding automophisms $\alpha_n \in {\rm Aut}(\tilde M)$ satisfy $\alpha_n \to {\rm id}$ in the point-ultraweak topology as $n \to \infty$. 

Since $P$ has no properly proximal direct summand, and since $e_P V_n^{\mathcal F} e_P$ is compact, it then follows from the argument in \cite[Proposition 9.1]{DKEP22} that there exists a $P$-central state $\varphi$ on $(M^{\rm op})' \cap \B(\oplus_{m\geq 1} \mathcal L_m )$ such that $\varphi_{|M} = \tau$. 
By Connes's versions of Day's and Namioka's tricks (see, e.g., Section 10.3 in \cite{ADPo21}), there then exists a net of unit vectors $\xi_i \in (\oplus_{m\geq 1} \mathcal L_m) \oovt{M} \overline{(\oplus_{m\geq 1} \mathcal L_m)} \cong L^2( (M^{\rm op})' \cap \B(\oplus_{m\geq 1} \mathcal L_m)$ such that $\langle x \xi_i, \xi_i \rangle = \langle \xi_i x, \xi_i \rangle = \tau(x)$ for all $x \in M$ and $\| a \xi_i - \xi_i a \| \to 0$ for all $a \in P$. We may then take tensor powers $\zeta_{i, k} = \xi_i^{\otimes_M^k}$ and view $\zeta_{i, k}$ as vectors in $\oplus_{m\geq 2k}\mathcal L_m$ that are left and right $M$-tracial, and are asymptotically $P$-central. 

We now show that the existence of the vectors $\zeta_{i,k}$ implies that $P$ is amenable. We fix $u_1, \ldots, u_n \in \mathcal U(P)$ and $\varepsilon > 0$. There then exists a finite-dimensional subspace $\mathcal H_0 \subset \mathcal H$ so that $\| E_{M_q(\mathcal H_0)}(u_j) - u_j \|_2 < \varepsilon/2n$ for all $1 \leq j \leq n$. By the argument of \cite[Proposition 4.1]{Av11}, \cite[Proposition 4.9]{CaIsWa21}, and \cite[Proposition 6.6]{Wil20} (see \cite[Lemma 3.5]{DiTa26} and \cite[Proposition 2.3]{DrVa25} together with the subsequent discussion for explicit details), for sufficiently large $k \geq 1$, we have that $\oplus_{m\geq 2k}\mathcal L_{m}$ is weakly contained in the coarse correspondence as a Hilbert $M_q(\mathcal H_0)$-bimodule. Thus, we have
\begin{align}
\| \sum_{j = 1}^n u_j \otimes u_j^{\rm op} \|_{M \otimes M^{\rm op}}
&\geq \| \sum_{j = 1}^n E_{M_q(\mathcal H_0)}(u_j) \otimes E_{M_q(\mathcal H_0)}(u_j)^{\rm op} \|  \nonumber \\
&\geq \lim_{i \to \infty} \| \sum_{j = 1}^n E_{M_q(\mathcal H_0)}(u_j) \zeta_{i, k} E_{M_q(\mathcal H_0)}(u_j)^* \| \nonumber \\
&=  \lim_{i \to \infty} \| \sum_{j = 1}^n E_{M_q(\mathcal H_0)}(u_j)  E_{M_q(\mathcal H_0)}(u_j)^* \zeta_{i, k} \| \nonumber \\
&= \| \sum_{j = 1}^n E_{M_q(\mathcal H_0)}(u_j)  E_{M_q(\mathcal H_0)}(u_j)^* \|_2 \nonumber \\
&\geq n - \varepsilon. \nonumber
\end{align}
 Since $\varepsilon > 0$ was arbitrary, it then follows from \cite{Ha85} that $P$ is amenable. 
\end{proof}

\begin{cor}
Let $\mathcal H$ be a real Hilbert space and $-1 \leq q \leq 1$, then $M_q(\mathcal H)$ is strongly solid.
\end{cor}
\begin{proof}
By \cite[Theorem A]{Av11} the von Neumann algebra $M = M_q(\mathcal H)$ has the complete metric approximation property (see \cite{Wa20} for a simple proof), and hence, by \cite[Theorem 3.5]{OzPo10a} for any embedding of a diffuse amenable von Neumann algebra $P \subset M_q(\mathcal H)$, we have that $\mathcal N_M(P) \actson P$ is weakly compact.  Theorem~\ref{thm:everysubalgebraproperlyp}, together with \cite[Theorem 6.11]{DKEP22} then shows that $\mathcal N_M(P)''$ is amenable. 
\end{proof}

\bibliographystyle{amsalpha}
\bibliography{ref}

@article{AD87,
  author = {Anantharaman-Delaroche, Claire},
  title = {Syst\`emes dynamiques non commutatifs et moyennabilit\'{e}},
  journal = {Math. Ann.},
  volume = {279},
  year = {1987},
  number = {2},
  pages = {297--315},
}

@article{AnHa14,
  author = {Ando, Hiroshi and Haagerup, Uffe},
  title = {Ultraproducts of von {N}eumann algebras},
  journal = {J. Funct. Anal.},
  volume = {266},
  year = {2014},
  number = {12},
  pages = {6842--6913},
}

@article{AnHaHoMa20,
  author = {Ando, Hiroshi and Haagerup, Uffe and Houdayer, Cyril and Marrakchi, Amine},
  title = {Structure of bicentralizer algebras and inclusions of type {III} factors},
  journal = {Math. Ann.},
  volume = {376},
  year = {2020},
  number = {3-4},
  pages = {1145--1194},
}

@misc{ADPo21,
  author = {Anantharaman, Claire and Popa, Sorin},
  title = {An introduction to {II$_1$} factors},
  year = {2021},
  url = {https://www.math.ucla.edu/~popa/Books/IIunV15.pdf},
  note = {\url{https://www.math.ucla.edu/~popa/Books/IIunV15.pdf}},
}

@misc{Av11,
  author = {Avsec, Stephen},
  title = {Strong solidity of the {$q$}-{G}aussian algebras for all {$-1 < q < 1$}},
  year = {2011},
  howpublished = {Preprint, arXiv:1110.4918},
  eprint = {1110.4918},
  archivePrefix = {arXiv},
}

@article{BoCa15,
  author = {Boutonnet, R\'{e}mi and Carderi, Alessandro},
  title = {Maximal amenable von {N}eumann subalgebras arising from maximal amenable subgroups},
  journal = {Geom. Funct. Anal.},
  volume = {25},
  year = {2015},
  number = {6},
  pages = {1688--1705},
}

@article{BCKW22,
  author = {Borst, Matthijs and Caspers, Martijn and Klisse, Mario and Wasilewski, Mateusz},
  title = {On the isomorphism class of {$q$}-{G}aussian {$C^*$}-algebras for infinite variables},
  journal = {Proc. Amer. Math. Soc.},
  volume = {151},
  year = {2023},
  number = {2},
  pages = {737--744},
}

@incollection{BEW20,
  author = {Buss, Alcides and Echterhoff, Siegfried and Willett, Rufus},
  title = {Injectivity, crossed products, and amenable group actions},
  booktitle = {{$K$}-theory in algebra, analysis and topology},
  series = {Contemp. Math.},
  volume = {749},
  publisher = {Amer. Math. Soc.},
  address = {[Providence], RI},
  year = {[2020] \copyright 2020},
  pages = {105--137},
}

@article{BIP21,
  author = {Boutonnet, R\'{e}mi and Ioana, Adrian and Peterson, Jesse},
  title = {Properly proximal groups and their von {N}eumann algebras},
  journal = {Ann. Sci. \'{E}c. Norm. Sup\'{e}r. (4)},
  volume = {54},
  year = {2021},
  number = {2},
  pages = {445--482},
}

@article{BaMaOz19,
  author = {Bannon, Jon and Marrakchi, Amine and Ozawa, Narutaka},
  title = {Full factors and co-amenable inclusions},
  journal = {Comm. Math. Phys.},
  volume = {378},
  year = {2020},
  number = {2},
  pages = {1107--1121},
}

@book{BO08,
  author = {Brown, Nathanial~P. and Ozawa, Narutaka},
  title = {{$C^*$}-algebras and finite-dimensional approximations},
  series = {Graduate Studies in Mathematics},
  volume = {88},
  publisher = {American Mathematical Society},
  address = {Providence, RI},
  year = {2008},
}

@article{Bo12,
  author = {Boutonnet, R\'{e}mi},
  title = {On solid ergodicity for {G}aussian actions},
  journal = {J. Funct. Anal.},
  volume = {263},
  year = {2012},
  number = {4},
  pages = {1040--1063},
}

@article{Bo13,
  author = {Boutonnet, R\'{e}mi},
  title = {{${\rm W}^*$}-superrigidity of mixing {G}aussian actions of rigid groups},
  journal = {Adv. Math.},
  volume = {244},
  year = {2013},
  pages = {69--90},
}

@article{BoSp91,
  author = {Bo\.{z}ejko, Marek and Speicher, Roland},
  title = {An example of a generalized {B}rownian motion},
  journal = {Comm. Math. Phys.},
  volume = {137},
  year = {1991},
  number = {3},
  pages = {519--531},
}

@article{Ca21,
  author = {Caspers, Martijn},
  title = {Gradient forms and strong solidity of free quantum groups},
  journal = {Math. Ann.},
  volume = {379},
  year = {2021},
  number = {1-2},
  pages = {271--324},
}

@misc{Cas22,
  author = {Martijn Caspers},
  title = {On the isomorphism class of {$q$}-{G}aussian {W}{$^\ast$}-algebras for infinite variables},
  year = {2022},
  howpublished = {Preprint, arXiv:2210.11128},
  eprint = {2210.11128},
  archivePrefix = {arXiv},
  note = {To appear in {C}omptes {R}endus de l'{A}cad{\'e}mie des {S}ciences},
}

@article{CdSS16,
  author = {Chifan, Ionut and de~Santiago, Rolando and Sinclair, Thomas},
  title = {{$W^*$}-rigidity for the von {N}eumann algebras of products of hyperbolic groups},
  journal = {Geom. Funct. Anal.},
  volume = {26},
  year = {2016},
  number = {1},
  pages = {136--159},
}

@article{ChIo18,
  author = {Chifan, Ionu\c{t} and Ioana, Adrian},
  title = {Amalgamated free product rigidity for group von {N}eumann algebras},
  journal = {Adv. Math.},
  volume = {329},
  year = {2018},
  pages = {819--850},
}

@article{CaIsWa21,
  author = {Caspers, Martijn and Isono, Yusuke and Wasilewski, Mateusz},
  title = {{$L_2$}-cohomology, derivations, and quantum {M}arkov semi-groups on {$q$}-{G}aussian algebras},
  journal = {Int. Math. Res. Not. IMRN},
  year = {2021},
  number = {9},
  pages = {6405--6441},
}

@article{ChPe13,
  author = {Chifan, Ionut and Peterson, Jesse},
  title = {Some unique group-measure space decomposition results},
  journal = {Duke Math. J.},
  volume = {162},
  year = {2013},
  number = {11},
  pages = {1923--1966},
}

@article{CS13,
  author = {Chifan, Ionut and Sinclair, Thomas},
  title = {On the structural theory of {${\rm II}_1$} factors of negatively curved groups},
  journal = {Ann. Sci. \'{E}c. Norm. Sup\'{e}r. (4)},
  volume = {46},
  year = {2013},
  number = {1},
  pages = {1--33 (2013)},
}

@article{Dep20,
  author = {Deprez, Tobe},
  title = {Ozawa's class {$\mathcal S$} for locally compact groups and unique prime factorization of group von {N}eumann algebras},
  journal = {Proc. Roy. Soc. Edinburgh Sect. A},
  volume = {150},
  year = {2020},
  number = {5},
  pages = {2656--2681},
}

@misc{DKEP22,
  author = {Ding, Changying and Kunnawalkam~Elayavalli, Srivatsav and Peterson, Jesse},
  title = {Properly proximal von {N}eumann algebras},
  year = {2022},
  howpublished = {preprint, arXiv:2204.00517},
  eprint = {2204.00517},
  archivePrefix = {arXiv},
  note = {To appear in Duke Math.\ J.},
}

@article{DaLi92,
  author = {Davies, E.~Brian and Lindsay, J.~Martin},
  title = {Noncommutative symmetric {M}arkov semigroups},
  journal = {Math. Z.},
  volume = {210},
  year = {1992},
  number = {3},
  pages = {379--411},
}

@article{DyNi93,
  author = {Dykema, Ken and Nica, Alexandru},
  title = {On the {F}ock representation of the {$q$}-commutation relations},
  journal = {J. Reine Angew. Math.},
  volume = {440},
  year = {1993},
  pages = {201--212},
}

@article{Dr22,
  author = {Drimbe, Daniel},
  title = {Measure equivalence rigidity via s-malleable deformations},
  journal = {Compos. Math.},
  volume = {159},
  year = {2023},
  number = {10},
  pages = {2023--2050},
}

@article{DiTa26,
  author = {Ding, Changying and Tan, Hui},
  title = {Structural and non-isomorphism results for {$q$}-{A}raki-{W}oods factors},
  journal = {J. Funct. Anal.},
  volume = {291},
  year = {2026},
  number = {5},
  pages = {Art. 111541},
}

@misc{DrVa25,
  author = {Drimbe, Daniel and Vaes, Stefaan},
  title = {{$W^*$}-correlations of {II$_1$} factors and rigidity of tensor products and graph products},
  year = {2025},
  howpublished = {Preprint, arXiv:2507.04691},
  eprint = {2507.04691},
  archivePrefix = {arXiv},
}

@article{EfOzRu01,
  author = {Effros, Edward~G. and Ozawa, Narutaka and Ruan, Zhong-Jin},
  title = {On injectivity and nuclearity for operator spaces},
  journal = {Duke Math. J.},
  volume = {110},
  year = {2001},
  number = {3},
  pages = {489--521},
}

@article{GaJu07,
  author = {Gao, Mingchu and Junge, Marius},
  title = {Examples of prime von {N}eumann algebras},
  journal = {Int. Math. Res. Not. IMRN},
  year = {2007},
  number = {15},
  pages = {Art. ID rnm042, 34},
}

@article{GuSh14,
  author = {Guionnet, A. and Shlyakhtenko, D.},
  title = {Free monotone transport},
  journal = {Invent. Math.},
  volume = {197},
  year = {2014},
  number = {3},
  pages = {613--661},
}

@article{Ha75,
  author = {Haagerup, Uffe},
  title = {The standard form of von {N}eumann algebras},
  journal = {Math. Scand.},
  volume = {37},
  year = {1975},
  number = {2},
  pages = {271--283},
}

@incollection{Ha85,
  author = {Haagerup, Uffe},
  title = {Injectivity and decomposition of completely bounded maps},
  booktitle = {Operator algebras and their connections with topology and ergodic theory ({B}u\c{s}teni, 1983)},
  series = {Lecture Notes in Math.},
  volume = {1132},
  publisher = {Springer},
  address = {Berlin},
  year = {1985},
  pages = {170--222},
}

@article{HoIs16,
  author = {Houdayer, Cyril and Isono, Yusuke},
  title = {Bi-exact groups, strongly ergodic actions and group measure space type {III} factors with no central sequence},
  journal = {Comm. Math. Phys.},
  volume = {348},
  year = {2016},
  number = {3},
  pages = {991--1015},
}

@article{HoIs17,
  author = {Houdayer, Cyril and Isono, Yusuke},
  title = {Unique prime factorization and bicentralizer problem for a class of type {III} factors},
  journal = {Adv. Math.},
  volume = {305},
  year = {2017},
  pages = {402--455},
}

@article{HaIsKa22,
  author = {Hasegawa, Kei and Isono, Yusuke and Kanda, Tomohiro},
  title = {Note on bi-exactness for creation operators on {F}ock spaces},
  journal = {J. Math. Soc. Japan},
  volume = {74},
  year = {2022},
  number = {3},
  pages = {903--944},
}

@misc{Ho07,
  author = {Houdayer, Cyril},
  title = {Sur la classification de certaines alg{\`{e}}bres de von {N}eumann},
  year = {2007},
  note = {Ph{D} thesis, {U}niversit{\'{e}} de {P}aris {VII}},
}

@article{HaPa11,
  author = {Han, Kyung~Hoon and Paulsen, Vern~I.},
  title = {An approximation theorem for nuclear operator systems},
  journal = {J. Funct. Anal.},
  volume = {261},
  year = {2011},
  number = {4},
  pages = {999--1009},
}

@misc{HiPePo23,
  author = {Hiatt, Patrick and Peterson, Jesse and Popa, Sorin},
  title = {Some classes of smooth bimodules over {II$_1$} factors and their associated {1}-cohomology spaces},
  year = {2023},
  howpublished = {preprint, arXiv:2304.06242},
  eprint = {2304.06242},
  archivePrefix = {arXiv},
}

@article{HoRa15,
  author = {Houdayer, Cyril and Raum, Sven},
  title = {Asymptotic structure of free {A}raki-{W}oods factors},
  journal = {Math. Ann.},
  volume = {363},
  year = {2015},
  number = {1-2},
  pages = {237--267},
}

@article{HoUe16a,
  author = {Houdayer, Cyril and Ueda, Yoshimichi},
  title = {Asymptotic structure of free product von {N}eumann algebras},
  journal = {Math. Proc. Cambridge Philos. Soc.},
  volume = {161},
  year = {2016},
  number = {3},
  pages = {489--516},
}

@article{HoUe16,
  author = {Houdayer, Cyril and Ueda, Yoshimichi},
  title = {Rigidity of free product von {N}eumann algebras},
  journal = {Compos. Math.},
  volume = {152},
  year = {2016},
  number = {12},
  pages = {2461--2492},
}

@article{Io07,
  author = {Ioana, Adrian},
  title = {Rigidity results for wreath product {${\rm II}_1$} factors},
  journal = {J. Funct. Anal.},
  volume = {252},
  year = {2007},
  number = {2},
  pages = {763--791},
}

@article{Ioa12,
  author = {Ioana, Adrian},
  title = {Uniqueness of the group measure space decomposition for {P}opa's {$\mathcal{HT}$} factors},
  journal = {Geom. Funct. Anal.},
  volume = {22},
  year = {2012},
  number = {3},
  pages = {699--732},
}

@article{Iso13,
  author = {Isono, Yusuke},
  title = {Weak exactness for {$C^\ast$}-algebras and application to condition ({AO})},
  journal = {J. Funct. Anal.},
  volume = {264},
  year = {2013},
  number = {4},
  pages = {964--998},
}

@article{Iso15,
  author = {Isono, Yusuke},
  title = {Examples of factors which have no {C}artan subalgebras},
  journal = {Trans. Amer. Math. Soc.},
  volume = {367},
  year = {2015},
  number = {11},
  pages = {7917--7937},
}

@article{Is15,
  author = {Isono, Yusuke},
  title = {On bi-exactness of discrete quantum groups},
  journal = {Int. Math. Res. Not. IMRN},
  year = {2015},
  number = {11},
  pages = {3619--3650},
}

@article{Is17,
  author = {Isono, Yusuke},
  title = {Some prime factorization results for free quantum group factors},
  journal = {J. Reine Angew. Math.},
  volume = {722},
  year = {2017},
  pages = {215--250},
}

@article{Is19,
  author = {Isono, Yusuke},
  title = {Cartan subalgebras of tensor products of free quantum group factors with arbitrary factors},
  journal = {Anal. PDE},
  volume = {12},
  year = {2019},
  number = {5},
  pages = {1295--1324},
}

@article{Iso22,
  author = {Yusuke Isono},
  title = {Unitary conjugacy for type {III} subfactors and {$W^*$}-superrigidity},
  journal = {J. Eur. Math. Soc.},
  volume = {24},
  year = {2022},
  number = {5},
  pages = {1679--1721},
}

@misc{Iso25,
  author = {Isono, Yusuke},
  title = {Weak relative {D}ixmier property and {P}opa's intertwining technique for type {III} subfactors},
  year = {2025},
  howpublished = {Preprint, arXiv:2508.17592},
  eprint = {2508.17592},
  archivePrefix = {arXiv},
}

@incollection{Kir95,
  author = {Kirchberg, Eberhard},
  title = {Exact {${\rm C}^*$}-algebras, tensor products, and the classification of purely infinite algebras},
  booktitle = {Proceedings of the {I}nternational {C}ongress of {M}athematicians, {V}ol. 1, 2 ({Z}\"{u}rich, 1994)},
  publisher = {Birkh\"{a}user},
  address = {Basel},
  year = {1995},
  pages = {943--954},
}

@article{Kir95B,
  author = {Kirchberg, Eberhard},
  title = {On subalgebras of the {CAR}-algebra},
  journal = {J. Funct. Anal.},
  volume = {129},
  year = {1995},
  number = {1},
  pages = {35--63},
}

@article{Ku22,
  author = {Kuzmin, Alexey},
  title = {C{CR} and {CAR} algebras are connected via a path of {C}untz-{T}oeplitz algebras},
  journal = {Comm. Math. Phys.},
  volume = {399},
  year = {2023},
  number = {3},
  pages = {1623--1645},
}

@article{Ma98,
  author = {Magajna, Bojan},
  title = {A topology for operator modules over {$W^*$}-algebras},
  journal = {J. Funct. Anal.},
  volume = {154},
  year = {1998},
  number = {1},
  pages = {17--41},
}

@article{Ma00,
  author = {Magajna, B.},
  title = {{$C^*$}-convex sets and completely bounded bimodule homomorphisms},
  journal = {Proc. Roy. Soc. Edinburgh Sect. A},
  volume = {130},
  year = {2000},
  number = {2},
  pages = {375--387},
}

@article{Ma17,
  author = {Marrakchi, Amine},
  title = {Solidity of type {III} {B}ernoulli crossed products},
  journal = {Comm. Math. Phys.},
  volume = {350},
  year = {2017},
  number = {3},
  pages = {897--916},
}

@misc{Mar23,
  author = {Amine Marrakchi},
  title = {Kadison's problem for type {III} subfactors and the bicentralizer conjecture},
  year = {2023},
  howpublished = {preprint, arXiv:2308.15163},
  eprint = {2308.15163},
  archivePrefix = {arXiv},
}

@book{Mo15,
  author = {Morris, Dave~Witte},
  title = {Introduction to arithmetic groups},
  publisher = {Deductive Press},
  address = {[place of publication not identified]},
  year = {2015},
  mrnumber = {3307755},
}

@article{Nou04,
  author = {Nou, Alexandre},
  title = {Non injectivity of the {$q$}-deformed von {N}eumann algebra},
  journal = {Math. Ann.},
  volume = {330},
  year = {2004},
  number = {1},
  pages = {17--38},
}

@book{Oc85,
  author = {Ocneanu, Adrian},
  title = {Actions of discrete amenable groups on von {N}eumann algebras},
  series = {Lecture Notes in Mathematics},
  volume = {1138},
  publisher = {Springer-Verlag},
  address = {Berlin},
  year = {1985},
}

@article{OzPo04,
  author = {Ozawa, Narutaka and Popa, Sorin},
  title = {Some prime factorization results for type {${\rm II}_1$} factors},
  journal = {Invent. Math.},
  volume = {156},
  year = {2004},
  number = {2},
  pages = {223--234},
}

@article{OzPo10,
  author = {Narutaka Ozawa and Sorin Popa},
  title = {On a class of {II}{$_1$} factors with at most one cartan subalgebra, {II}},
  journal = {Amer. J. Math.},
  volume = {132},
  year = {2010},
  number = {3},
  pages = {841--866},
}

@article{OzPo10a,
  author = {Ozawa, Narutaka and Popa, Sorin},
  title = {On a class of {${\rm II}_1$} factors with at most one {C}artan subalgebra},
  journal = {Ann. of Math. (2)},
  volume = {172},
  year = {2010},
  number = {1},
  pages = {713--749},
}

@misc{Oza01,
  author = {Ozawa, Narutaka},
  title = {Nuclearity of reduced amalgamated free product {$C^*$}-algebras},
  year = {2002},
  howpublished = {no. 1250},
  note = {Theory of operator algebras and its applications (Japanese) (Kyoto, 2001)},
  pages = {49--55},
}

@article{Oza04,
  author = {Ozawa, Narutaka},
  title = {Solid von {N}eumann algebras},
  journal = {Acta Math.},
  volume = {192},
  year = {2004},
  number = {1},
  pages = {111--117},
}

@incollection{Oza06b,
  author = {Ozawa, Narutaka},
  title = {Amenable actions and applications},
  booktitle = {International {C}ongress of {M}athematicians. {V}ol. {II}},
  publisher = {Eur. Math. Soc.},
  address = {Z\"{u}rich},
  year = {2006},
  pages = {1563--1580},
}

@article{Oza06,
  author = {Ozawa, Narutaka},
  title = {A {K}urosh-type theorem for type {$\rm II_1$} factors},
  journal = {Int. Math. Res. Not.},
  year = {2006},
  pages = {Art. ID 97560, 21},
}

@article{Oza07,
  author = {Ozawa, Narutaka},
  title = {Weakly exact von {N}eumann algebras},
  journal = {J. Math. Soc. Japan},
  volume = {59},
  year = {2007},
  number = {4},
  pages = {985--991},
}

@article{Oza09,
  author = {Ozawa, Narutaka},
  title = {An example of a solid von {N}eumann algebra},
  journal = {Hokkaido Math. J.},
  volume = {38},
  year = {2009},
  number = {3},
  pages = {557--561},
}

@incollection{Oza10,
  author = {Ozawa, Narutaka},
  title = {A comment on free group factors},
  booktitle = {Noncommutative harmonic analysis with applications to probability {II}},
  series = {Banach Center Publ.},
  volume = {89},
  publisher = {Polish Acad. Sci. Inst. Math.},
  address = {Warsaw},
  year = {2010},
  pages = {241--245},
}

@article{Pet09,
  author = {Peterson, Jesse},
  title = {{$L^2$}-rigidity in von {N}eumann algebras},
  journal = {Invent. Math.},
  volume = {175},
  year = {2009},
  number = {2},
  pages = {417--433},
}

@misc{Po86,
  author = {Popa, Sorin},
  title = {Correspondences},
  year = {1986},
  url = {https://www.math.ucla.edu/~popa/popa-correspondences.pdf},
  note = {{INCREST} {P}reprint, 56/1986, \url{https://www.math.ucla.edu/~popa/popa-correspondences.pdf}},
}

@article{Po06B,
  author = {Sorin Popa},
  title = {Strong rigidity of {$\rm II_1$} factors arising from malleable actions of {$w$}-rigid groups. {I}},
  journal = {Invent. Math.},
  volume = {165},
  year = {2006},
  number = {2},
  pages = {369--408},
}

@article{po07b,
  author = {Sorin Popa},
  title = {On {O}zawa's property for free group factors},
  journal = {Int. Math. Res. Not. IMRN},
  year = {2007},
  number = {11},
  pages = {Art. ID rnm036, 10},
}

@article{PoVa14b,
  author = {Popa, Sorin and Vaes, Stefaan},
  title = {Unique {C}artan decomposition for {$\rm II_{1}$} factors arising from arbitrary actions of hyperbolic groups},
  journal = {J. Reine Angew. Math.},
  volume = {694},
  year = {2014},
  pages = {215--239},
}

@article{PV14,
  author = {Popa, Sorin and Vaes, Stefaan},
  title = {Unique {C}artan decomposition for {$\rm II_{1}$} factors arising from arbitrary actions of hyperbolic groups},
  journal = {J. Reine Angew. Math.},
  volume = {694},
  year = {2014},
  pages = {215--239},
}

@article{Sa09,
  author = {Sako, Hiroki},
  title = {The class {$S$} as an {ME} invariant},
  journal = {Int. Math. Res. Not. IMRN},
  year = {2009},
  number = {15},
  pages = {2749--2759},
}

@incollection{Sa90,
  author = {Sauvageot, Jean-Luc},
  title = {Quantum {D}irichlet forms, differential calculus and semigroups},
  booktitle = {Quantum probability and applications, {V} ({H}eidelberg, 1988)},
  series = {Lecture Notes in Math.},
  volume = {1442},
  publisher = {Springer},
  address = {Berlin},
  year = {1990},
  pages = {334--346},
}

@article{Sa99,
  author = {Sauvageot, Jean-Luc},
  title = {Strong {F}eller semigroups on {$C^\ast$}-algebras},
  journal = {J. Operator Theory},
  volume = {42},
  year = {1999},
  number = {1},
  pages = {83--102},
}

@article{Shl04,
  author = {Shlyakhtenko, Dimitri},
  title = {Some estimates for non-microstates free entropy dimension with applications to {$q$}-semicircular families},
  journal = {Int. Math. Res. Not.},
  year = {2004},
  number = {51},
  pages = {2757--2772},
}

@book{SinSmi95,
  author = {Sinclair, Allan~M. and Smith, Roger~R.},
  title = {Hochschild cohomology of von {N}eumann algebras},
  series = {London Mathematical Society Lecture Note Series},
  volume = {203},
  publisher = {Cambridge University Press},
  address = {Cambridge},
  year = {1995},
}

@book{VoDyNi92,
  author = {Voiculescu, D.~V. and Dykema, K.~J. and Nica, A.},
  title = {Free random variables},
  series = {CRM Monograph Series},
  volume = {1},
  publisher = {American Mathematical Society},
  address = {Providence, RI},
  year = {1992},
  note = {A noncommutative probability approach to free products with applications to random matrices, operator algebras and harmonic analysis on free groups},
}

@article{VaVe06,
  author = {Vaes, Stefaan and Vergnioux, Roland},
  title = {The boundary of universal discrete quantum groups, exactness, and factoriality},
  journal = {Duke Math. J.},
  volume = {140},
  year = {2007},
  number = {1},
  pages = {35--84},
}

@article{Wa20,
  author = {Wasilewski, Mateusz},
  title = {A simple proof of the complete metric approximation property for {$q$}-{G}aussian algebras},
  journal = {Colloq. Math.},
  volume = {163},
  year = {2021},
  number = {1},
  pages = {1--14},
}

@mastersthesis{Wil20,
  author = {Wildschut, Guillermo},
  title = {Strong solidity of {$q$}-{G}aussian algebras},
  school = {Delft University of Technology},
  year = {2020},
}

\end{document}